\documentclass[11pt]{article}

\usepackage[T1]{fontenc}
\usepackage[utf8]{inputenc}
\usepackage{lmodern}
\usepackage[margin=1.05in]{geometry}
\usepackage{microtype}
\usepackage{amsmath,amssymb,amsthm,mathtools}
\usepackage{enumitem}
\usepackage{xspace}
\usepackage{tikz}
\usetikzlibrary{arrows.meta,calc,positioning,fit}
\usepackage{fancyhdr}
\usepackage[colorlinks=true,linkcolor=blue!55!black,
  citecolor=blue!55!black,urlcolor=blue!55!black]{hyperref}
\usepackage[nameinlink,capitalise,noabbrev]{cleveref}
\crefformat{chatgptappendix}{#2Appendix#3}
\Crefformat{chatgptappendix}{#2Appendix#3}

\newcommand{\provenanceheader}{%
  \ifnum\value{page}<11 Human-written part%
  \else\ifnum\value{page}<26 Human-verified part%
  \else\ifnum\value{page}<29 References%
  \else Appendix by AI%
  \fi\fi\fi}

\newcommand{\configurepaperheader}{%
  \fancyhf{}%
  \fancyhead[L]{\scriptsize\sffamily\color{black}%
    A nearcircumsphere-Ramsey theorem}%
  \fancyhead[R]{\scriptsize\sffamily\bfseries\color{black}%
    \provenanceheader}%
  \fancyfoot[C]{\thepage}%
  \renewcommand{\headrulewidth}{0.3pt}%
  \renewcommand{\footrulewidth}{0pt}}
\fancypagestyle{paper}{\configurepaperheader}
\fancypagestyle{plain}{%
  \fancyhf{}%
  \fancyfoot[C]{\thepage}%
  \renewcommand{\headrulewidth}{0pt}%
  \renewcommand{\footrulewidth}{0pt}}
\newtheorem{theorem}{Theorem}
\newtheorem{proposition}[theorem]{Proposition}
\newtheorem{lemma}[theorem]{Lemma}
\newtheorem{corollary}[theorem]{Corollary}
\newtheorem*{claim}{Claim}
\theoremstyle{definition}
\newtheorem{definition}[theorem]{Definition}

\theoremstyle{remark}
\newtheorem{remark}[theorem]{Remark}

\def\dim{\mathrm{dim}}
\def\eps{\varepsilon}

\def\N{\mbox{\ensuremath{\mathbb N}}\xspace}
\def\R{\mbox{\ensuremath{\mathbb R}}\xspace}
\def\S{\mbox{\ensuremath{\mathbb S}}\xspace}

\def\Z{\mbox{\ensuremath{\mathbb Z}}\xspace}

\def\D{\mbox{\ensuremath{\mathcal D}}\xspace}

\def\P{\mbox{\ensuremath{\mathcal P}}\xspace}

\mathcode`l="8000
\begingroup
\makeatletter
\lccode`\~=`\l
\DeclareMathSymbol{\lsb@l}{\mathalpha}{letters}{`l}
\lowercase{\gdef~{\ifnum\the\mathgroup=\m@ne \ell \else \lsb@l \fi}}%
\endgroup

\newcommand{\id}{\mathrm{id}}
\newcommand{\Sh}{\mathrm{Sh}}
\newcommand{\KG}{\mathrm{KG}}
\newcommand{\KSh}{\mathrm{KSh}}

\title{A nearcircumsphere-Ramsey Theorem for\\
Solvable Transitive Configurations}
\author{D\"om\"ot\"or P\'alv\"olgyi\thanks{ELTE E\"otv\"os Lor\'and University and Alfr\'ed R\'enyi Institute of Mathematics, Budapest, Hungary. Supported by the NRDI EXCELLENCE--24 grant no.~151504, Combinatorics and Geometry, and by the ERC Advanced Grant no.~101054936, ERMiD. E-mail: \href{mailto:domotor.palvolgyi@ttk.elte.hu}{\texttt{domotor.palvolgyi@ttk.elte.hu}}.}}
\date{}
\hypersetup{
  pdftitle={A nearcircumsphere-Ramsey Theorem for Solvable Transitive Configurations},
  pdfauthor={D\"om\"ot\"or P\'alv\"olgyi},
  pdfsubject={Euclidean Ramsey theory},
  pdfkeywords={Euclidean Ramsey theory, transitive configuration, solvable group,
  nearcircumsphere-Ramsey, Kneser-shift graph, Tucker lemma}}

\begin{document}
\maketitle

\begin{abstract}
Let $P\subset\R^d$ be a finite spherical set.
We prove that if $P$ admits a solvable group of isometries acting transitively on it, then every $r$-coloring of a sufficiently high-dimensional sphere of radius slightly larger than the circumradius of $P$ contains a monochromatic congruent copy of $P$.
Our proof builds on the group-theoretic argument of K\v r\'i\v z and combines it with a topological method that may be of independent interest.
\end{abstract}



\section{Introduction}\label{sec:introduction}

Euclidean Ramsey theory was started by Erd\H{o}s, Graham, Montgomery, Rothschild, Spencer, and Straus in a series of three papers just over half a century ago \cite{ErdosEtAl1973,ErdosEtAl1975II,ErdosEtAl1975III}.
Its central question is to determine for which point sets it is true that every finite coloring of a sufficiently high-dimensional Euclidean space contains a monochromatic (congruent) copy of the set.
More precisely, a finite set of points $P\subset\R^d$ is called \emph{Ramsey} if for every $r\in\N$ there is a dimension $n$ such that every $r$-coloring of $\R^n$ contains a monochromatic set congruent to $P$, and $P$ is called \emph{spherical} if there is a sphere that contains $P$.
They proved that every Ramsey set is spherical and conjectured that this necessary condition is also sufficient.
Using a product theorem, they showed that several classes of spherical sets, including all bricks, are Ramsey.
Later, Frankl and R\"odl~\cite{FranklRodl1987,FranklRodl1990} established that every simplex is also Ramsey.
For more results, see Graham's survey~\cite{Graham2017}; here we will only focus on results that are directly related to the topic of our paper.

Call $P\subset \R^d$ \emph{transitive} if its isometry group acts on it transitively, i.e., for any $p,q\in P$ there is an isometry of $\R^d$ that takes $P$ to $P$ and $p$ to $q$.
Moreover, $P$ is \emph{solvable transitive} if there is a solvable group of isometries of $\R^d$ that preserves $P$ and acts transitively on it.
For example, all regular polygons are solvable transitive, and so are all regular simplices, as some cyclic group acts on them transitively, while the vertex set of the regular dodecahedron is transitive but not solvable transitive.
Every transitive set is spherical, but the converse is not true in general \cite{LeaderRussellWalters2011,LeaderRussellWalters2012}.
K\v r\'i\v z~\cite{Kriz1991} showed that solvable transitive sets are even Ramsey; in fact, he proved the stronger statement that it is enough for a set to admit a solvable isometry group with at most two orbits, which is the case for the dodecahedron.
Therefore, he established that the vertex sets of all regular $n$-gons and Platonic solids are Ramsey; later Cantwell~\cite{Cantwell2007} proved the same for every regular polytope.
A set $P\subset \R^d$ is (solvable) \emph{subtransitive} if $P\subset P'$ for some (solvable) transitive set $P'$ in a possibly higher-dimensional space.
For example, any acute triangle is solvable subtransitive because it embeds in a three-dimensional brick.
As every subset of a Ramsey set is Ramsey, all solvable subtransitive sets are Ramsey, and Leader, Russell, and Walters \cite{LeaderRussellWalters2012} conjectured that Ramsey sets are precisely the subtransitive sets.
Behague~\cite{Behague2025} recently showed that, with only two possible exceptions, all previously known Ramsey configurations are solvable subtransitive in this sense.

For a spherical set $P\subset \R^d$, there is a unique point $o\in\operatorname{aff}P$ and a unique $\rho\geq0$ such that $\|p-o\|=\rho$ for every $p\in P$.
We call $o$ the \emph{circumcenter}, $\rho$ the \emph{circumradius}, and $\S_\rho^{d-1}(o)=\{x\in\R^d:\|x-o\|=\rho\}$ the \emph{circumsphere} of $P$.
Without loss of generality, we will assume that $o$ is the origin and use the notation $\S_R^n=\{x\in\R^{n+1}:\|x\|=R\}$.
Now we will define various versions of \emph{sphere-Ramseyness}.

\begin{definition}\label{def:sphereR}
	Let $P$ be a spherical set with circumradius $\rho$.
	
	$P$ is \emph{sphere-Ramsey} if for every $r\in\N$ there are an $R>0$ and an $n\in\N$ such that every $r$-coloring of $\S_R^n$ contains a monochromatic copy of $P$.
	
	$P$ is \emph{circumsphere-Ramsey} if for every $r\in\N$ there is an $n\in\N$ such that every $r$-coloring of $\S_\rho^n$ contains a monochromatic copy of $P$.
	
	$P$ is \emph{nearcircumsphere-Ramsey}, abbreviated \emph{ncs-Ramsey}, if for every $r\in\N$ and every $\eps>0$ there is an $n\in\N$ such that every $r$-coloring of $\S_{\rho+\eps}^n$ contains a monochromatic copy of $P$.
\end{definition}

The first version of these definitions appeared in Graham~\cite{Graham1983}, who used the term sphere-Ramsey for what we call circumsphere-Ramsey and pioneered its study \cite{Graham1985,Graham1990}.
The terminology later shifted, and sphere-Ramsey came to mean the weaker property defined above, leading to some confusion; the distinction is blurred even in Graham's last survey~\cite{Graham2017}.
By embedding a sphere as a subsphere of a larger sphere of one higher dimension, we obtain the chain of implications
\[
\text{circumsphere-Ramsey}
\ \Longrightarrow\
\text{ncs-Ramsey}
\ \Longrightarrow\
\text{sphere-Ramsey}
\ \Longrightarrow\
\text{Ramsey}
\]

Some of these implications may also hold in the reverse direction.
In fact, it is possible that the last three notions are all equivalent; according to Reiher \cite{Reiher2021}, Graham conjectured that all spherical sets are ncs-Ramsey.
We have not found this conjecture stated in Graham's papers, except for the special case of bricks, which was later settled by Frankl and R\"odl \cite{FranklRodl1990}.
However, circumsphere-Ramseyness is substantially different: a pair of antipodal points is not circumsphere-Ramsey.
A necessary condition for circumsphere-Ramseyness, based on Rado's theorem about partition regularity over the nonzero reals \cite{Rado1945}, was proved by Graham~\cite{Graham1983}.

Graham~\cite{Graham1983} and Lov\'asz~\cite{Lovasz1983} independently proved that every two-point set is ncs-Ramsey.
Equivalently, every pair of points at distance less than $2$ is Ramsey on a sufficiently high-dimensional unit sphere.
Raigorodskii~\cite{Raigorodskii2012}, apparently aware only of Lov\'asz's topological proof, later rediscovered Graham's linear algebra argument based on the Frankl--Wilson theorem~\cite{FranklWilson1981}.

Frankl and R\"odl~\cite{FranklRodl1990} also used linear algebra to prove that every simplex is sphere-Ramsey.
Matou\v sek and R\"odl~\cite{MatousekRodl1995}, using a Banach-space argument based on Krivine's theorem, strengthened this by proving that every simplex is ncs-Ramsey.
Several of these papers establish still stronger conclusions, but we keep the presentation minimal and do not introduce notions such as super-Ramsey, hyper-Ramsey, or $P$-Ramsey.
One more, geometric strengthening is the notion of a diameter-Ramsey configuration, for which the finite Ramsey witness is required to have the same diameter as the target; see Frankl, Pach, Reiher, and R\"odl~\cite{FranklPachReiherRodl2018}.

The proof of K\v r\'i\v z~\cite{Kriz1991} also implies that all solvable transitive sets are sphere-Ramsey.
The main result of our paper is to improve this to ncs-Ramsey, making an important step towards what Reiher called ``Graham's radius conjecture''.

\begin{theorem}\label{thm:main}
If $P$ is solvable transitive, then $P$ is nearcircumsphere-Ramsey.
More explicitly, if $P$ is a finite spherical set with circumradius $\rho$ and admits a solvable group of isometries that acts transitively on $P$, then for every $r\in\N$ and every $\eps>0$ there is a dimension $n$ such that every $r$-coloring of $\S_{\rho+\eps}^n$ contains a monochromatic congruent copy of $P$.
\end{theorem}

Note that \cref{thm:main} does not automatically imply the same conclusion for solvable subtransitive sets $P$, because this would require a solvable transitive set containing $P$ with circumradius only slightly larger than $\rho$.
This was implicitly established for simplices by Matou\v sek and R\"odl~\cite{MatousekRodl1995}, and for triangles an explicit, short proof can be found in Leader, Russell, and Walters \cite{LeaderRussellWalters2012}, but we leave it as an open problem whether such a (solvable) transitive set exists for every (solvable) subtransitive set.

Apart from the trivial one-point case, the conclusion of \cref{thm:main} cannot be strengthened by replacing $\rho+\eps$ with $\rho$: no transitive spherical configuration of circumradius $\rho$ is circumsphere-Ramsey.
Indeed, color every sphere by the sign of the first nonzero coordinate.
The sum of finitely many vectors of either color has the same lexicographic sign, so neither color class contains a finite set with barycenter zero.
On the other hand, the barycenter of a transitive configuration is its circumcenter.
If a copy $Q$ lies on a sphere of radius $\rho$ centered at the origin and has barycenter $b$, then
\[
\rho^2
=\frac{1}{|Q|}\sum_{q\in Q}\|q-b\|^2
=\frac{1}{|Q|}\sum_{q\in Q}\|q\|^2-\|b\|^2
=\rho^2-\|b\|^2,
\]
so $b=0$, a contradiction.
Thus \cref{thm:main} is sharp in the radius for solvable transitive configurations: every strictly larger prescribed radius works, while the circumradius itself does not.

\cref{thm:main} also implies the following special case.

\begin{corollary}\label{cor:polygons}
	Every regular polygon is ncs-Ramsey.
\end{corollary}

Our proof combines topological methods with the original argument of K\v r\'i\v z, though most of our notation is closer to that used in later works \cite{KanellopoulosKaramanlis2020,LeaderRussellWalters2012}.
We use the topological machinery to establish a combinatorial result that can be of independent interest, giving a common generalization to shift graphs and Kneser graphs.
Now we turn our attention to this topic.

\bigskip

The vertices of a \emph{shift graph} $\Sh_p(n)$, defined by Erd\H{o}s and Hajnal \cite{ErdosHajnal1964}, are the increasing $(p-1)$-tuples of $[n]$, and its edges are $(v_1,\ldots,v_{p-1})(v_2,\ldots,v_p)$ for every $v_1<\cdots<v_p$.
They showed that shift graphs have unbounded chromatic number: $\chi(\Sh_p(n))\to\infty$ as $n\to\infty$ for every fixed $p$, roughly as a $(p-2)$-times iterated logarithm of $n$.

The vertices of a \emph{Kneser graph} $\KG(n,k)$ are the $k$-subsets of $[n]$, denoted by $\binom{[n]}k$, and two vertices $A,B$ are adjacent when $A\cap B=\emptyset$.
Lov\'asz~\cite{Lovasz1978} famously proved, using topological methods, Kneser's conjecture, that $\chi(\KG(n,k))=n-2k+2$ for all $n\ge 2k$.
In particular, $\chi(\KG(n,k))\to\infty$ as $n-2k\to\infty$.
\footnote{Interestingly, a slightly weaker version of this statement was proved earlier by Szemer\'edi using a theorem of Kleitman~\cite{Kleitman1966}; he showed that $\chi\bigl(\KG(\lceil(2+\eps)k\rceil,k)\bigr)\to\infty$ as $k\to\infty$ for any $\eps>0$ \cite[Lemma~4]{ErdosSimonovits1973}.
A similar bound in \cref{thm:kneser-shift} would also suffice for us, but I could not find a simpler proof for this weaker statement.}

The vertices of a $p$-uniform \emph{Kneser hypergraph} $\KG^{(p)}(n,k)$ are also $\binom{[n]}k$, and its hyperedges are the sets $\{A_1,\ldots,A_p\}$ of pairwise disjoint vertices.
Alon, Frankl, and Lov\'asz~\cite{AlonFranklLovasz1986} generalized Lov\'asz's theorem by proving $\chi(\KG^{(p)}(n,k))=\left\lceil\frac{n-p(k-1)}{p-1}\right\rceil$ for every $n\geq pk$.
Instead of the $\Z_2$-equivariant Borsuk--Ulam theorem used originally by Lov\'asz, they used the $\Z_p$-equivariant B\'ar\'any--Shlosman--Sz\H{u}cs theorem \cite{BaranyShlosmanSzucs1981} for prime $p$, and a simple product argument for composite $p$.

Now we introduce a new family, \emph{Kneser-shift graphs}.
The vertices of the Kneser-shift graph $\KSh_p(n,k)$ are the ordered $(p-1)$-tuples $(A_1,\ldots,A_{p-1})$ of pairwise disjoint sets from $\binom{[n]}k$, and there is an edge between $(A_1,\ldots,A_{p-1})$ and $(A_2,\ldots,A_p)$ whenever $A_1$ and $A_p$ are also pairwise disjoint.
Our second main result is the following.

\begin{theorem}[Kneser-shift theorem]\label{thm:kneser-shift}
$\chi(\KSh_p(n,k))\to\infty$ as $n-pk\to\infty$ for every prime $p$.
In other words, for every prime $p$ and every $r\in\N$ there is a (least) constant $C(p,r)$, depending only on $p$ and $r$, such that whenever $n-pk$ is at least $C(p,r)$, every $r$-coloring of the vertices of $\KSh_p(n,k)$ admits pairwise disjoint $A_1,\ldots,A_p\in\binom{[n]}k$ for which $c(A_1,\ldots,A_{p-1})=c(A_2,\ldots,A_p)$.
\end{theorem}

Quantitative upper and lower bounds for $C(p,r)$ with the same tower height are given in \cref{prop:kneser-shift-bounds} after the proof of the theorem.
I do not know whether the primality of $p$ is necessary in this theorem.

For $k=1$, the subgraph induced by the increasing tuples is precisely the shift graph $\Sh_p(n)$; the additional vertices record all other orderings.
In fact, for every larger $k$, our graph $\KSh_p(n,k)$ also contains as a subgraph a shift graph $\Sh_p(\lfloor n/k\rfloor)$, by partitioning an initial segment of $[n]$ into consecutive atoms of size $k$.
This already implies that $\chi(\KSh_p(n,k))\to\infty$ as $n\to\infty$ for any fixed $p$ and $k$, but we need the statement when $n-pk\to\infty$.
Indeed, if we also imposed the restriction $A_1<\cdots<A_p$ on the edges, meaning that each element of $A_i$ needs to be smaller than each element of $A_j$ for all $i<j$, we would only get iterated-logarithmic growth; more precisely,
\[
\chi\!\left(\Sh_p(\lfloor n/k\rfloor)\right)
\leq \chi(\text{restricted graph of } \KSh_p(n,k))
\leq \chi(\Sh_p(n))
\leq \chi(\KSh_p(n,k)).
\]

For $p=2$, we get back exactly the Kneser graph: $\KSh_2(n,k)=\KG(n,k)$.
In the first new case, $p=3$, the vertices of $\KSh_3(n,k)$ are oriented edges of the ordinary Kneser graph, and we have an edge between $(A,B)$ and $(B,C)$ when $A,B$, and $C$ are pairwise disjoint, i.e., when these three vertices form a hyperedge in the $3$-uniform Kneser hypergraph $\KG^{(3)}(n,k)$, or equivalently, a triangle in the Kneser graph $\KG(n,k)$.
A somewhat related result is due to Poljak and R\"odl \cite{PoljakRodl1981}, who studied the so-called arc-chromatic number of symmetric digraphs, defined as the minimum number of colors needed to color the directed edges so that no path of length two is monochromatic, and proved that it is exactly $\min\left\{r:\chi(G)\leq\binom{r}{\lfloor r/2\rfloor}\right\}$.
This theorem does not imply \cref{thm:kneser-shift} for $p=3$, since adjacency in $\KSh_3(n,k)$ imposes the additional condition $A\cap C=\emptyset$.
Further iterations of the arc-graph construction, which are related to the cases $p>3$ of our theorem, were also studied \cite{Poljak1991,RorabaughEtAl2018}.

Also note that a proper $r$-coloring $c$ of $\KG^{(p)}(n,k)$ gives a proper $r^{p-1}$-coloring $c'$ of $\KSh_p(n,k)$ defined by
$c'(A_1,\ldots,A_{p-1})=(c(A_i))_{i=1}^{p-1}$.
Consequently, $\chi(\KSh_p(n,k)) \leq \chi(\KG^{(p)}(n,k))^{p-1}$, whereas there is no obvious implication in the other direction.
Nevertheless, a recursive history construction will allow us to derive \cref{thm:kneser-shift} from the same $\Z_p$-equivariant topological machinery that underlies the Kneser hypergraph theorem of Alon, Frankl, and Lovász \cite{AlonFranklLovasz1986}.
More precisely, we use \cref{lem:ziegler-tucker}, a straightforward special case of Ziegler's $\Z_p$-Tucker lemma \cite{Ziegler2002}.

\subsection*{Organization and AI use}

The rest of this paper is organized as follows.
In \cref{sec:triangle} we present the main ideas of the proof of \cref{thm:main} for the case when $P$ is an equilateral triangle; we believe that reading these is sufficient for any expert to understand the whole proof.
These first two sections were drafted by ChatGPT, which was used heavily in the preparation of the whole manuscript, although none of the main mathematical ideas underlying the ncs-Ramsey theorem originated from it.
ChatGPT contributed significantly to the proof of \cref{thm:kneser-shift}, which involved a lot of joint brainstorming, and to improving the bounds.
Perhaps more importantly, ChatGPT was used to disprove several of my incorrect proof approaches, saving a significant amount of time.
It was also used for checking arguments, editing the exposition, and locating references.

The later sections give a detailed proof of the methods used in \cref{sec:triangle} for the general case and combine them with K\v r\'i\v z's method.
These were written by ChatGPT and edited by me very little, as I don't think that they contain any novelty.
Dense block maps and the geometric reduction are developed in \cref{sec:block-maps}.
The Kneser-shift theorem is proved in \cref{sec:kneser-shift}, dense simultaneous rotation systems in \cref{sec:rotation}, and cyclic extensions and solvable groups in \cref{sec:groups}.
The main part of the paper is concluded with open problems in \cref{sec:conclusion}.

After this part (which was v1 on arXiv), I asked ChatGPT to work a bit more on the problems, and it claims to have proved several interesting things, which can be found in the \cref{app:chatgpt-results}.
Among these are that all regular polytopes are ncs-Ramsey and a solution to a problem of Leader, Russell, and Walters \cite{LeaderRussellWalters2012}, that their Block Sets Conjecture holds for the templates $112233$ and $12345$.
I didn't verify these results because they use a lot of group theory which lies outside my field of expertise/interest, but I have no reason to doubt their correctness, and I believe that including these here can help future research.

\section{Main ideas presented for the equilateral triangle}\label{sec:triangle}

We use the equilateral triangle to showcase the main new ideas used in the proof of \cref{thm:main}.
So our goal in this section is to establish that an equilateral triangle is ncs-Ramsey.
Note that this already follows from Matou\v sek and R\"odl~\cite{MatousekRodl1995}, but our proof is entirely different and generalizes in a straightforward manner to other cyclic transitive sets, and, by basic group theory, to solvable transitive sets.
We start with the following special case of \cref{thm:kneser-shift}.
The proof below shows that we may take $C(r)=3\cdot2^r-2$; by definition, the optimal threshold satisfies $C(3,r)\leq C(r)$.

\begin{theorem}[Kneser-shift theorem for triangle]\label{thm:triangle-kneser-shift}
	$\chi(\KSh_3(n,k))\to\infty$ as $n-3k\to\infty$.
	In other words, whenever $n-3k\geq C(r)$, every $r$-coloring $c$ of the ordered pairs $(A,B)$, where $A,B\in\binom{[n]}k$ are disjoint, admits pairwise disjoint sets $A_1,A_2,A_3\in\binom{[n]}k$ for which $c(A_1,A_2)=c(A_2,A_3)$.
\end{theorem}
\begin{proof}
	Suppose that $c$ is a proper $r$-coloring of $\KSh_3(n,k)$.
	For a family $\mathcal A_0\subseteq\binom{[n]}k$ and any $B\in\binom{[n]}k$, define
	\[
		\tau^{\mathcal A_0}(B)
		=\{c(A,B):A\in\mathcal A_0,\ A\cap B=\emptyset\}\subseteq[r].
	\]
	Thus $\tau^{\mathcal A_0}(B)$ is an element of the Boolean lattice $2^{[r]}$, ordered by inclusion.
	Note that if $\mathcal A_0=\emptyset$, then $\tau^{\mathcal A_0}(B)=\emptyset\in2^{[r]}$ for any $B$.
	
	Let $\P_3(n,k)$ be the poset of triples $(\mathcal A_1,\mathcal A_2,\mathcal A_3)$ of families of $k$-subsets of $[n]$, not all empty, such that every member of one coordinate family is disjoint from every member of a different coordinate family, where in the poset order $(\mathcal A_1,\mathcal A_2,\mathcal A_3)$ is less or equal to $(\mathcal A_1',\mathcal A_2',\mathcal A_3')$ if $\mathcal A_i\subseteq \mathcal A_i'$ for every $i$.
	Let
	\[
		\mathcal A=(\mathcal A_1,\mathcal A_2,\mathcal A_3)
		\in\P_3(n,k),
	\]
	and read subscripts cyclically.
	For $i\in[3]$, let
	\[
		D_i(\mathcal A)
		=\mathord\downarrow
		\{\tau^{\mathcal A_{i-1}}(A):A\in\mathcal A_i\},
	\]
	where $\mathord\downarrow$ denotes the downward closure, taken in $2^{[r]}$.
	Note that $D_i(\mathcal A)=\emptyset$ if and only if 	$\mathcal A_i=\emptyset$, and $D_i(\mathcal A)=\{\emptyset\}$ if and only if $\mathcal A_i\neq\emptyset$ and $\mathcal A_{i-1}=\emptyset$.
	
	Let $\mathfrak D_r$ denote the finite poset of all downsets of $2^{[r]}$, ordered by inclusion, so $D_i(\mathcal A)\in\mathfrak D_r$.
	If $\mathcal A\leq\mathcal A'$ in $\P_3(n,k)$, then we obviously have $D_i(\mathcal A)\subseteq D_i(\mathcal A')$ for every $i\in[3]$.

\begin{claim}
	$D(\mathcal A)
	=(D_1(\mathcal A),D_2(\mathcal A),D_3(\mathcal A))$
	is never constant if $\mathcal A\in \P_3(n,k)$.
\end{claim}
\begin{proof}
	First, $D_i(\mathcal A)$ is empty exactly when $\mathcal A_i$ is empty, so the claim follows immediately if some, but not all, coordinate families are empty.
	This is because if $\mathcal A_i\neq\emptyset$ and $\mathcal A_{i-1}=\emptyset$, then $D_i(\mathcal A)=\{\emptyset\}$, while $D_{i-1}(\mathcal A)=\emptyset$.
	
	Suppose now that all three coordinate families are nonempty.
	We show that $D_{i+1}(\mathcal A)\not\subseteq D_i(\mathcal A)$ for every $i\in[3]$.
	If instead $D_{i+1}(\mathcal A)\subseteq D_i(\mathcal A)$, choose some $A_{i+1}\in\mathcal A_{i+1}$.
	Then $\tau^{\mathcal A_i}(A_{i+1})\in D_i(\mathcal A)$, so there is an $A_i\in\mathcal A_i$ such that $\tau^{\mathcal A_i}(A_{i+1})\subseteq\tau^{\mathcal A_{i-1}}(A_i)$.
	In particular, $c(A_i,A_{i+1})\in\tau^{\mathcal A_{i-1}}(A_i)$, and hence there is an $A_{i-1}\in\mathcal A_{i-1}$ such that $c(A_{i-1},A_i)=c(A_i,A_{i+1})$.
	The sets $A_{i-1},A_i,A_{i+1}$ are pairwise disjoint, contradicting the properness of $c$.
	This proves $D_{i+1}(\mathcal A)\not\subseteq D_i(\mathcal A)$, and therefore the claim.
\end{proof}

	Let $(\mathfrak D_r^3)^*$ be the subposet of nonconstant triples in $\mathfrak D_r^3$, ordered coordinatewise, so $D(\mathcal A)\in(\mathfrak D_r^3)^*$.
	Cyclic shift acts freely on $(\mathfrak D_r^3)^*$, so we may choose an equivariant map
	\[
		\mu_1:(\mathfrak D_r^3)^*\longrightarrow\Z_3.
	\]
	For $E=(E_1,E_2,E_3)\in(\mathfrak D_r^3)^*$, put
	\[
		\mu_2(E)=|E_1|+|E_2|+|E_3|.
	\]
	Each $E_i$ is a subset of $2^{[r]}$, and a triple with total size either $0$ or $3\cdot2^r$ is necessarily constant.
	Consequently,
	\[
		1\leq\mu_2(E)\leq3\cdot2^r-1=C(r)+1,
	\]
	Since $\mu_2$ is invariant under cyclic shifts, $\mu=(\mu_1,\mu_2)$ is an equivariant map from $(\mathfrak D_r^3)^*$ to $\Z_3\times[C(r)+1]$.
	The property of this labeling that we need is that if $E\leq F$ and $\mu_2(E)=\mu_2(F)$, then $\mu_1(E)=\mu_1(F)$.
	Indeed, $E_i\subseteq F_i$ for every $i$, and equality of the sums of their cardinalities forces $E_i=F_i$ for every $i$.
		
	We now apply Ziegler's \cite{Ziegler2002} $\Z_3$-Tucker \cref{lem:ziegler-tucker} directly.
	Let $\mathcal X_3(n)$ be the poset of triples $X=(X_1,X_2,X_3)$ of pairwise disjoint subsets of $[n]$, not all empty, ordered by coordinatewise inclusion.
	We construct an equivariant labeling
	\[
		\lambda:\mathcal X_3(n)\longrightarrow\Z_3\times[m],
		\qquad
		m=\left\lceil\frac{3k-2+C(r)}2\right\rceil.
	\]
	
	If $|X_i|<k$ for every $i\in[3]$, define
	\[
		\ell(X)=|X_1|+|X_2|+|X_3|
	\]
	and set $\lambda_1(X)$ so as to make $\lambda$ equivariant.
	
	If $|X_i|\geq k$ for at least one $i$, let
	\[
		\mathcal A(X)
		=\left(\binom{X_1}{k},\binom{X_2}{k},\binom{X_3}{k}\right)
		\in\P_3(n,k)
	\]
	and define
	\[
		\ell(X)=3k-3+\mu_2(D(\mathcal A(X))),
		\qquad
		\lambda_1(X)=\mu_1(D(\mathcal A(X))).
	\]
	In both cases put
	\[
		\lambda_2(X)=\left\lceil\frac{\ell(X)}2\right\rceil.
	\]
	The two cases are preserved by cyclic shifts, so $\lambda$ is equivariant.
	Moreover, in the first case $1\leq\ell(X)\leq3k-3$, while in the second case
	\[
		3k-2\leq\ell(X)\leq3k-2+C(r).
	\]
	Thus $\lambda_2(X)\in[m]$ in both cases.
		
	Suppose that $n-3k\geq C(r)$.
	Then
	\[
		2m\leq3k+C(r)-1\leq n-1,
	\]
	so $m\leq\lfloor(n-1)/2\rfloor$, as required in \cref{lem:ziegler-tucker}.
	That lemma gives a strict chain
	$X^{(1)}<X^{(2)}<X^{(3)}$ in the poset $\mathcal X_3(n)$
	whose three signs $\lambda_1(X^{(j)})$ are distinct while the three magnitudes $\lambda_2(X^{(j)})$ are equal.
	We claim that the integers $\ell(X^{(1)}),\ell(X^{(2)}),\ell(X^{(3)})$ are pairwise distinct.
	Two members in the first case have distinct $\ell$-values because total size strictly increases along a strict chain.
	A value in the first case is at most $3k-3$, whereas a value in the second case is at least $3k-2$.
	Finally, if two comparable members in the second case had the same $\ell$-value, then their associated families and history vectors $D(\mathcal A(X^{(i)}))$ would be comparable and their $\mu_2$-values would be equal, so the property of $\mu$ proved above would force their signs to be equal.
	This contradicts the distinctness of the signs on the Tucker chain.
	Thus the three $\ell$-values are distinct, although they all belong to one fiber of $t\mapsto\lceil t/2\rceil$.
	Every such fiber contains at most two positive integers, a contradiction.
	This finishes the proof of \cref{thm:triangle-kneser-shift}.
\end{proof}

Now we are ready to start the proof of \cref{thm:main} for the case when $P$ is an equilateral triangle.
Let us introduce one notation: we identify each element $x\in [3]^n$ with the triple $(A_1,A_2,A_3)$ where $i\in A_j$ if $x_i=j$; in this way we obtain a bijection with the ordered partitions $[n]=A_1\mathbin{\dot\cup}A_2\mathbin{\dot\cup}A_3$.
Just like in \cite{FranklRodl1987,Kriz1991,LeaderRussellWalters2012}, our goal is to show the following.

\begin{theorem}\label{thm:triangle-coord}
	For every $r\in\N$ and $\eta>0$ there are $n,k\in\N$ with $n=\lfloor(1+\eta)3k\rfloor$ such that every coloring $c:[3]^n\to[r]$ admits three pairwise disjoint blocks of coordinates of size $k$, so $B_1,B_2,B_3\in\binom{[n]}k$, and a fixing of all coordinates outside these blocks, with the following property.
	For $j\in[3]$, define $x^{(j)}\in[3]^n$ to take the fixed values outside $B_1\cup B_2\cup B_3$ and to equal $i-j$ on $B_i$, where the arithmetic is modulo $3$ and the residue $0$ is denoted by $3$.
	Then we have $c(x^{(1)})=c(x^{(2)})=c(x^{(3)})$.
\end{theorem}

This differs from earlier, similar results in the condition that $n\leq(1+\eta)3k$; this enables us to obtain a nearcircumsphere in \cref{thm:main} for the equilateral triangle by the following, standard scaling argument.
Assume that we want a triangle that is congruent to $v_1v_2v_3\subset \R^d$, which is translated so that its circumcenter is the origin, and write $\rho=\|v_1\|=\|v_2\|=\|v_3\|$.
Encode each word $x\in[3]^n$ by
\[
\phi(x)=\frac{1}{\sqrt{3k}}(v_{x_1},\ldots,v_{x_n})\in(\R^d)^n\cong\R^{dn}
\]
so we can pull back any coloring of $\R^{dn}$ (or a sphere) to a coloring of $[3]^n$ via $\phi$.
Every pair among $x^{(1)},x^{(2)},x^{(3)}$, given by \cref{thm:triangle-coord}, differs on all $3k$ coordinates belonging to $B_1\cup B_2\cup B_3$ and agrees everywhere else.
Moreover, at each coordinate where they differ, the corresponding vectors are two distinct vertices of the original triangle.
Consequently, for $i\neq j$,
\[
\|\phi(x^{(i)})-\phi(x^{(j)})\|^2
=\frac{3k}{3k}\|v_1-v_2\|^2,
\]
so $\phi(x^{(1)}),\phi(x^{(2)}),\phi(x^{(3)})$ form a congruent copy of the original triangle.
On the other hand, every encoded word has norm
\[
\|\phi(x)\|=\rho\sqrt{\frac{n}{3k}}\leq\rho\sqrt{1+\eta}.
\]
Thus, after choosing $\eta>0$ so that $\rho\sqrt{1+\eta}<\rho+\eps$, we may append the same orthogonal coordinate to every $\phi(x)$ to place all encoded words on the sphere $\S^{dn}_{\rho+\eps}$ without changing any distances.
Pulling back a coloring of this sphere to $[3]^n$ and applying the theorem therefore produces a monochromatic congruent copy of the triangle.
This finishes the sketch of how to derive \cref{thm:main} for the equilateral triangle from \cref{thm:triangle-coord}.\\

The main idea of the proof of \cref{thm:triangle-coord} is of course to apply \cref{thm:triangle-kneser-shift}.
From $c:[3]^n\to[r]$ we can get an $r$-coloring $c'$ of $\KSh_3(n,k)$ by setting $c'(A_1,A_2)=c(A_1, A_2, [n]\setminus (A_1\cup A_2))$.
\cref{thm:triangle-kneser-shift} gives some $c'(A_1,A_2)=c'(A_2,A_3)$, which corresponds to $c(A_1,A_2,[n]\setminus(A_1\cup A_2))=c(A_2,A_3,[n]\setminus(A_2\cup A_3))$.
In other words, we get some $x,y\in[3]^n$ for which $c(x)=c(y)$, both $x$ and $y$ have exactly $k$ coordinates that are $1$ and exactly $k$ coordinates that are $2$, and, apart from $n-3k=O(\eta k)$ coordinates on which $x_i=y_i=3$, we have $x_i\equiv y_i+1\pmod 3$; see \cref{fig:first-triangle-shift}.

\begin{figure}[ht]
	\centering
\begin{tikzpicture}[
	cell/.style={draw,minimum width=19mm,minimum height=7mm,font=\small},
	head/.style={minimum width=19mm,minimum height=6mm,font=\small},
	lastcell/.style={draw,minimum width=7mm,minimum height=7mm,font=\small},
	lasthead/.style={minimum width=7mm,minimum height=6mm,font=\small}]
	\foreach \q/\name in {0/A_1,1/A_2,2/A_3}{
		\node[head] at (1.9*\q,1.4) {$\name$};
	}
	\node[lasthead] at (5.1,1.4) {$R$};
	
	\foreach \q/\entry in {0/1,1/2,2/3}{
		\node[cell] at (1.9*\q,0.7) {$\entry$};
	}
	\node[lastcell] at (5.1,0.7) {$3$};
	
	\foreach \q/\entry in {0/3,1/1,2/2}{
		\node[cell] at (1.9*\q,0) {$\entry$};
	}
	\node[lastcell] at (5.1,0) {$3$};
	
	\node[left=4mm] at (-0.7,0.7) {$x$};
	\node[left=4mm] at (-0.7,0) {$y$};
	\node[font=\scriptsize] at (0,-0.65) {$k$};
	\node[font=\scriptsize] at (1.9,-0.65) {$k$};
	\node[font=\scriptsize] at (3.8,-0.65) {$k$};
	\node[font=\scriptsize] at (5.1,-0.65) {$O(\eta k)$};
\end{tikzpicture}
	\caption{The first Kneser-shift move.
	The columns indicate the coordinate blocks, and $R=[n]\setminus(A_1\cup A_2\cup A_3)$ denotes the few fixed coordinates.}
	\label{fig:first-triangle-shift}
\end{figure}

Thus, after fixing the common $3$'s, we have three $k$-size blocks for which schematically $c(123)=c(312)$, where we denoted the value in each monochromatic block by one symbol.
This is nice, but not what we wanted, because we would also need $c(231)$ to be equal to them for the conclusion of \cref{thm:triangle-coord}; we still need the group theoretic trick of K\v r\'i\v z, which we describe next.

Our goal will be to obtain a partition of all but a fixed, $O(\eta k)$ number of coordinates of $[n]$ into triples of blocks, $B_{b,1},B_{b,2},B_{b,3}$, of size $k_0$ each, where $k_0$ is also a carefully chosen, fixed integer.
The number of triples, $m$, needs to be large enough so that we can apply the argument described in the previous paragraph for them one more time, so that the triples of blocks would play the role of the coordinates.
More precisely, in the relevant $x$, each block $B_{b,j}$ will be constant, and the pattern on a triple of blocks will be one of $123$, $312$, and $231$.
We encode these three patterns by their last symbols, respectively $\bar3$, $\bar2$, and $\bar1$, so a configuration of the $m$ triples becomes a word in $\{\bar1,\bar2,\bar3\}^m$.
The triples will be chosen so that the induced coloring is insensitive to interchanging $\bar2$ and $\bar3$ in any collection of coordinates; equivalently, two macro-words have the same color whenever the positions occupied by $\bar1$ agree.
This already appeared in Shelah's proof of the Hales--Jewett theorem \cite{Shelah1988}, and it was termed a \emph{fliptop coloring} in \cite{GrahamRothschildSpencer1990}.
It also plays a central role in the Polymath1 proof of the density Hales--Jewett theorem, where such sets are called $23$-insensitive \cite[Section~5.4 and Section~8]{Polymath2012}.

We will apply the Kneser-shift move also at this outer level, but this time we need to encode more into $c'$ (reusing this symbol, with a slight abuse of notation), so we define $c'(A_1,A_2)=(c(A_1, A_2, [n]\setminus (A_1\cup A_2)),c(A_2, A_1, [n]\setminus (A_1\cup A_2)))$.
An application of \cref{thm:triangle-kneser-shift} gives three $k$-size blocks for which, after fixing the common $\bar3$'s, we schematically have both $c(\bar1\bar2\bar3)=c(\bar3\bar1\bar2)$ and  $c(\bar2\bar1\bar3)=c(\bar3\bar2\bar1).$
The first-level $(\bar2,\bar3)$-insensitivity and outer rotations then give
$c(\bar1\bar2\bar3)
=c(\bar3\bar1\bar2)
=c(\bar2\bar1\bar3)
=c(\bar3\bar2\bar1)
=c(\bar2\bar3\bar1)$; see \cref{fig:triangle}.
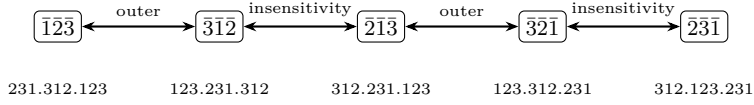
\begin{figure}[ht]
	\centering
	\begin{tikzpicture}[
		word/.style={draw,rounded corners=2pt,inner sep=2.5pt,font=\scriptsize},
		both/.style={{Stealth[length=1.7mm]}-{Stealth[length=1.7mm]},thick},
		node distance=15mm]
		\node[word] (a) {$\bar1\bar2\bar3$};
		\node[word,right=of a] (b) {$\bar3\bar1\bar2$};
		\node[word,right=of b] (c) {$\bar2\bar1\bar3$};
		\node[word,right=of c] (d) {$\bar3\bar2\bar1$};
		\node[word,right=of d] (e) {$\bar2\bar3\bar1$};
		\draw[both] (a)--node[above,font=\tiny]{outer} (b);
		\draw[both] (b)--node[above,font=\tiny]{insensitivity} (c);
		\draw[both] (c)--node[above,font=\tiny]{outer} (d);
		\draw[both] (d)--node[above,font=\tiny]{insensitivity} (e);
		\node[below=4mm of a,font=\tiny] {$231.312.123$};
		\node[below=4mm of b,font=\tiny] {$123.231.312$};
		\node[below=4mm of c,font=\tiny] {$312.231.123$};
		\node[below=4mm of d,font=\tiny] {$123.312.231$};
		\node[below=4mm of e,font=\tiny] {$312.123.231$};
	\end{tikzpicture}
	\caption{The two-level fliptop calculation.
		Outer arrows use the Kneser-shift rotation, and the other arrows interchange $\bar2$ and $\bar3$ while preserving the positions of $\bar1$.}
	\label{fig:triangle}
\end{figure}

Decoding the three relevant macro-words $\bar1\bar2\bar3$, $\bar3\bar1\bar2$, and $\bar2\bar3\bar1$ gives
\[
231.312.123,
\qquad
123.231.312,
\qquad
312.123.231,
\]
which are coordinatewise cyclic rotations of one another.
Thus they represent a monochromatic equilateral triangle.\\

The only thing missing is thus to explain how we can obtain $m$ triples of blocks that satisfy the above property.
For simplicity, in this section's proof sketch we will only do this for $m=2$, but the same idea works for any larger value.
This trick is also a new contribution.

\begin{proposition}\label{prop:two-triangle-gadgets}
	For every $r\in\N$ and $\eta>0$ there are $n,k\in\N$ with $n=\lfloor(1+\eta)6k\rfloor$ such that every coloring $c:[3]^n\to[r]$ admits six pairwise disjoint blocks of coordinates of size $k$, so
	\[
		B_{1,1},B_{1,2},B_{1,3},B_{2,1},B_{2,2},B_{2,3}\in\binom{[n]}k,
	\]
	and a fixing of the $O(\eta k)$ remaining coordinates outside the blocks, with the following property.
	Call $x\in[3]^n$ \emph{good} if it is constant on each of the six blocks and takes the fixed values on the remaining coordinates.
	Represent each good $x$ by the word $\tilde x\in[3]^6$ formed by these six constant values, and, with a slight abuse of notation, write $c(\tilde x)=c(x)$.
	Then, for every $\tilde x\in[3]^6$, we have
	\[
		c(\tilde x_1\tilde x_2\tilde x_3 123)
		=c(\tilde x_1\tilde x_2\tilde x_3 312)
	\]
	and
	\[
		c(123\tilde x_4\tilde x_5\tilde x_6)
		=c(312\tilde x_4\tilde x_5\tilde x_6).
	\]
\end{proposition}

\begin{proof}
	In the literature, similar statements are proved by first picking the last group of blocks, in our case $B_{2,1},B_{2,2},B_{2,3}$.
	If we split the coordinates after some initial segment of length $n_0$, the usual method would color a word $x'\in[3]^{n-n_0}$ by its complete color profile
	\[
		c'(x')=\bigl(c(xx'):x\in[3]^{n_0}\bigr).
	\]
	However, this uses $r^{3^{n_0}}$ colors, which depends on $n_0$.
	We cannot afford this when $n\approx2n_0\approx6k$, since the excess $(n-n_0)-3k$ is only of order $\eta k$, much less than $r^{3^{n_0}}$, so we could not invoke \cref{thm:triangle-kneser-shift}.

	To overcome this problem, partition the first $n_0$ coordinates into \emph{atoms} of size $a$ and record only those words that are constant on every atom.
	If there are $n_0/a$ atoms, the resulting profile coloring of the last coordinates has only $r^{3^{n_0/a}}$ colors.
	After applying \cref{thm:triangle-kneser-shift} to this profile coloring, we obtain the blocks $B_{2,1},B_{2,2},B_{2,3}$ while preserving all atom-constant contexts on the first part.
	This argument implies $c(\tilde x_1\tilde x_2\tilde x_3 123)
	=c(\tilde x_1\tilde x_2\tilde x_3 312)$.

	We then regard the atoms themselves as coordinates and apply \cref{thm:triangle-kneser-shift} once more to select $B_{1,1},B_{1,2},B_{1,3}$ as unions of $k/a$ atoms each.
	At this stage we color each atom-word by the profile of its colors over the $3^3=27$ constant assignments on the already selected blocks $B_{2,1},B_{2,2},B_{2,3}$, so the number of colors is at most $r^{27}$.
	This implies $c(123\tilde x_4\tilde x_5\tilde x_6)
	=c(312\tilde x_4\tilde x_5\tilde x_6).$
	
	Therefore, to be able to apply \cref{thm:triangle-kneser-shift} both times, the number of remaining coordinates in the two cases need to satisfy the following inequalities, respectively:
	\[
		\eta k/a\geq C(r^{27})
	\qquad\text{and}\qquad
		\eta k\geq C\!\left(r^{3^{n_0/a}}\right).
	\]
	This is easy to arrange in the correct order.
	First choose $k/a$ large enough for the first inequality; this also fixes $n_0/a$.
	Then choose $k$ large enough for the second inequality.
	This finishes the sketch of the proof of \cref{prop:two-triangle-gadgets}, and thus the sketch of the proof of \cref{thm:triangle-coord}.
\end{proof}

For arbitrary $m$, the same construction uses a reverse hierarchy of atom sizes $a_1,a_2,\ldots$, equivalently of atom counts $q_b=k/a_b$.
When a block tuple at one scale is chosen, its complete color profile records every atom-constant context on the scales that will be processed later.
The general Kneser-shift input is \cref{thm:kneser-shift}, proved in \cref{sec:kneser-shift}, while the precise simultaneous atom construction is \cref{thm:rotation-system}.
For $p=3$, $X=[3]$, $u=3$, $z_1=1$, and $z_2=2$, the identity \eqref{eq:rotation} is the contextual version of the two identities in \cref{prop:two-triangle-gadgets}: it makes $123$ and $312$, hence $\bar3$ and $\bar2$, interchangeable in every context.
The outer fliptop calculation is formalized by \cref{lem:one-orbit}, whose balanced one-variable conclusion for $X=[3]$ gives \cref{thm:triangle-coord}; the passage to an ncs-Ramsey theorem is the geometric reduction in \cref{prop:geometry}.

For a cyclic group of prime order $p$, the same argument performs $p-1$ successive fusions; this is K\v r\'i\v z's method~\cite{Kriz1991} and is proved in \cref{lem:one-orbit}.
Primality enters through both the $\Z_p$-Tucker step and the freeness needed to choose an equivariant sign on nonconstant history vectors; see \cref{rem:prime-case}.
Multiple $C_p$-orbits are handled in \cref{prop:cyclic-action}, and \cref{prop:extension} passes the dense block property through a normal subgroup with prime cyclic quotient.
Iterating this along a composition series proves \cref{thm:DB-solvable}, and \cref{prop:geometry} then yields \cref{thm:main}.
The rest of the paper supplies these details.

\section{Dense block maps and the geometric reduction}\label{sec:block-maps}

We first introduce a compact language for the block constructions used in \cref{sec:triangle}.
Let a finite group $G$ act on a finite alphabet $X$ on the left, so $g(hx)=(gh)x$.
In the triangle case, $X=[3]$ and $G=C_3\cong\Z_3$ acts by cyclically permuting the three letters.
The terminology is adapted from the fixed-degree $G$-copies of Leader, Russell, and Walters~\cite[Section~2, Conjecture~C]{LeaderRussellWalters2012}; a related uniform Hales--Jewett framework is used by Kanellopoulos and Karamanlis~\cite{KanellopoulosKaramanlis2020}.

\begin{definition}[Block map]\label{def:block-map}
A \emph{$G$-block map of dimension $m$, length $n$, and block size $k$} is a map $\Phi:X^m\to X^n$ for which there are pairwise disjoint sets $I_1,\ldots,I_m\subseteq[n]$, each of cardinality $k$, labels $\gamma_\ell\in G$ for $\ell\in I_1\cup\cdots\cup I_m$, and fixed letters $z_\ell\in X$ outside this union, such that
\[
  \Phi(x_1,\ldots,x_m)_\ell=
  \begin{cases}
    \gamma_\ell x_j,&\ell\in I_j,\\
    z_\ell,&\ell\notin I_1\cup\cdots\cup I_m.
  \end{cases}
\]
The sets $I_1,\ldots,I_m$ are the \emph{blocks} of the map, and $mk/n$ is its \emph{active proportion}.
\end{definition}

Thus each input variable is repeated on a block of $k$ coordinates, possibly after applying different elements of $G$, while all remaining coordinates are fixed.
For the three words in \cref{thm:triangle-coord}, there is one variable and its block is divided into three equal parts labeled by the three elements of $C_3$.
No equidistribution of the labels is required in the general definition; only the common block size matters for distances.

\begin{definition}[Dense block property]\label{def:DB}
We say that the action $G\curvearrowright X$ has the \emph{dense block property}, abbreviated $\mathrm{DB}(G\curvearrowright X)$, if for every $r,m\in\N$ and every $0<\eta<1$ there are $n,k\in\N$ such that
\[
 mk\geq(1-\eta)n
\]
and every coloring $c:X^n\to[r]$ admits a $G$-block map $\Phi:X^m\to X^n$ of block size $k$ satisfying
\begin{equation}\label{eq:orbit-insensitivity}
  c(\Phi(x_1,\ldots,x_m))=c(\Phi(y_1,\ldots,y_m))
  \quad\text{whenever }y_j\in Gx_j\text{ for all }j.
\end{equation}
\end{definition}

In other words, after restricting the coloring to the selected block words, the color depends only on the $G$-orbit of each variable.
Up to a harmless reparameterization of $\eta$, the density inequality is the general form of $n\leq(1+\eta)3k$ in \cref{thm:triangle-coord}.
This is a density-strengthened, multi-variable version of \cite[Conjecture~C]{LeaderRussellWalters2012}.

An analogous exact floor normalization can also be imposed here.

\begin{lemma}[Normalization]\label{lem:DB-equivalence}
In \cref{def:DB}, one may equivalently replace the inequality
\[
 mk\geq(1-\eta)n
\]
by the exact equality $mk=\lfloor(1-\eta)n\rfloor$.
\end{lemma}

\begin{proof}
Suppose first that the inequality formulation holds, and fix $r,m\in\N$ and $0<\eta<1$.
Choose $0<\delta<\eta$, and let $n_0,k$ be supplied by the inequality formulation with loss $\delta$.
Put $\ell=mk$ and
\[
 n=\left\lceil\frac{\ell}{1-\eta}\right\rceil.
\]
Since $\ell\geq(1-\delta)n_0$ and $\delta<\eta$, we have $n\geq n_0$, while
\[
 \ell\leq(1-\eta)n<\ell+1.
\]
Fix a letter $u\in X$.
Given a coloring of $X^n$, fix its last $n-n_0$ coordinates to $u$, apply the inequality formulation to the resulting coloring of $X^{n_0}$, and append these fixed coordinates to the block map obtained there.
The number of active coordinates remains $\ell=\lfloor(1-\eta)n\rfloor$, proving the exact formulation.

Conversely, suppose that the exact formulation holds, and fix $r,m\in\N$ and $0<\eta<1$.
Choose $q\in\N$ so large that $1/(qm)<\eta/2$, and apply the exact formulation with $qm$ variables and parameter $\eta/2$.
Identify each of $m$ groups of $q$ input variables.
The unions of the corresponding $q$ blocks form an $m$-dimensional block map of block size $qk$, and orbit-insensitivity is preserved under this identification.
Writing
\[
 \ell=qmk=\lfloor(1-\eta/2)n\rfloor,
\]
we have $\ell<n$, while $k\geq1$ gives $\ell\geq qm$.
Thus $n>qm$, and hence
\[
 \frac{\ell}{n}>1-\frac{\eta}{2}-\frac1n>1-\eta.
\]
This proves the inequality formulation.
\end{proof}

The composition arguments below use the inequality formulation in \cref{def:DB}.

The combinatorial core of the paper is the following theorem.

\begin{theorem}[Dense block theorem]\label{thm:DB-solvable}
If a finite solvable group $G$ acts on a finite set $X$, then $\mathrm{DB}(G\curvearrowright X)$ holds.
\end{theorem}

For $G=C_3$ acting transitively on $[3]$, the proof will yield the stronger label-balanced conclusion that gives \cref{thm:triangle-coord}.
The density conclusion is exactly what will bring the radius of the ambient sphere arbitrarily close to the circumradius.

We use arbitrary $m$ in \cref{def:DB} only because these block maps will be composed; the geometric application needs only $m=1$.

\begin{lemma}[Composition of block maps]\label{lem:block-composition}
Suppose $\Phi:X^{n_0}\to X^n$ and $\Psi:X^m\to X^{n_0}$ are $G$-block maps of block sizes $k_1$ and $k_2$, respectively.
Then $\Phi\circ\Psi$ is a $G$-block map of block size $k_1k_2$.
Moreover,
\[
  \frac{m k_1k_2}{n}
  =\frac{n_0 k_1}{n}\frac{m k_2}{n_0}.
\]
\end{lemma}

\begin{proof}
If the $i$th macrocoordinate of $\Phi$ is active in $k_1$ physical coordinates and the $j$th variable of $\Psi$ is active on $k_2$ macrocoordinates, then the $j$th variable of the composite is active on $k_1k_2$ physical coordinates.
At such a coordinate its label is the product of the two labels.
All remaining coordinates are fixed.
The formula says that the active proportions multiply.
\end{proof}

We next isolate the geometric consequence.
Its product-embedding argument is the dense analogue of the reduction from group copies to transitive Euclidean sets in~\cite[Proposition~2.1]{LeaderRussellWalters2012}.

\begin{proposition}[Geometric reduction]\label{prop:geometry}
Let $P\subset\R^d$ be a spherical set with circumradius $\rho$, and suppose that a finite group $G$ of isometries acts transitively on $P$.
If $\mathrm{DB}(G\curvearrowright P)$ holds, then $P$ is ncs-Ramsey.
\end{proposition}

\begin{proof}
Translate the circumcenter of $P$ to the origin in its affine span.
The circumcenter is unique there, hence every isometry in $G$ fixes it.
Thus $\|x\|=\rho$ for every $x\in P$.

Fix $r\in\N$ and $\eps>0$, put $R=\rho+\eps$, and choose $0<\eta<1$ so that
\begin{equation}\label{eq:eta-radius}
  1-\eta\geq \frac{\rho^2}{R^2}.
\end{equation}
Apply \cref{def:DB} with $m=1$, and let $n$ be the length and $k$ the block size it supplies.
Thus $k\geq(1-\eta)n$, and by \eqref{eq:eta-radius} we have $n\rho^2/k\leq R^2$.
Therefore
\[
  t=\sqrt{R^2-\frac{n\rho^2}{k}}
\]
is real.
The finite grid
\[
  \iota(P^n)=
  \left\{\left(k^{-1/2}x_1,\ldots,k^{-1/2}x_n,t\right):
  (x_1,\ldots,x_n)\in P^n\right\}
\]
lies on $\S_R^{dn}$, since for every displayed point
\[
 \|\iota(x_1,\ldots,x_n)\|^2
 =\frac1k\sum_{i=1}^n\|x_i\|^2+t^2
 =\frac{n\rho^2}{k}+t^2=R^2.
\]

Restrict an arbitrary $r$-coloring of this sphere to the grid and pull it back to a coloring of $P^n$.
Let $\Phi:P\to P^n$ be the block map supplied above.
Because $G$ is transitive, all points $\iota(\Phi(x))$, $x\in P$, have one color.
If $I$ is the active set of $\Phi$, then for $x,y\in P$,
\[
\begin{split}
 \|\iota(\Phi(x))-\iota(\Phi(y))\|^2
 &=\frac1k\sum_{\ell\in I}\|\gamma_\ell x-\gamma_\ell y\|^2\\
 &=\frac1k\sum_{\ell\in I}\|x-y\|^2
 =\|x-y\|^2.
\end{split}
\]
Thus $x\mapsto\iota(\Phi(x))$ is an isometric embedding of $P$ into the sphere, completing the proof.
\end{proof}

The rest of the paper proves the dense block theorem, \cref{thm:DB-solvable}.

\section{Proof of the Kneser-shift theorem}\label{sec:kneser-shift}

We prove \cref{thm:kneser-shift}.
The successive-shift conclusion, rather than only the usual Kneser-hypergraph theorem, is central to the later atom construction.
We use Ziegler's $\Z_p$-Tucker lemma as our sole topological black box, isolate the lifting step already used in the triangle case, and then generalize the rank-sum history labeling from that proof.

We first quote the one external lemma used in this section.
It is Ziegler's prime-cyclic extension of Tucker's original combinatorial lemma \cite{Tucker1946}.
We state precisely the special case of \cite[Lemma~5.3]{Ziegler2002} used below.
Throughout this section, cyclic coordinates are indexed by $[p]$; when we identify them with $\Z_p$, the index $p$ represents zero and all subscripts are reduced cyclically.
Let $\mathcal X_p(n)$ be the poset of tuples
\[
 X=(X_1,\ldots,X_p)
\]
of pairwise disjoint subsets of $[n]$, not all empty, ordered by coordinatewise inclusion.
Let
\[
\sigma(X_1,\ldots,X_p)=(X_p,X_1,\ldots,X_{p-1}).
\]
The group $\Z_p$ acts by powers of $\sigma$, and on $\Z_p\times[m]$ we use $\sigma(a,t)=(a+1,t)$.

\begin{lemma}[$\Z_p$-Tucker lemma; Ziegler \cite{Ziegler2002}]\label{lem:ziegler-tucker}
Let $p\geq2$ be prime and let $n,m\geq1$.
If
\[
 \lambda:\mathcal X_p(n)\longrightarrow\Z_p\times[m]
\]
is equivariant and
\[
 m\leq\left\lfloor\frac{n-1}{p-1}\right\rfloor,
\]
then there is a strict chain
\[
 X^{(1)}<X^{(2)}<\cdots<X^{(p)}
\]
such that all $\lambda_2(X^{(j)})$ are equal and the $p$ signs $\lambda_1(X^{(j)})$ are the distinct elements of $\Z_p$.
\end{lemma}

The form above already has the signed-integer target needed for the lifting argument.
We derive from it the exact lifting statement needed later.

For $p\geq2$, let $\P_p(n,k)$ consist of tuples
\[
  \mathcal A=(\mathcal A_1,\ldots,\mathcal A_p),
\]
where the $\mathcal A_i$ are families of $k$-subsets of $[n]$, not all empty, and every member of one coordinate family is disjoint from every member of any other.
We order these tuples by coordinatewise inclusion, and $\Z_p$ acts by a cyclic shift.
This action is free for every $p$.
Indeed, a tuple fixed by a nonidentity shift would repeat every nonempty coordinate family in at least two distinct coordinates, forcing each of its members to be disjoint from itself.
Unlike the action on arbitrary nonconstant $p$-tuples used later, this action remains free when $p$ is composite.

\begin{lemma}[Kneser--Tucker lemma]\label{lem:kneser-tucker}
Let $p$ be prime and let $m\geq1$.
Suppose that
\[
 \mu=(\mu_1,\mu_2):
 \P_p(n,k)\longrightarrow\Z_p\times[m]
\]
is equivariant under cyclic shifts.
Suppose moreover that whenever $\mathcal A\leq\mathcal B$,
\begin{equation}\label{eq:signed-tie}
 \mu_2(\mathcal A)=\mu_2(\mathcal B)
 \quad\Longrightarrow\quad
 \mu_1(\mathcal A)=\mu_1(\mathcal B).
\end{equation}
Then $n-pk<m-1$.
Equivalently, no such labeling exists if $n-pk\geq m-1$.
\end{lemma}

\begin{proof}
We apply \cref{lem:ziegler-tucker} to $\mathcal X_p(n)$.
Put $\ell_0=p(k-1)$.
We construct the Tucker labeling $\lambda$ in parallel with the proof for $p=3$.
If $|X_i|<k$ for every $i$, set
\[
 \ell(X)=\sum_{i=1}^p|X_i|.
\]
Let $\lambda_1(X)$ be the index in $\Z_p$ of the unique coordinate containing the least element of $\bigcup_{i=1}^pX_i$.
This is equivariant under cyclic shifts of the coordinates.

If $|X_i|\geq k$ for at least one $i$, set
\[
 \mathcal A(X)=
 \left(\binom{X_1}{k},\ldots,\binom{X_p}{k}\right)\in\P_p(n,k)
\]
and define
\[
 \ell(X)=\ell_0+\mu_2(\mathcal A(X)),\qquad
 \lambda_1(X)=\mu_1(\mathcal A(X)).
\]
In both cases put
\[
 \lambda_2(X)=\left\lceil\frac{\ell(X)}{p-1}\right\rceil.
\]
The two cases are invariant under the cyclic action, so $\lambda$ is equivariant.
Moreover $1\leq\ell(X)\leq\ell_0+m$; hence its range is contained in $\Z_p\times[m_0]$, where
\[
 m_0=\left\lceil\frac{\ell_0+m}{p-1}\right\rceil.
\]
If $n-pk\geq m-1$, then
\[
 (p-1)m_0\leq \ell_0+m+p-2=pk+m-2\leq n-1.
\]
The numerical hypothesis of \cref{lem:ziegler-tucker} is therefore satisfied.

It remains to check that its conclusion is impossible.
Consider the $p$-element chain supplied by \cref{lem:ziegler-tucker}, whose signs are pairwise distinct.
Two low-case members have distinct $\ell$-values because their total sizes strictly increase along the chain.
A low-case value is at most $\ell_0$, whereas a high-case value is at least $\ell_0+1$.
Finally, if $X<Y$ are two high-case members with $\ell(X)=\ell(Y)$, then $\mathcal A(X)\leq\mathcal A(Y)$ and $\mu_2(\mathcal A(X))=\mu_2(\mathcal A(Y))$.
The tie condition \eqref{eq:signed-tie} would then give $\lambda_1(X)=\lambda_1(Y)$, contrary to the distinct signs on the Tucker chain.
Thus the chain has $p$ distinct $\ell$-values.
But all these values lie in one fiber of $t\mapsto\lceil t/(p-1)\rceil$, and every such fiber contains at most $p-1$ positive integers.
This is a contradiction.
\end{proof}

The primality hypothesis is deliberate.
Ziegler proves \cref{lem:ziegler-tucker} for prime $p$ and obtains the ordinary composite-uniformity Kneser theorem separately, by induction on the prime factors~\cite[Section~5]{Ziegler2002}.
This does not show that the Kneser-shift conclusion is false for composite uniformities; it only shows that the present prime-cyclic Tucker argument does not establish it.

For us, no prime-power variant will be needed.
If a cyclic quotient has order $p^a$, its subgroup chain refines that extension into $a$ extensions with quotient $C_p$.
More generally, the solvable-group argument proceeds through a composition series whose factors all have prime order.
This refinement is carried out in \cref{prop:cyclic-action,prop:extension}; see the proof of \cref{thm:DB-solvable} for the iteration along a composition series.

The labeling used next has close precedents in work on monochromatic monotone paths.
Fox, Pach, Sudakov, and Suk assign to an ordered $(k-1)$-tuple the lengths of the longest monochromatic tight paths ending there; comparison of successive tuples drives their uniformity-reduction recurrence \cite[Theorem~2.1]{FoxPachSudakovSuk2012}.
Moshkovitz and Shapira develop the predecessor-downset mechanism in every uniformity.
In the $3$-uniform base case---the one directly connected with cups and caps---they map a vertex to the downset generated by its predecessor pair-labels \cite[Lemma~2.2]{MoshkovitzShapira2014}; for general uniformity they recursively iterate the same construction \cite[Lemma~3.2]{MoshkovitzShapira2014}.
Our histories use an analogous iteration in a cyclic equivariant poset of disjoint set systems.
We use neither result as a black box.

We now turn a hypothetical proper coloring into the labeling required by \cref{lem:kneser-tucker}.
At the top level, properness makes the colors of successive shifts incomparable.
Repeatedly taking downsets of predecessor histories propagates this nondomination down to one-set histories.
The resulting cyclic vector of histories is nonconstant and equivariant, and, just as in the triangle proof, we use the sum of the cardinalities of its coordinates as an invariant integer rank.

For $h\geq1$, let $t_h(r)$ denote the cardinality of the poset obtained from the $r$-element antichain by $h-1$ successive applications of $\D$.
Thus $t_1(r)=r$, while the posets $T_1$ and $T_0$ below have cardinalities $t_{p-1}(r)$ and $t_p(r)$, respectively.

\begin{proof}[Proof of \cref{thm:kneser-shift}]
Assume, toward a contradiction, that $c$ is a proper $r$-coloring of $\KSh_p(n,k)$.
Define finite posets
\[
  T_{p-1}=[r]\quad\text{with the antichain order},
  \qquad T_j=\D(T_{j+1})\quad(0\leq j\leq p-2),
\]
where $\D(T)$ is the poset of downsets of $T$, ordered by inclusion.
Write $\mathfrak D_{p,r}=T_0$.
When $p=3$, we have $T_1=2^{[r]}$ and $\mathfrak D_{3,r}=\mathfrak D_r$, so the construction below specializes exactly to the one in \cref{sec:triangle}.

Fix $\mathcal A=(\mathcal A_1,\ldots,\mathcal A_p)\in\P_p(n,k)$.
From now on, all subscripts attached to these cyclically ordered families and their members lie in $[p]$ and are read cyclically.
Whenever all displayed members have been chosen from their corresponding families, define the top-level history of a consecutive $(p-1)$-tuple by
\[
 \tau_{p-1}^{\mathcal A}(A_{i-p+2},\ldots,A_i)
 =c(A_{i-p+2},\ldots,A_i).
\]
Recursively, for $1\leq j\leq p-2$, put
\begin{equation}\label{eq:history}
 \tau_j^{\mathcal A}(A_{i-j+1},\ldots,A_i)
 =\mathord\downarrow\!\left\{
 \tau_{j+1}^{\mathcal A}(B,A_{i-j+1},\ldots,A_i):
 B\in\mathcal A_{i-j}\right\}.
\end{equation}
These are the recursive shift histories.

\begin{lemma}[History properties]\label{lem:histories}
The histories have the following properties.
\begin{enumerate}[label=\textup{(\roman*)}]
\item If $\mathcal A\leq\mathcal B$, then every history built from members of $\mathcal A$ can only increase when computed in $\mathcal B$.
\item If every coordinate family of $\mathcal A$ is nonempty, then, for every $1\leq j\leq p-1$ and every consecutive $(j+1)$-tuple,
\begin{equation}\label{eq:nondomination}
 \tau_j^{\mathcal A}(A_{i-j+1},\ldots,A_i)
 \not\leq
 \tau_j^{\mathcal A}(A_{i-j},\ldots,A_{i-1}).
\end{equation}
\end{enumerate}
\end{lemma}

\begin{proof}
Part (i) follows by downward induction.
At the top level the color is unchanged.
At each lower level, every old generator can only increase by induction, while enlarging the relevant coordinate family may also add generators; taking downward closures therefore preserves inclusion.

For (ii), use downward induction on $j$.
When $j=p-1$, the two histories are the colors of adjacent successive tuples, which are distinct because $c$ is proper; distinct elements of $T_{p-1}$ are incomparable.
At a lower level, choose $A_{i-j}\in\mathcal A_{i-j}$, which is possible because all coordinate families are nonempty.
The element
\[
 \tau_{j+1}^{\mathcal A}
 (A_{i-j},A_{i-j+1},\ldots,A_i)
\]
is one of the generators of the downset on the left of \eqref{eq:nondomination}.
If that downset were contained in the one on the right, then, by the definition of a generated downset, there would be some $B\in\mathcal A_{i-j-1}$ such that
\[
 \tau_{j+1}^{\mathcal A}
 (A_{i-j},A_{i-j+1},\ldots,A_i)
 \leq
 \tau_{j+1}^{\mathcal A}
 (B,A_{i-j},\ldots,A_{i-1}).
\]
This is exactly the domination excluded by the induction hypothesis at level $j+1$.
\end{proof}

For $i\in[p]$, define
\[
 D_i(\mathcal A)=\mathord\downarrow
 \{\tau_1^{\mathcal A}(A):A\in\mathcal A_i\}\in\D(T_1).
\]
Notice that $D_i(\mathcal A)$ is empty exactly when $\mathcal A_i$ is empty.
If $p\geq3$, $\mathcal A_i\neq\emptyset$, and $\mathcal A_{i-1}=\emptyset$, then $\tau_1^{\mathcal A}(A)=\emptyset$ for every $A\in\mathcal A_i$, but $D_i(\mathcal A)=\mathord\downarrow\{\emptyset\}=\{\emptyset\}\neq\emptyset$.
If all coordinate families are nonempty, then \cref{lem:histories}(ii) gives
\begin{equation}\label{eq:D-nondomination}
  D_{i+1}(\mathcal A)\not\leq D_i(\mathcal A)
  \qquad\text{for every }i.
\end{equation}
Indeed, containment would place some $\tau_1^{\mathcal A}(A_{i+1})$ below a $\tau_1^{\mathcal A}(A_i)$, contradicting \eqref{eq:nondomination}.

Write
\[
 D(\mathcal A)
 =(D_1(\mathcal A),\ldots,D_p(\mathcal A))
 \in\mathfrak D_{p,r}^p.
\]
This vector is nonconstant.
If some, but not all, coordinate families are empty, then some, but not all, of the $D_i(\mathcal A)$ are empty; if all coordinate families are nonempty, \eqref{eq:D-nondomination} excludes a constant vector.
An induction on $j$ in \eqref{eq:history} also shows that cyclically shifting the coordinate families shifts every history index in the same sense.
Consequently the map $\mathcal A\mapsto D(\mathcal A)$ is equivariant.

We now assign a signed integer to every possible nonconstant history vector.
Let $(\mathfrak D_{p,r}^p)^*$ be the subposet of nonconstant vectors in $\mathfrak D_{p,r}^p$, with coordinatewise order.
Since $p$ is prime, cyclic shift acts freely on $(\mathfrak D_{p,r}^p)^*$: a vector fixed by a nontrivial shift would have all coordinates equal.
We may therefore choose an equivariant map
\[
 \mu_1:(\mathfrak D_{p,r}^p)^*\longrightarrow\Z_p.
\]
For $E=(E_1,\ldots,E_p)\in(\mathfrak D_{p,r}^p)^*$, put
\[
 \mu_2(E)=\sum_{i=1}^p|E_i|.
\]
Each $E_i$ is a downset in $T_1$, regarded as a subset of $T_1$.
Since $E$ is nonconstant, its total size is neither $0$ nor $p|T_1|$, and hence
\[
 1\leq\mu_2(E)\leq p|T_1|-1.
\]
Consequently,
\[
 \mu=(\mu_1,\mu_2):
 (\mathfrak D_{p,r}^p)^*\longrightarrow
 \Z_p\times[p|T_1|-1]
\]
is equivariant, because $\mu_2$ is invariant under cyclic shifts.
Therefore $\mathcal A\mapsto\mu(D(\mathcal A))$ is an equivariant labeling on $\P_p(n,k)$.
We verify the tie property required in \cref{lem:kneser-tucker}.
Suppose that $\mathcal A\leq\mathcal B$ and $\mu_2(D(\mathcal A))=\mu_2(D(\mathcal B))$.
By \cref{lem:histories}(i), we have $D_i(\mathcal A)\subseteq D_i(\mathcal B)$ for every $i$.
Equality of the sums of the coordinate cardinalities therefore forces $D_i(\mathcal A)=D_i(\mathcal B)$ for every $i$.
Thus $D(\mathcal A)=D(\mathcal B)$ and $\mu_1(D(\mathcal A))=\mu_1(D(\mathcal B))$.

The map $\mathcal A\mapsto\mu(D(\mathcal A))$ now satisfies the hypotheses of \cref{lem:kneser-tucker}.
Therefore a proper $r$-coloring can exist only when
\[
 n-pk<p|T_1|-2.
\]
By the definition of the optimal threshold $C(p,r)$, this gives
\begin{equation}\label{eq:C-bound}
 C(p,r)
 \leq p|T_1|-2
 =p\,t_{p-1}(r)-2.
\end{equation}
This proves \cref{thm:kneser-shift}.
\end{proof}

We now compare the explicit upper bound supplied by the history construction with a lower bound coming from iterated arc colorings.
For $s\in\N$, put
\[
 B(s)=\binom{s}{\lfloor s/2\rfloor},
\]
and let $B^{\circ j}$ denote the $j$-fold iterate of $B$, with $B^{\circ0}(s)=s$.
Define
\[
 \operatorname{twr}_1(x)=x,
 \qquad
 \operatorname{twr}_{h+1}(x)=2^{\operatorname{twr}_h(x)}.
\]

\begin{proposition}[Quantitative bounds]\label{prop:kneser-shift-bounds}
For every prime $p$ and every $r\in\N$,
\begin{equation}\label{eq:C-lower-bound}
 \max\left\{0,B^{\circ(p-2)}(r)-p+1\right\}
 \leq C(p,r)
 \leq p\,t_{p-1}(r)-2
 \leq p\,\operatorname{twr}_{p-1}(r)-2.
\end{equation}
\end{proposition}

Since
\[
B(s)=2^{s-\frac12\log_2s+O(1)},
\]
the upper and lower bounds in \cref{prop:kneser-shift-bounds} both have tower height $p-1$ as $r\to\infty$, for every fixed prime $p\geq3$.
For $p=2$, the lower bound is exact: $C(2,r)=r-1$.
The auxiliary constant in \cref{sec:triangle} is $C(r)=3\cdot2^r-2$, which is the upper bound above for $p=3$.
In particular,
\[
\binom{r}{\lfloor r/2\rfloor}-2
\leq C(3,r)
\leq3\cdot2^r-2.
\]
Thus the two bounds differ by a factor of $O(\sqrt r)$, and $\log_2C(3,r)=r+O(\log r)$.

\begin{proof}
Since a poset with $s$ elements has at most $2^s$ downsets, the $p-2$ iterations in the definition of $T_1$ give
\[
 t_{p-1}(r)\leq\operatorname{twr}_{p-1}(r).
\]
Together with \eqref{eq:C-bound}, this proves
\begin{equation}\label{eq:C-tower-bound}
 C(p,r)\leq p\,t_{p-1}(r)-2
 \leq p\,\operatorname{twr}_{p-1}(r)-2.
\end{equation}

For the lower bound, we use the standard antichain coloring behind the arc-graph theorem of Poljak and R\"odl~\cite{PoljakRodl1981}, and include the short argument.
Suppose that $\KSh_s(n,k)$ has a proper $t$-coloring $c$, where $s\geq2$ and $t\leq B(u)$.
Choose distinct sets
\[
 F_1,\ldots,F_t\in\binom{[u]}{\lfloor u/2\rfloor}.
\]
For a vertex $X=(A_1,\ldots,A_s)$ of $\KSh_{s+1}(n,k)$, put
\[
 \alpha=c(A_1,\ldots,A_{s-1}),
 \qquad
 \beta=c(A_2,\ldots,A_s),
\]
and color $X$ by the least element of $F_\beta\setminus F_\alpha$.
The two tuples used to define $\alpha$ and $\beta$ are adjacent in $\KSh_s(n,k)$, so $\alpha\neq\beta$.
Since $F_\alpha$ and $F_\beta$ are distinct sets of the same size, $F_\beta\setminus F_\alpha$ is nonempty.

Suppose that
\[
 X=(A_1,\ldots,A_s)
 \quad\text{and}\quad
 Y=(A_2,\ldots,A_{s+1})
\]
are adjacent in $\KSh_{s+1}(n,k)$.
The color of $X$ belongs to $F_{c(A_2,\ldots,A_s)}$, whereas the color of $Y$ does not belong to this set.
Thus the resulting $u$-coloring is proper.

For $n\geq pk$, starting from $\KSh_2(n,k)=\KG(n,k)$ and using $\chi(\KG(n,k))=n-2k+2$, iterating this construction shows that
\[
 \chi(\KSh_p(n,k))\leq r
 \qquad\text{whenever}\qquad
 n-2k+2\leq B^{\circ(p-2)}(r).
\]
Put $q=B^{\circ(p-2)}(r)$ and take $k=1$ and $n=q$.
If $q<p$, the asserted lower bound is trivial.
If $q\geq p$, this gives a proper $r$-coloring with excess $q-p$, and hence
\[
 C(p,r)\geq q-p+1.
\]
This finishes the proof of \cref{prop:kneser-shift-bounds}.
\end{proof}

\begin{remark}\label{rem:prime-case}
The recursive history construction makes sense for every integer $p\geq2$.
In the present proof, primality is used twice: in \cref{lem:ziegler-tucker}, and in the fact that every nonconstant $p$-tuple has a free orbit under cyclic shift.
The group-theoretic argument below requires only these prime-order cases, since a finite solvable group has a composition series with cyclic factors of prime order.
\end{remark}

\section{Dense simultaneous rotation systems}\label{sec:rotation}

We now give the atom construction from \cref{prop:two-triangle-gadgets} for arbitrary $m$ and prime $p$.
Fix a finite alphabet $X$, a baseline letter $u\in X$, and an integer $p\geq2$.
Given ordered pairwise disjoint blocks $B_1,\ldots,B_p$ and a word which is constant on them, write
\[
 (z_1,\ldots,z_p)_{B_1,\ldots,B_p}
\]
for the operation of placing $z_i$ on $B_i$ and leaving every other coordinate as prescribed by the context.

\begin{theorem}[Dense rotation systems]\label{thm:rotation-system}
Let $p$ be prime.
For every finite alphabet $X$, baseline $u\in X$, $r,m\in\N$, and $\eta>0$, there are $n,k\in\N$ with the following property.
Every coloring $c:X^n\to[r]$ admits pairwise disjoint blocks
\[
  B_{b,1},\ldots,B_{b,p}\qquad (b\in[m]),
\]
all of size $k$, which cover at least $(1-\eta)n$ coordinates.
The coordinates outside these blocks are fixed to $u$.
For every $b\in[m]$, every $z_1,\ldots,z_{p-1}\in X$, and every choice of constant letters on the other selected blocks, changing the pattern on $(B_{b,1},\ldots,B_{b,p})$ from $(z_1,\ldots,z_{p-1},u)$ to $(u,z_1,\ldots,z_{p-1})$ does not change the color:
\begin{equation}\label{eq:rotation}
 c(\ldots;(z_1,\ldots,z_{p-1},u)_{B_{b,1},\ldots,B_{b,p}};\ldots)
 =c(\ldots;(u,z_1,\ldots,z_{p-1})_{B_{b,1},\ldots,B_{b,p}};\ldots).
\end{equation}
\end{theorem}

For $p=3$, this says that the patterns $z_1z_2u$ and $uz_1z_2$ are interchangeable in every context.
Thus it is the general form of the two identities in \cref{prop:two-triangle-gadgets}.

\begin{proof}
It is enough to treat $0<\eta<1$.
We will choose integers $q_b$ and $d_b$ for $b=1,\ldots,m$, where $q_b$ will eventually equal $k/a_b$, the number of $a_b$-atoms in each selected block, and $d_b$ will be the unused excess at that atom scale.
Suppose that $q_j,d_j$ have already been chosen for $j<b$, and set
\[
 h_b=\sum_{j=1}^{b-1}(pq_j+d_j).
\]
The number of possible profiles needed at stage $b$ will be at most
\begin{equation}\label{eq:profile-count}
 r_b=r^{\,|X|^{h_b+p(m-b)+p-1}}.
\end{equation}
First choose
\[
 d_b\geq C(p,r_b),
\]
and then choose $q_b$ sufficiently large that
\begin{equation}\label{eq:excess-ratio}
 \frac{d_b}{q_b}<\frac{p\eta}{1-\eta}.
\end{equation}
When $b\geq2$, we also require $q_b$ to be a multiple of $q_{b-1}$.
Thus the parameters are chosen in the order
\[
 (q_1,d_1),\ldots,(q_{b-1},d_{b-1})
 \ \longrightarrow\ r_b\ \longrightarrow\ d_b\ \longrightarrow\ q_b.
\]

After all these choices, put
\[
 k=q_m,
\qquad
 a_b=\frac{k}{q_b}\quad(b\in[m]).
\]
Since $q_1\mid q_2\mid\cdots\mid q_m$, all the atom sizes $a_b$ are integers, and $a_m=1$.
For each $b\in[m]$, take a disjoint coordinate interval $W_b$ and partition it into $pq_b+d_b$ atoms of size $a_b$.
Thus
\begin{equation}\label{eq:stage-size}
 |W_b|=a_b(pq_b+d_b)=pk+a_bd_b.
\end{equation}
Let $n=|W_1|+\cdots+|W_m|$.

We select the block systems in the reverse order $b=m,m-1,\ldots,1$, just as the second block system was selected before the first in \cref{prop:two-triangle-gadgets}.
At stage $b$, color an ordered $(p-1)$-tuple $(A_1,\ldots,A_{p-1})$ of disjoint $q_b$-sets of the $a_b$-atoms in $W_b$ by its complete color profile.
A profile entry is specified by arbitrary atom-constant values on all the atoms in $W_1,\ldots,W_{b-1}$, arbitrary constant values on the blocks already selected in $W_{b+1},\ldots,W_m$, and arbitrary letters $z_1,\ldots,z_{p-1}\in X$.
For this entry, put $z_i$ on the union of the atoms in $A_i$, put $u$ on every as-yet unspecified coordinate, and evaluate $c$.
There are $h_b$ earlier atoms, $p(m-b)$ already selected blocks, and $p-1$ displayed letters, so the profile has $|X|^{h_b+p(m-b)+p-1}$ entries and there are at most $r_b$ different profiles.

The ground set at this stage has $pq_b+d_b$ atoms, and $d_b\geq C(p,r_b)$.
Therefore \cref{thm:kneser-shift} gives pairwise disjoint $q_b$-sets of atoms $A_1,\ldots,A_p$ such that the profiles of $(A_1,\ldots,A_{p-1})$ and $(A_2,\ldots,A_p)$ agree.
Let $B_{b,i}$ be the union of the atoms in $A_i$.
Each block has physical size $a_bq_b=k$, and equality of the two profiles is precisely \eqref{eq:rotation}.

The identities selected at later stages survive because their profiles recorded every atom-constant assignment on $W_b$.
The new identity is uniform over every atom-constant assignment on $W_1,\ldots,W_{b-1}$, so it will in turn survive all subsequent selections.
Hence all $m$ rotation identities hold simultaneously, and all coordinates outside the resulting $mp$ blocks are fixed to $u$.

Finally, the blocks cover $mpk$ coordinates, while \eqref{eq:stage-size} gives
\[
 n=k\left(mp+\sum_{b=1}^m\frac{d_b}{q_b}\right).
\]
By \eqref{eq:excess-ratio},
\[
 \sum_{b=1}^m\frac{d_b}{q_b}
 <\frac{mp\eta}{1-\eta},
\]
and therefore $mpk/n>1-\eta$.
This completes the proof.
\end{proof}

For $m=2$ and $p=3$, we have $q_1=k/a_1$, while our definition gives $q_2=k$, so the construction first fixes $k/a_1$ and only afterwards makes $k$ large.
If only the fixed patterns $123$ and $312$ from \cref{prop:two-triangle-gadgets} are recorded, the two profile counts reduce to $r^{27}$ and $r^{3^{3q_1+d_1}}$, exactly as in \cref{sec:triangle}.

\begin{remark}\label{rem:quantitative}
Although \cref{prop:kneser-shift-bounds} gives an explicit upper bound on the Kneser-shift threshold, the full rotation-system construction remains quantitatively enormous.
To make the construction explicit, one may take $d_b\geq p\,t_{p-1}(r_b)-2$ at each stage, thereby inserting the profile count $r_b$ from \eqref{eq:profile-count} into the iterated-downset rank bound.
We make no attempt to optimize the resulting values of $n$ or $k$.
\end{remark}

\section{Cyclic actions and solvable extensions}\label{sec:groups}

We now turn the rotation identity into the full cyclic insensitivity needed for a group action; this is done practically the same way as by K\v r\'i\v z \cite{Kriz1991}.
It is useful to allow a temporary equivalence relation.
If $E$ is an equivalence relation on $X$, say that a coloring $f:X^n\to[r]$ is \emph{coordinatewise $E$-insensitive} if $f(x)=f(y)$ whenever $x_iE y_i$ for every $i$.

\subsection{One orbit of a prime cycle}

Let $\tau$ be a permutation of $X$ of prime order $p$, and let
\[
 Y=\{x_1,\ldots,x_p\},\qquad \tau x_i=x_{i+1}
\]
with indices in $[p]$ read cyclically.
The $p$ words
\[
 w_i=(x_i,x_{i+1},\ldots,x_p,x_1,\ldots,x_{i-1})
 \qquad(i\in[p])
\]
are the analogues of the three macro-patterns $123,312,231$ in \cref{sec:triangle}.
For this cyclic action, call a block map of block size $k$ \emph{label-balanced} if $p\mid k$ and, for each variable, each of the labels $\id,\tau,\ldots,\tau^{p-1}$ occurs exactly $k/p$ times on its active coordinates.
For $1\leq t\leq p$, let $E_t$ be the equivalence relation whose classes are
$\{x_1,\ldots,x_t\}$ and the singleton sets $\{x\}$ for
$x\in X\setminus\{x_1,\ldots,x_t\}$.
Thus $E_1$ is equality.

\begin{lemma}[One-orbit fusion]\label{lem:one-orbit}
For every $r,m\in\N$ and $\eta>0$, there are $n,k\in\N$ such that every coloring $c:X^n\to[r]$ admits a $\langle\tau\rangle$-block map $\Phi:X^m\to X^n$ of block size $k$ and active proportion at least $1-\eta$ whose induced coloring is coordinatewise $E_p$-insensitive, where $E_p$ has $Y$ as its only nonsingleton class.
If $X=Y$, then $\Phi$ may in addition be chosen label-balanced.
\end{lemma}

\begin{proof}
We prove by induction on $t$ that, with arbitrarily small loss, the induced coloring can be made coordinatewise $E_t$-insensitive.
For $t=1$, this follows from the identity block map.
Suppose it has been proved for some $t<p$.
Fix $r,m\in\N$ and $\eta>0$, and choose $0<\eta'<1$ so that $(1-\eta')^2\geq1-\eta$.
Use \cref{thm:rotation-system} with $m$ requested gadgets, baseline $x_1$, and loss $\eta'$, and denote the resulting ambient length and block size by $n_0$ and $k_2$.
Apply the induction hypothesis with $n_0$ macrovariables and loss $\eta'$ to choose an ambient length $n$ and a block size $k_1$.
Now fix an arbitrary coloring $c:X^n\to[r]$, and let $\Phi:X^{n_0}\to X^n$ be the resulting block map.
Then
\[
 f=c\circ\Phi:X^{n_0}\to[r]
\]
is coordinatewise $E_t$-insensitive.

Apply \cref{thm:rotation-system} to this induced coloring $f$ and obtain the promised outer blocks.

On each outer $p$-tuple of blocks encode a letter $y\in X$ by
\[
 w(y)=(y,\tau y,\ldots,\tau^{p-1}y).
\]
This defines a $\langle\tau\rangle$-block map $\Psi:X^m\to X^{n_0}$ of block size $pk_2$; its labels on the $p$ equal blocks are $\id,\tau,\ldots,\tau^{p-1}$.
By \cref{lem:block-composition}, $\Phi\circ\Psi$ has active proportion at least $(1-\eta')^2\geq1-\eta$.
Its block size is the fixed integer $pk_1k_2$.

It remains to check $E_{t+1}$-insensitivity.
Write $w_i=w(x_i)$, in agreement with the notation above.
The word $w_2$ ends in the baseline $x_1$, so \eqref{eq:rotation} gives $f(w_2)=f(w_1)$ in every context.
For $2\leq i\leq t$, use $E_t$-insensitivity to replace the first $x_i$ and the tail entries $x_2,\ldots,x_{i-1}$ of $w_i$ by $x_1$.
This produces
\[
 U=(x_1,x_{i+1},\ldots,x_p,
    \underbrace{x_1,\ldots,x_1}_{i-1\text{ entries}}).
\]
Since $U$ begins with $x_1$, its left cyclic shift is
\[
 V=(x_{i+1},\ldots,x_p,
    \underbrace{x_1,\ldots,x_1}_{i\text{ entries}}),
\]
which ends with the baseline $x_1$; applying \eqref{eq:rotation} in the reverse direction gives $f(V)=f(U)$.
Finally, use $E_t$-insensitivity to replace the last $i$ copies of $x_1$ in $V$ by $x_1,x_2,\ldots,x_i$.
The result is $w_{i+1}$.
Hence
\[
  f(w_1)=f(w_2)=\cdots=f(w_{t+1})
\]
in arbitrary contexts, which is precisely coordinatewise $E_{t+1}$-insensitivity after composition.
Indeed, one may make this replacement in any one of the $m$ outer block-tuples while fixing all the others, and then change the coordinates one at a time.

This proves the induction step, and hence the required $E_p$-insensitivity.

For the final assertion, assume $X=Y$ and track the labels in the same induction.
In the first nontrivial step, from $E_1$-insensitivity to $E_2$-insensitivity, the encoding $w(y)$ places each of $\id,\tau,\ldots,\tau^{p-1}$ on one of $p$ equal outer blocks, so the resulting map is label-balanced.
Thereafter label-balance is preserved by composition: every balanced inner label multiset is multiplied by the corresponding outer label in $C_p$ and therefore remains uniform.
Thus the final map is label-balanced.
\end{proof}

For $X=Y=[3]$ and $m=1$, the balanced conclusion partitions the active coordinates into three equal blocks carrying the three cyclic labels.
Indeed, given $\eta>0$, choose $\delta>0$ so that $(1-\delta)^{-1}\leq1+\eta$, write the resulting block size as $k_0=3k$, and denote the resulting length by $n_0$.
Then
\[
 n_0\leq\frac{k_0}{1-\delta}\leq(1+\eta)3k,
\]
so padding by fixed coordinates gives the exact length $n=\lfloor(1+\eta)3k\rfloor$ in \cref{thm:triangle-coord}.
The inactive coordinates form the common fixed word allowed there.

The intermediate relations $E_t$ are not $\tau$-invariant when $1<t<p$.
This causes no problem: they are used only in the explicit $U,V$ calculation inside one orbit.
Composition with other orbit fusions occurs only after all of $Y$ has been merged, when the resulting equivalence relation is $\tau$-invariant.

\begin{lemma}[Relative one-orbit fusion]\label{lem:relative-one-orbit}
Let $E$ be a $\tau$-invariant equivalence relation on $X$ for which every point of $Y$ is an $E$-singleton, and let $E'=E\vee E_p$.
For every $r,m\in\N$ and $\eta>0$ there are $n,k\in\N$ such that every coordinatewise $E$-insensitive coloring $f:X^n\to[r]$ admits a $\langle\tau\rangle$-block map $\Phi:X^m\to X^n$ of block size $k$ and active proportion at least $1-\eta$ for which $f\circ\Phi$ is coordinatewise $E'$-insensitive.
\end{lemma}

\begin{proof}
We repeat the preceding induction, keeping the $E$-insensitivity throughout.
More precisely, we prove by induction on $t$ that, with arbitrarily small loss, the induced coloring can be made coordinatewise $(E\vee E_t)$-insensitive.
Since the points of $Y$ are $E$-singletons, $E\vee E_1=E$, so the case $t=1$ follows from the identity block map.
Suppose the statement has been proved for some $t<p$.
Fix $r,m\in\N$ and $\eta>0$, and choose $0<\eta'<1$ so that $(1-\eta')^2\geq1-\eta$.
Use \cref{thm:rotation-system} with baseline $x_1$ to obtain $n_0,k_2$ for $m$ gadgets and loss $\eta'$.
Next use the induction hypothesis with $n_0$ macrovariables and loss $\eta'$ to obtain $n,k_1$.
Given a coordinatewise $E$-insensitive coloring $f:X^n\to[r]$, choose $\Phi:X^{n_0}\to X^n$ so that $g=f\circ\Phi$ is coordinatewise $(E\vee E_t)$-insensitive, and apply the rotation theorem to $g$.
On each resulting $p$-tuple of blocks encode $y\in X$ by
\[
 w(y)=(y,\tau y,\ldots,\tau^{p-1}y),
\]
obtaining a $\langle\tau\rangle$-block map $\Psi:X^m\to X^{n_0}$ of block size $pk_2$.
The $U,V$ calculation in the proof of \cref{lem:one-orbit}, using the $E_t$-insensitivity of $g$, shows that $g\circ\Psi$ is coordinatewise $E_{t+1}$-insensitive.
Moreover, every active coordinate of $\Psi$ has the form $y\mapsto\tau^j y$, so the $E$-insensitivity of $g$ survives because $E$ is $\tau$-invariant; the inactive coordinates agree identically.
Changing one coordinate along an alternating $E$-/$E_{t+1}$-chain shows
that $g\circ\Psi$ is therefore coordinatewise
$(E\vee E_{t+1})$-insensitive.
By \cref{lem:block-composition}, $\Phi\circ\Psi$ has block size $pk_1k_2$ and active proportion at least $(1-\eta')^2\geq1-\eta$.
This proves the induction step.
For $t=p$, we have $E\vee E_p=E'$, proving the lemma.
\end{proof}

\begin{proposition}[Prime cyclic actions]\label{prop:cyclic-action}
If $C_p$ is cyclic of prime order and acts on a finite set $X$, then $\mathrm{DB}(C_p\curvearrowright X)$ holds.
\end{proposition}

\begin{proof}
Fix a generator $\tau$ of $C_p$.
List the nontrivial $C_p$-orbits as $Y_1,\ldots,Y_s$, and let $F_j$ be the equivalence relation whose nonsingleton classes are $Y_1,\ldots,Y_j$.
If $s=0$, the identity block map proves the assertion.
Assume henceforth that $s\geq1$.
We prove by induction on $j$ that the dense block property holds with coordinatewise $F_j$-insensitivity.
The case $j=1$ is \cref{lem:one-orbit} applied to $Y_1$.

Suppose the statement has been proved for $j-1$, where $2\leq j\leq s$.
Fix $r,m\in\N$ and $0<\eta<1$, and choose $0<\eta'<1$ so that $(1-\eta')^2\geq1-\eta$.
First choose the macro-length $n_0$ and block size $k_2$ supplied by \cref{lem:relative-one-orbit} for $Y_j$, $E=F_{j-1}$, $r$ colors, $m$ variables, and loss $\eta'$.
Use the induction hypothesis with $n_0$ variables and loss $\eta'$ to choose an ambient length $n$ and block size $k_1$.
Fix an arbitrary coloring $c:X^n\to[r]$, and let $\Phi:X^{n_0}\to X^n$ be the resulting block map.
The induced coloring $c\circ\Phi$ is coordinatewise $F_{j-1}$-insensitive, so \cref{lem:relative-one-orbit} supplies a second block map whose induced coloring is coordinatewise $F_j$-insensitive.
The relation $F_{j-1}$ is $C_p$-invariant, and the points of $Y_j$ are its singletons, as required by that lemma.
The active proportion of the composite is at least $(1-\eta')^2\geq1-\eta$ by \cref{lem:block-composition}.
Its block size is the fixed integer $k_1k_2$.
This proves the induction step.

Finally, $F_s$ is exactly the equivalence relation of lying in the same $C_p$-orbit (fixed points remain singleton classes).
Thus $F_s$-insensitivity is precisely \eqref{eq:orbit-insensitivity}, proving the proposition.
\end{proof}

\subsection{A prime cyclic extension}

\begin{proposition}[Extension step]\label{prop:extension}
Let $H\triangleleft G$ be finite groups with $G/H\cong C_p$ for a prime $p$.
Let $G$ act on a finite set $X$, and restrict this action to $H$.
If $\mathrm{DB}(H\curvearrowright X)$ holds, then $\mathrm{DB}(G\curvearrowright X)$ holds.
\end{proposition}

\begin{proof}
Let $\bar X=X/H$ be the set of $H$-orbits.
Normality makes
\[
  (gH)(Hx)=H(gx)
\]
a well-defined action of $G/H$ on $\bar X$.
Fix $r,m\in\N$ and $0<\eta<1$, and choose $0<\eta'<1$ with $(1-\eta')^2\geq1-\eta$.

By \cref{prop:cyclic-action}, choose a macro-length $n_0$ and block size $k_2$ such that every coloring of $\bar X^{n_0}$ has a $(G/H)$-block map $\Psi:\bar X^m\to\bar X^{n_0}$ of block size $k_2$, active proportion at least $1-\eta'$, and orbit-insensitive induced coloring.
Next use $\mathrm{DB}(H\curvearrowright X)$ with $n_0$ macrovariables and loss $\eta'$ to choose an ambient length $n$ and block size $k_1$.
Fix an arbitrary coloring $c:X^n\to[r]$.
This gives an $H$-block map $\Phi_H:X^{n_0}\to X^n$ of block size $k_1$ such that
\[
 \bar c(Hx_1,\ldots,Hx_{n_0})
 =c(\Phi_H(x_1,\ldots,x_{n_0}))
\]
is well-defined on $\bar X^{n_0}$.

Apply the chosen cyclic conclusion to $\bar c$ and obtain $\Psi$ of block size $k_2$.
Let $J_1,\ldots,J_m\subseteq[n_0]$ be its active blocks.
For $i\in J_j$, lift the quotient label $g_iH$ to a representative $g_i\in G$.
If $Z_i\in\bar X$ is the fixed orbit at an inactive coordinate, choose $z_i\in X$ with $Hz_i=Z_i$.
Define
\[
 \widetilde\Psi(x_1,\ldots,x_m)_i=
 \begin{cases}
   g_i x_j,&i\in J_j,\\
   z_i,&i\notin J_1\cup\cdots\cup J_m.
 \end{cases}
\]
This is a $G$-block map lifting $\Psi$.
If a physical coordinate of $\Phi_H$ has label $h\in H$ and lies above an active macrocoordinate $i\in J_j$, then its label in $\Phi_H\circ\widetilde\Psi$ is $hg_i\in G$; the order is dictated by our left-action convention.
Hence the composite is a $G$-block map of block size $k_1k_2$, and
\[
 \frac{m k_1k_2}{n}
 =\frac{n_0 k_1}{n}\frac{m k_2}{n_0}
 \geq(1-\eta')^2\geq1-\eta.
\]
If $y_j\in Gx_j$, then $Hy_j$ and $Hx_j$ lie in the same $(G/H)$-orbit in $\bar X$.
For every $x=(x_1,\ldots,x_m)$, the definitions and the representative-independence of $\bar c$ give
\[
 c\bigl((\Phi_H\circ\widetilde\Psi)(x)\bigr)
 =\bar c\bigl(\Psi(Hx_1,\ldots,Hx_m)\bigr).
\]
The orbit-insensitivity of the induced coloring $\bar c\circ\Psi$ now gives \eqref{eq:orbit-insensitivity} for the composite.
\end{proof}

\begin{proof}[Proof of \cref{thm:DB-solvable}]
A finite solvable group has a composition series
\[
  \{e\}=G_1\triangleleft G_2\triangleleft\cdots\triangleleft G_t=G
\]
with $G_i/G_{i-1}$ cyclic of prime order for every $2\leq i\leq t$.
For the trivial group, the identity block map with $n=m$ and $k=1$ proves the dense block property.
Starting from $G_1$ and applying \cref{prop:extension} successively for $i=2,\ldots,t$ proves $\mathrm{DB}(G\curvearrowright X)$.
\end{proof}

\begin{proof}[Proof of \cref{thm:main}]
Let $\Gamma$ be a solvable group of isometries acting transitively on $P$, and let $G$ be its permutation image on $P$.
As a homomorphic image of $\Gamma$ in $\operatorname{Sym}(P)$, the group $G$ is finite, solvable, and transitive on $P$.
Every permutation in $G$ preserves the pairwise distances in $P$ and therefore extends uniquely to an isometry of $\operatorname{aff}P$; uniqueness makes these extensions an action of $G$.
Apply \cref{thm:DB-solvable} to $G\curvearrowright P$, and then apply \cref{prop:geometry}.
\end{proof}

\section{Concluding remarks}\label{sec:conclusion}

Several questions remain open; I mention a few that fascinate me most.

\begin{enumerate}[label=\textup{(\arabic*)},leftmargin=*]

\item 
Does \cref{thm:main} remain true without the solvability assumption?
Since every subset of a Ramsey set is Ramsey, an affirmative answer would prove the ``if'' direction of the conjecture of Leader, Russell, and Walters~\cite{LeaderRussellWalters2012} that the Ramsey sets are precisely the subtransitive sets.
A first step could be to prove the same conclusion for every finite spherical set admitting a solvable group of isometries with at most two orbits, in parallel with K\v r\'i\v z's ordinary Ramsey theorem~\cite[Theorem~4.4]{Kriz1991}.
A stronger combinatorial question is whether $\mathrm{DB}(G\curvearrowright X)$ holds for every finite group action.

\item
Is every solvable subtransitive set ncs-Ramsey?
A stronger, purely geometric statement would imply this: if a finite set $P$ of circumradius $\rho$ embeds in a solvable transitive set, must it, for every $\eps>0$, embed in a solvable transitive set of circumradius less than $\rho+\eps$?
One may ask the same two questions without the word \emph{solvable}.
We mention that Moore~\cite{Moore2026} and, independently, Mirabi~\cite{Mirabi2026} recently proved the related ordinary Ramsey statement that adjoining any point outside the affine hull of a Ramsey set preserves Ramseyness, but their methods do not seem to imply anything about our question.

\item 
Does \cref{thm:kneser-shift} remain valid for composite $p$?
The prime-factor reduction used for ordinary Kneser hypergraphs does not apply directly here, because the colors live on ordered $(p-1)$-tuples and the desired equality concerns a shift by one set.
Moreover, for composite $p$ the cyclic action on nonconstant history vectors need not be free: for example, $(E,F,E,F)$ is fixed by a half-turn when $p=4$.
Thus the equivariant sign assignment used in the present proof genuinely breaks down.

\item
Determine, or estimate, $\chi(\KSh_p(n,k))$.
In particular, what is the order of growth of the optimal threshold $C(p,r)$ as $r\to\infty$ for fixed prime $p\geq3$?
The bounds in \cref{prop:kneser-shift-bounds} have the same tower height.
For $p=3$, they determine $C(3,r)$ within a factor of $O(\sqrt r)$, but its precise asymptotic order remains open.

\item 
Ivan, Leader, and Walters~\cite{IvanLeaderWalters2026} conjecture that for every fixed template the block degree can be chosen independently of the number of colors, while the ambient word length may still grow with it.
Our density requirement necessarily forces the block size to grow with the number of colors.
Indeed, let $G$ act transitively on $X$, put $q=|X|\geq2$ and $\ell=\lfloor\log_q r\rfloor$, and color a word by its first $\min\{\ell,n\}$ coordinates.
This coloring uses at most $r$ colors; injectivity rules out $n\leq\ell$, while for $n>\ell$ orbit-insensitivity forces all first $\ell$ coordinates to be unused.
Consequently
\[
 \frac{mk}{n}\leq\frac{mk}{mk+\ell},
\]
so an active proportion of at least $1-\eta$ requires
\[
 k\geq\frac{1-\eta}{m\eta}\lfloor\log_q r\rfloor.
\]
How close is this elementary logarithmic lower bound to the truth?
More generally, what is the optimal tradeoff between the block size, the number of colors, and the unused proportion in \cref{thm:DB-solvable}?

\item 
Is there a density version of \cref{thm:main}?
More precisely, if $P$ is solvable transitive with circumradius $\rho$, is it true that for every $\eps,\delta>0$ there are an $n$ and a finite set $X\subseteq\S_{\rho+\eps}^n$ such that every $Y\subseteq X$ with $|Y|\geq\delta|X|$ contains a congruent copy of $P$?
Frankl and R\"odl~\cite{FranklRodl1990} proved the corresponding finite-witness density statement for simplices when $X$ may be an arbitrary finite subset of Euclidean space.
The distinction between ordinary Ramsey and finite-density properties is studied further by Reiher, R\"odl, and Sales~\cite{ReiherRodlSales2024} and by R\"odl and Sales~\cite{RodlSales2024}.
A related but formally different version replaces the finite set $X$ by the whole sphere and relative cardinality by normalized surface measure.
Guruswami and Li~\cite{GuruswamiLi2025} recently proved results of this measurable spherical type for a broad class of inductive configurations, including regular simplices.
The present Kneser-shift input does not directly yield the desired finite-witness statement: for every $p\geq3$, the graph $\KSh_p(n,k)$ has an independent set containing at least one quarter of its vertices.
Indeed, randomly $2$-color $\binom{[n]}k$ and retain the tuples $(A_1,\ldots,A_{p-1})$ for which $A_1$ has color $0$ and $A_2$ has color $1$; the expected size of the retained set is one quarter of all vertices, and it is independent because a retained tuple and its shifted successor would require $A_2$ to have both colors.

\item 
Is there a canonical version of \cref{thm:main}?
More precisely, if $P$ has circumradius $\rho$ and admits a solvable transitive group of isometries, is it true that for every $\eps>0$ there is an $n$ such that every coloring of $\S_{\rho+\eps}^n$, with an arbitrary palette, contains either a monochromatic or a rainbow copy of $P$?
Our proof methods seem useless for this version.
Recent progress in canonical Euclidean Ramsey theory includes acute triangles and hypercubes~\cite{GeherSagdeevToth2025}, all triangles and rectangles~\cite{FangEtAl2025}, all simplices~\cite{GeShuXuYu2026}, all products of simplices~\cite{Shaw2026}, and regular polygons of prime order~\cite{ShawPrime2026}.
These results do not impose the near-circumradius condition above.

\end{enumerate}

\section*{Acknowledgments}

I would like to thank Arsenii Sagdeev and Géza Tóth for several useful preliminary discussions, and Imre B\'ar\'any for pointing me to Szemerédi's result in \cite{ErdosSimonovits1973}. 

\section*{Statement on the use of artificial intelligence}

See at the end of \cref{sec:introduction} (so click on this: \cref{sec:triangle}, and then scroll up a bit, because I didn't bother to add a label to that part).
As discussed there, the rest of the paper is entirely by ChatGPT.

\phantomsection
\addcontentsline{toc}{section}{References}
\begingroup
\footnotesize
\sloppy

\endgroup

\clearpage
\normalsize
\fussy
\appendix
\phantomsection
\begin{center}
{\LARGE\bfseries Appendix\par}
\vspace{0.8em}
{\large Further results claimed by ChatGPT\par}
\end{center}
\addcontentsline{toc}{section}{Appendix: Further results claimed by ChatGPT}
\label[chatgptappendix]{app:chatgpt-results}

\paragraph{Principal results.}
The results below are grouped by topic and ordered roughly by mathematical
interest and importance.  Each reference points to the statement and proof
of the result.

\begin{enumerate}[leftmargin=*,label=\textup{\arabic*.}]
\item \textbf{Fixed-degree block sets for word templates.}
The Block Sets Conjecture for the template $112233$, posed as Problem~G by
Leader, Russell, and Walters~\cite[Problem~G]{LeaderRussellWalters2012}, and
its later color-independent fixed-degree strengthening, conjectured by Ivan,
Leader, and Walters~\cite[Conjecture~4]{IvanLeaderWalters2026}, both hold: a
balanced literal role map of degree $120$ into $1344$ coordinates places all
its permutations in one orbit of a solvable coordinate-permutation group
(Theorem~\ref{thm:appendix-112233}).
An affine-plane construction gives the same conclusion for every template
$1234^k$, $1\leq k\leq6$
(Proposition~\ref{prop:appendix-affine-template-family}).
Combining its $12344$ case with a new pair-splitting gadget proves the
fixed-degree conjecture for $12345$, with a balanced literal role map of
degree $35$ into $504$ coordinates
(Theorem~\ref{thm:appendix-12345}).
Consequently the conjecture holds for every template of length at most five
(Corollary~\ref{cor:appendix-all-templates-length-five}), and also for every
template on at most three symbols and of length at most six
(Corollary~\ref{cor:appendix-short-three-letter-templates}).
A flag-divisibility obstruction shows that a balanced literal certificate
for $123456$ contained in one solvable coordinate orbit must use fixed
padding, and that this orbit has at least $3960$ states
(Proposition~\ref{prop:appendix-literal-flag-divisibility} and
Corollary~\ref{cor:appendix-pure-projection-obstruction}).

\item \textbf{Exact-radius solvable extensions of regular polytopes.}
Every regular convex polytope embeds in a finite solvable transitive
configuration with exactly the same circumradius, and is therefore
ncs-Ramsey (Corollary~\ref{cor:app-all-regular-polytopes}).
The exceptional ingredients are a solvable regular action on the $600$-cell
(Proposition~\ref{prop:app-600-cell}) and an explicit exact-radius solvable
extension of the $120$-cell in $\mathbb R^{600}$
(Theorem~\ref{thm:app-120-cell}); the dodecahedron is handled by a separate
five-coordinate orbit design (Proposition~\ref{prop:app-dodecahedron}).

\item \textbf{The two-orbit problem.}
Setwise-stabilizer compression proves ncs-Ramseyness whenever the compressed
translate degree is at most nine, as well as in the stated weight cases of
degree ten (Theorem~\ref{thm:setwise-stabilizer-compression}, based on
Theorem~\ref{thm:small-johnson-layers}).
Subgroup-chain substitution extends this to imprimitive compressed actions
built from resolved local factors (Theorem~\ref{thm:subgroup-chain-substitution}).
For primitive compressed actions, the entire affine O'Nan--Scott case is
eliminated (Corollary~\ref{cor:no-affine-primitive-counterexample}), and the
remaining nonabelian-socle candidates obey a faithful-character and
divisibility certificate (Proposition~\ref{prop:faithful-common-constituent}).

\item \textbf{Kneser--shift graphs.}
For every uniformity $p$, whether or not $p$ is prime, the exact two-color
threshold is $C(p,2)=0$ for odd $p$ and $C(p,2)=1$ for even $p$
(Proposition~\ref{prop:exact-two-color-kneser-shift}); an alternating-union
homomorphism also gives the general even-$p$ lower bound $C(p,r)\geq r-1$
(Proposition~\ref{prop:alternating-union-projection}).
For prime $p$, the upper bound improves to
$C(p,r)\leq(p-1)t_{p-1}(r)-1$
(Proposition~\ref{prop:appendix-sharper-ksh-bound}), while for $p=3$ a
profile coloring gives the improved lower bound
$C(3,r)\geq2\binom{r-1}{\lfloor(r-1)/2\rfloor}-2$
(Proposition~\ref{prop:appendix-sharper-ksh-lower-bound}).

\item \textbf{Palette-free canonical consequences.}
A dense rotation theorem applied to color equality gives a palette-free
dense kernel-rotation lemma (Lemma~\ref{lem:dense-kernel-rotation}) and,
for cyclic alphabets of prime order, dense one-edge interchangeability
(Corollary~\ref{cor:dense-one-edge-interchangeability}).
The obstruction to iterating this argument is made exact by a reduction to
monochromatic triangles in Kneser graphs
(Proposition~\ref{prop:exact-kneser-obstruction}) and by a five-point profile
coloring of $KG(3k+2,k)$ (Theorem~\ref{thm:five-point-profile}).
In a different direction, Shaw's canonical theorem~\cite{ShawPrime2026}
extends from regular
prime polygons to every transitive Euclidean configuration of prime
cardinality and to all of its powers
(Proposition~\ref{prop:prime-cardinality-canonical}).

\item \textbf{A sharp obstruction for full simplex products.}
For every constant-weight layer $J(v,k)$ with $2\leq k\leq v-2$, the optimal
radius of an embedding in an arbitrary full Cartesian product of regular
simplices is computed exactly; except in the balanced case, this leaves a
strict radius gap and shows why the restricted products used above are essential
(Theorem~\ref{thm:appendix-simplex-product-optimum}).

\item \textbf{The Leader--Russell--Walters kite.}
For powers of a fixed spherical alphabet, the optimal achievable radius is
computed exactly and is either already one or uniformly bounded away from one
(Proposition~\ref{prop:lrw-fixed-alphabet}).
The natural six-letter grids admit explicit kite-free colorings at every
bounded radius (Proposition~\ref{prop:lrw-grid-coloring}), and every
cyclic-path generalized-prism construction has the same unavoidable
$\Omega(1/|\lambda|)$ radius divergence
(Proposition~\ref{prop:lrw-cyclic-path} and
Corollary~\ref{cor:lrw-cyclic-divergence}).
Finally, infinite chromatic number of a concrete projective hypergraph is a
sufficient condition for the kite to be ncs-Ramsey
(Proposition~\ref{prop:lrw-projective-reduction}), while its coherently
orientable finite subhypergraphs are all $2$-colorable
(Proposition~\ref{prop:lrw-coherent-coloring}).
\end{enumerate}

\paragraph{How this appendix was produced.}
Every statement and proof in this appendix was generated by OpenAI's ChatGPT
in research conversations with the author after the first arXiv version of
this paper.  They are presented explicitly as mathematical claims of
ChatGPT.  The author supplied the manuscript, its open problems, and
corrections to unsuccessful approaches.  ChatGPT then ran several parallel
proof searches, compared the surviving arguments, subjected each substantive
section to a separate adversarial proof audit, and checked the finite group
and incidence certificates by exhaustive verification programs.  Arguments
for which a gap or counterexample was found were discarded rather than
reported conditionally.  The author has nevertheless not independently
verified every claim, so this appendix should not be treated as transferring
mathematical responsibility for these results to the author before such
verification has taken place.  None of the remaining unrestricted questions
is declared solved.  In particular, the unrestricted two-orbit problem, the
Leader--Russell--Walters kite problem, and the Block Sets Conjecture for the
template $123456$ remain open.

\section{Orbit designs and the two-orbit problem}
\label{sec:appendix-two-orbit}

\subsection{Compressed two-orbit supports and small constant-weight layers}
\label{subsec:compressed-small-layers}

We identify a binary word of length $t$ with its support, and write
\[
 J(t,\lambda)=\{A\subseteq[t]:|A|=\lambda\}.
\]
We first record the closure notion used below.

\begin{definition}[iterated solvable coordinate scheme]
\label{def:iterated-coordinate-scheme}
We define the class of finite word families having an \emph{iterated
solvable coordinate scheme} by structural recursion.

At a direct step, a family $\mathcal W\subseteq A^s$ has a scheme if it is
contained in a supertemplate $\mathcal C\subseteq A^s$ which is one orbit of
a solvable group $L\leq S_s$.

At a recursive step, choose an $L$-invariant supertemplate
$\mathcal C\supseteq\mathcal W$, where $L\leq S_s$ is solvable, and
permutations $\pi_1,\ldots,\pi_q\in S_s$ preserving $\mathcal C$.  Put
\[
 \Lambda(w)=\bigl(L\pi_1w,\ldots,L\pi_qw\bigr)\in(L\backslash\mathcal C)^q.
\]
If $\Lambda(\mathcal W)$ has a scheme, then so does $\mathcal W$.

At a substitution step, let $\mathcal A\subseteq2^X$ be a nonempty
constant-weight family and let $\mathcal B\subseteq2^Y$ be a
constant-weight family of positive weight, both having schemes.  Write
$\mathcal A[\mathcal B]\subseteq2^{X\times Y}$ for the family whose set of
nonempty $Y$-fibres belongs to $\mathcal A$, whose nonempty fibres belong to
$\mathcal B$, and whose remaining fibres are empty.  Every subfamily of
$\mathcal A[\mathcal B]$ has a scheme.

Finally, the property is preserved by a bijective relabelling of the
alphabet.  Equivalently, the families having schemes form the least class
closed under these four rules.
\end{definition}

The requirement $\pi_i\mathcal C=\mathcal C$ is what makes the recursion geometric rather than merely formal.
It ensures that every $\pi_i$ lifts to an orthogonal coordinate permutation of the same restricted product.

\begin{lemma}[iterated orbit design]
\label{lem:iterated-orbit-design}
Let $P\subseteq\rho\S(V)$ be a finite spherical configuration whose intrinsic circumcenter is the origin and whose intrinsic circumradius is $\rho$, and let a solvable group $G\leq\operatorname{Iso}(P)$ have orbits $O_1,\ldots,O_u$.
Suppose that $h_1,\ldots,h_s\in\operatorname{Iso}(P)$ and that the orbit-type family
\[
 \mathcal W=\left\{\bigl(j(h_1x),\ldots,j(h_sx)\bigr):x\in P\right\}\subseteq[u]^s,
 \qquad h_ix\in O_{j(h_ix)},
\]
has an iterated solvable coordinate scheme.
Then $P$ embeds isometrically in a finite solvable transitive configuration of circumradius exactly $\rho$.
Consequently $P$ is ncs-Ramsey.
\end{lemma}

\begin{proof}
We argue by structural induction on a scheme certificate, using a relative
induction invariant.  Simultaneously, we prove the following pointed form of
the assertion: if $\mathcal U\subseteq D^m$ has a scheme and
$a\in D$ with $a^m\notin\mathcal U$, then its certificate can be run
on $\mathcal U\cup\{a^m\}$ with $a^m$ retained as a separate orbit throughout
and with $\mathcal U$ fused as prescribed.  At later stages the distinguished
word may have a different alphabet letter, but it remains
coordinate-constant.

At any stage we have an isometric embedding
\[
 e:P\hookrightarrow Z\subseteq\rho\S(W),
\]
a solvable group $A$ preserving $Z$, and a list of isometries of $Z$ whose orbit-type words are governed by the remainder of the coordinate scheme.
Initially $Z=P$, $A=G$, and the isometries are the given $h_i$.

For a direct step whose current word length is $m$ and whose supertemplate $\mathcal C$ is transitive, form
\[
 Y=\left\{m^{-1/2}(y_1,\ldots,y_m):y_i\in Z,\ (Ay_1,\ldots,Ay_m)\in\mathcal C\right\}.
\]
The appropriate wreath product $A^m\rtimes L$ is solvable and transitive on $Y$, while the diagonal map obtained from the current isometries embeds $P$ isometrically.
Indeed, for the initial stage this map is
\[
 x\longmapsto s^{-1/2}(h_1x,\ldots,h_sx)
\]
and the same calculation applies at every later stage.
Every point of $Y$ has norm $\rho$, so its circumradius is at most $\rho$, and the contained copy of $P$ makes it at least $\rho$.
For the pointed form, replace $\mathcal C$ by
$\mathcal C\cup\{a^m\}$.  Coordinate permutations fix $a^m$; moreover, the
actual product points having this constant orbit type form one orbit under
the base direct power $A^m$.  Thus they remain a single orbit separate from
the $\mathcal C$-orbit.

At a recursive step, use the full invariant supertemplate $\mathcal C$ in the same construction, but do not yet require transitivity.
The $A^s\rtimes L$-orbits of the resulting restricted product are exactly the $L$-orbits on $\mathcal C$.
Each $\pi_i$ induces an isometry of that restricted product because it preserves $\mathcal C$, so these induced maps and the new restricted product satisfy the same relative invariant for the labelled orbit words.
For the pointed form, use $\mathcal C\cup\{a^s\}$; both $L$ and every
$\pi_i$ fix $a^s$, and its labelled-orbit image is again
coordinate-constant.  Thus the pointed induction hypothesis applies as well.
The induction hypothesis applied to $\Lambda(\mathcal W)$ completes the
recursive case.

At a substitution step, index the current maps by $X\times Y$ and suppose
that their type family is contained in $\mathcal A[\mathcal B]$.
Run identical copies of the construction for $\mathcal B$ independently in
the $Y$-fibres over all $x\in X$, retaining the all-zero fibre as a separate
orbit at every stage.  This is possible because coordinate permutations fix
the zero word, while the positive weight of $\mathcal B$ prevents an active
fibre from entering its orbit.  Once the inner constructions terminate, the
two relevant fibre types are the empty fibre and the single fused positive
type.  Apply the construction for $\mathcal A$ to the resulting word of
fibre types, using its pointed form whenever the whole substitution occurs
inside a still larger inactive fibre.  All acting groups are subgroups of
iterated wreath products of
the solvable groups in the two certificates and are therefore solvable.
Associativity of normalized Cartesian products preserves both the isometric
embedding and the exact radius.  Alphabet relabelling merely renames orbit
types, transports the distinguished constant word, and requires no
construction.

Structural induction now proves the first assertion.
The final host is solvable transitive and has the same circumradius as $P$, so the main theorem applied to the host proves the last assertion.
\end{proof}

In particular, binary complementation preserves the existence of a scheme.

\begin{lemma}[even pair layers]
\label{lem:even-pair-layers}
If $q$ is a prime power, then $J(2q,2)$ has an iterated solvable coordinate scheme.
\end{lemma}

\begin{proof}
Partition $[2q]$ into $q$ distinguished pairs and let
\[
 L=C_2^q\rtimes\operatorname{AGL}(1,q)
\]
act by independent flips inside the pairs and by affine permutations of the pairs.
The group $L$ has two orbits on $J(2q,2)$: the distinguished matching and its complement.
Choose a one-factorization
\[
 E(K_{2q})=F_1\sqcup\cdots\sqcup F_{2q-1}
\]
and a permutation $\pi_i$ taking $F_i$ to the distinguished matching.
For every edge $e$, the word recording whether $\pi_i(e)$ is in the distinguished matching has weight one, because $e$ lies in exactly one factor $F_i$.
The outer supertemplate is therefore $J(2q-1,1)$, on which $C_{2q-1}$ is transitive.
All $\pi_i$ preserve the full inner supertemplate $J(2q,2)$, so this is a valid recursive step.
\end{proof}

The next theorem is the finite ingredient needed for the compressed two-orbit result.

\begin{theorem}[small constant-weight layers]
\label{thm:small-johnson-layers}
Every full layer $J(t,\lambda)$ has an iterated solvable coordinate scheme when $t\leq9$.
For $t=10$, the same is true for
\[
 \lambda\in\{0,1,2,3,7,8,9,10\}.
\]
\end{theorem}

\begin{proof}
The layers of weights zero and $t$ are trivial, and complementation lets us restrict to $\lambda\leq t/2$.
For every $t$, the layer $J(t,1)$ is a single orbit of the cyclic group $C_t$; complementation handles $J(t,t-1)$.
For $t\leq4$, the solvable group $S_t$ is transitive on every layer.
The affine group $\operatorname{AGL}(1,q)$ is solvable and two-transitive on $\mathbb F_q$, so it handles weights one and two, and their complements, when $q$ is a prime power.
It is also transitive on $J(8,3)$.
Indeed, the setwise stabilizer of a triple injects into $S_3$, since an
affine map fixing two distinct points is the identity.
It cannot contain an element inducing a three-cycle because $3\nmid56$.
It cannot contain an element inducing a transposition either: such an
element would be an affine involution fixing the remaining point, whereas
every nonidentity involution in $\operatorname{AGL}(1,8)$ is a nonzero
translation and hence is fixed-point-free.
Thus the stabilizer is trivial, and the orbit has size $56=\binom83$.
Thus only the seven recursive constructions in \cref{tab:small-layer-certificates} need verification.

\begin{table}[ht]
\centering
\small
\begin{tabular}{@{}c p{.49\textwidth} p{.27\textwidth}@{}}
\hline
Layer & Inner orbit partition and translate certificate & Outer supertemplate \\
\hline
$J(6,2)$, $J(10,2)$ & $2q$ points in $q=3,5$ distinguished pairs; a one-factorization of $K_{2q}$ & $J(2q-1,1)$ \\
$J(6,3)$ & $S_4\times C_2$ on the edges of $K_4$; five translates of the path orbit & $J(5,3)$ \\
$J(7,3)$ & $\operatorname{AGL}(1,7)$; five translates of its $14$-set orbit & $J(5,2)$ \\
$J(8,4)$ & $\operatorname{AGL}(1,8)$; ten translates of its affine-plane orbit & subfamily of $J(10,2)$ \\
$J(9,3)$ & $\operatorname{AGL}(2,3)$; seven disjoint affine-plane line systems & $J(7,1)$ \\
$J(9,4)$ & the same seven line systems; four distinct constituent lines & $J(7,3)$, after complementing \\
$J(10,3)$ & five distinguished pairs and the nine factors of $K_{10}$ & $J(9,3)$ \\
\hline
\end{tabular}
\caption{Recursive certificates for the exceptional small layers.}
\label{tab:small-layer-certificates}
\end{table}

For $J(6,2)$ and $J(10,2)$, use \cref{lem:even-pair-layers} with $q=3$ and $q=5$, respectively.
The remaining certificates are given explicitly below.

For $J(6,3)$, identify the six points with the edges
\[
 01,02,03,12,13,23
\]
of $K_4$, in this order.
The action of $S_4$ on the edges has three orbits on triples: the four stars, the four triangles, and the twelve three-edge paths.
Adjoining the central permutation which exchanges each edge with its opposite gives a solvable group $S_4\times C_2$ with two orbits, namely the paths and their complement.
Let $g=(1\ 2\ 3\ 4\ 5)$ fix $0$ in the displayed numerical labelling.
For the four representatives $012,013,123,124$ of the $\langle g\rangle$-orbits on triples, the sets of $i\in\Z_5$ for which the representative lies in $g^iO$, where $O$ is the path orbit, are
\[
 \{1,2,4\},\qquad \{1,2,3\},\qquad \{0,2,4\},\qquad \{0,3,4\}.
\]
Every inner orbit word therefore lies in $J(5,3)$, which is one orbit of the solvable group $\operatorname{AGL}(1,5)$.

For $J(7,3)$, label the ground set by $\mathbb F_7$ and let $O$ be the $\operatorname{AGL}(1,7)$-orbit of $\{0,1,3\}$:
\[
\begin{split}
O=\{&013,015,023,026,045,046,124,126,134,156,\\
    &235,245,346,356\}.
\end{split}
\]
Put $g=(0\ 1\ 2\ 4\ 3)$, fixing $5$ and $6$.
The affine group has exactly two orbits on $J(7,3)$: the displayed orbit $O$ of size $14$ and the orbit of $012$ of size $21$ (the latter triple has setwise stabilizer of order two).
The two intersection certificates
\[
 O\cap gO=\{023,045\},\qquad
 O\cap g^2O=\{026,124,156,245,346\}
\]
have five pairwise disjoint $\langle g\rangle$-translates each, and these ten sets partition all $35$ triples.
Thus each triple belongs to exactly two of $O,gO,\ldots,g^4O$, and the outer word lies in $J(5,2)$.

For $J(8,4)$, write $\mathbb F_8=\mathbb F_2(\alpha)$ with
$\alpha^3+\alpha+1=0$, and label
\[
 0,1,2,3,4,5,6,7
 =0,1,\alpha,1+\alpha,\alpha^2,1+\alpha^2,
   \alpha+\alpha^2,1+\alpha+\alpha^2.
\]
Take the affine-plane orbit
\[
\begin{split}
O=\{&0123,0145,0167,0246,0257,0347,0356,\\
    &1247,1256,1346,1357,2345,2367,4567\}.
\end{split}
\]
The affine group has exactly two orbits on $J(8,4)$: the displayed affine-plane orbit of size $14$ and an orbit of size $56$.
For example, in the indicated field labelling the stabilizer of $0124$ is trivial, so the second orbit is regular.
For $g=(1\ 2)(3\ 4\ 5\ 6\ 7)$, fixing $0$, direct expansion of the displayed list gives
\[
 \{\{i:g^{-i}A\in O\}:A\in J(8,4)\}
 =\bigl\{\{i,j\}:i\ne j,\ i\equiv j\pmod2\bigr\}
  \cup\bigl\{\{i,i+5\}:i=0,\ldots,4\bigr\}.
\]
Here the indices are in $\Z_{10}$, the five second-kind supports occur six times each, and the twenty first-kind supports occur twice each, accounting for all $70$ four-sets.
This is a subfamily of the full outer template $J(10,2)$, which has a scheme by \cref{lem:even-pair-layers}.

For the two layers on nine points, use the affine plane $\operatorname{AG}(2,3)$ with line set
\[
\begin{split}
\mathcal L=\{&012,036,048,057,138,147,156,237,246,258,\\
             &345,678\}.
\end{split}
\]
If $g=(0\ 1\ 4\ 2\ 5\ 3\ 6)$ fixes $7$ and $8$, direct expansion shows that the seven systems $g^i\mathcal L$, $i\in\Z_7$, are pairwise disjoint and partition $J(9,3)$.
The line/nonline orbit partition of the solvable group $\operatorname{AGL}(2,3)$ therefore sends every triple to $J(7,1)$.
On four-sets, the same affine group has two orbits, according as the set contains a line or is a four-arc.
The four triples inside a fixed four-set belong to four different systems $g^i\mathcal L$, since two blocks of a Steiner triple system cannot share two points.
Hence the four-set contains a line in four coordinates and is a four-arc in the other three, so its orbit word lies in $J(7,3)$ after interchanging the two labels.

Finally, for $J(10,3)$ partition the ten points into five distinguished pairs and use
\[
 L=C_2^5\rtimes\operatorname{AGL}(1,5).
\]
Its two triple-orbits consist of triples containing a distinguished pair and triples meeting three distinct pairs.
Choose a one-factorization $E(K_{10})=F_1\sqcup\cdots\sqcup F_9$ and permutations $\pi_i$ taking $F_i$ to the distinguished matching.
A triple has three edges, and these lie in three distinct one-factors because adjacent edges cannot lie in one matching.
Its inner orbit word consequently has weight three and lies in the full outer supertemplate $J(9,3)$, which was just resolved.

All coordinate permutations used above preserve the indicated full constant-weight inner supertemplate, so the supertemplate condition in \cref{def:iterated-coordinate-scheme} is satisfied.
The base cases, the seven certificates, and complementation prove the theorem.
\end{proof}

\begin{theorem}[setwise-stabilizer compression]
\label{thm:setwise-stabilizer-compression}
Let a finite group $H\leq\operatorname{Iso}(P)$ act transitively on a spherical configuration $P$, and let a solvable subgroup $G\leq H$ have two orbits, one of which is $O$.
Put
\[
 M=\operatorname{Stab}_H(O),\qquad
 t=[H:M],\qquad
 \lambda=\frac{t|O|}{|P|}.
\]
Then $\lambda$ is an integer.
If $J(t,\lambda)$ has an iterated solvable coordinate scheme, then $P$ embeds in a finite solvable transitive configuration of the same circumradius and is therefore ncs-Ramsey.
In particular, this conclusion holds whenever $t\leq9$, and also when $t=10$ and $\lambda\in\{1,2,3,7,8,9\}$.
\end{theorem}

\begin{proof}
Translate $P$ so that its circumcenter is the origin.
The $t$ distinct $H$-translates of $O$ form a tactical configuration on $P$.
Since $H$ is transitive, every point lies in the same number $\lambda$ of these translates, and counting incidences gives $\lambda|P|=t|O|$.
Choose $h_1,\ldots,h_t\in H$ so that the distinct translates are $h_i^{-1}O$.
For every $x\in P$, the support
\[
 S_x=\{i:h_ix\in O\}
\]
therefore lies in $J(t,\lambda)$.
Apply \cref{lem:iterated-orbit-design} with this full layer as a supertemplate, and then use \cref{thm:small-johnson-layers}.
\end{proof}

The index $t$ can be much smaller than the raw coset index $[H:G]$.
Indeed, if $s=[H:G]$, if $a=s|O|/|P|$, and if $e=[M:G]$, then
\[
 t=s/e,\qquad \lambda=a/e.
\]
Thus the common factor forced by the setwise stabilizer should be cancelled before the layer is tested.

\subsection{Substitution along a subgroup chain}
\label{subsec:subgroup-chain-substitution}

The compressed support family has additional structure which resolves every imprimitive case built from smaller resolved factors.
Choose $x_0\in O$, put $J=H_{x_0}$, and identify $P$ with $H/J$.
Since $M$ is transitive on $O$, the compressed family on $H/M$ is
\[
 \mathcal F(H;M,J)=\{h(JM/M):h\in H\}\subseteq2^{H/M},
 \qquad JM/M:=\{jM:j\in J\}.
\]

\begin{theorem}[subgroup-chain substitution]
\label{thm:subgroup-chain-substitution}
For every intermediate subgroup $M\leq U\leq H$, after choosing a transversal of $H/U$ and thereby identifying each $U$-fibre with $U/M$,
\[
 \mathcal F(H;M,J)
 \subseteq
 \mathcal F(H;U,J)\bigl[\mathcal F(U;M,J\cap U)\bigr].
\]
Consequently, for a chain
\[
 M=U_0<U_1<\cdots<U_d=H,
\]
the compressed family has an iterated solvable coordinate scheme whenever every local factor
\[
 \mathcal F_i=\mathcal F(U_i;U_{i-1},J\cap U_i)
 \subseteq J(t_i,\lambda_i),
\]
has one, where
\[
 t_i=[U_i:U_{i-1}],\qquad
 \lambda_i=[J\cap U_i:J\cap U_{i-1}].
\]
In particular, a chain whose nontrivial degrees $t_i$ are all at most nine is sufficient.
\end{theorem}

\begin{proof}
Partition the coordinate set into $U$-fibres,
\[
 H/M=\bigsqcup_{hU\in H/U}hU/M.
\]
The active fibres of $rJM/M$ are exactly the members of $rJU/U$.
If $rjU/M$ is active, where $j\in J$, then
\[
 (rJM/M)\cap(rjU/M)
 =rj\bigl((J\cap U)M/M\bigr).
\]
Under the chosen identification of this fibre with $U/M$, left translation by $rj$ identifies the intersection with a $U$-translate of
\[
 (J\cap U)M/M\subseteq U/M.
\]
This proves the displayed containment in the wreath supertemplate.
Apply the substitution step in \cref{def:iterated-coordinate-scheme} to run
the inner scheme independently in every active fibre, retaining an inactive
fibre as a separate terminal type, and then run the outer scheme on the fibre
types.
Iteration along the chain proves the second assertion.
The formula for $t_i$ is immediate, and the support weight is $\lambda_i$ by orbit--stabilizer.
The final assertion follows from \cref{thm:small-johnson-layers}.
\end{proof}

If $(J\cap U_i)U_{i-1}$ is a subgroup of $U_i$, the corresponding local supports are disjoint coset blocks and a cyclic permutation of the blocks resolves that step directly.
In particular, normal or more generally permutable steps are free, so they may be stripped from the chain before a primitive residual factor is considered.

\section{Exact-radius solvable extensions of regular polytopes}
\label{app:regular-polytopes}

We call a finite spherical configuration $P$ \emph{exact-radius solvable subtransitive} if it embeds isometrically in a finite solvable transitive configuration having the same circumradius as $P$.
Such a configuration is ncs-Ramsey by \cref{thm:main}: apply that theorem to the host and then retain the embedded copy of $P$.
The purpose of this section is to prove that every regular convex polytope has this property.

We use the iterated orbit-design construction of \cref{lem:iterated-orbit-design}, including its invariant-supertemplate requirement.
The exact-radius conclusion there does not require any additional centering hypothesis on an intermediate restricted product: every host point has norm $\rho$, while the contained isometric copy of $P$ has intrinsic circumradius $\rho$.

\subsection{The dodecahedron}

We begin with the exceptional three-dimensional case needed later.

\begin{proposition}
\label{prop:app-dodecahedron}
The vertex set of the regular dodecahedron has an exact-radius solvable transitive extension.
\end{proposition}

\begin{proof}
Let $H\cong A_5\times C_2$ be the full symmetry group, where the second
factor is central inversion.
Choose one of the five cubes inscribed in the dodecahedron, and let
$K\cong A_4$ be its stabilizer in the rotational factor.
The eight vertices of this cube split into two antipodal $K$-orbits of
size four.
Every remaining vertex has trivial stabilizer in $K$, so the remaining
twelve vertices form one regular $K$-orbit.
Consequently the solvable group $G=K\times C_2$ has two vertex-orbits $O_0,O_1$ of sizes $8$ and $12$.

The subgroup $G$ has index five in $H$.
The setwise stabilizer of $O_0$ in $H$ is exactly $G$: it contains $G$, while $A_4$ is maximal in $A_5$ and $O_0$ is not $H$-invariant.
Choose representatives $h_1,\ldots,h_5$ of its five right cosets $G\backslash H$ and label a coordinate by $1$ when $h_ix\in O_0$.
For every vertex $x$, the resulting binary word has weight
\[
 \frac{5|O_0|}{|P|}=\frac{5\cdot8}{20}=2.
\]
The supports form a single $A_5$-orbit of two-subsets of the five cosets, so they are all of $J(5,2)$, each with the same multiplicity.
The solvable group $\operatorname{AGL}(1,5)$ is two-transitive on the five coordinates and hence transitive on $J(5,2)$.
The direct case of \cref{lem:iterated-orbit-design} gives the required extension, in fact inside a normalized fivefold product of the original three-dimensional space.
\end{proof}

\subsection{The regular 600-cell}

Let $I\leq\operatorname{Sp}(1)$ be the binary icosahedral group,
traditionally denoted $2I$, of order $120$.
In the standard quaternionic model, its elements are the vertices of the regular $600$-cell; see \cite{ChoiLee2018}.
Write $Z=\{\pm1\}$ and let
\[
 \pi:I\longrightarrow I/Z\cong A_5
\]
be the quotient map.
Let
\[
 T=\pi^{-1}(A_4),
 \qquad
 C=\pi^{-1}(C_5)\cong C_{10}.
\]

\begin{proposition}
\label{prop:app-600-cell}
The regular $600$-cell is solvable transitive.
\end{proposition}

\begin{proof}
The group $T$ is the binary tetrahedral group, traditionally denoted $2T$.
Since $A_4\cap C_5=1$ and $|A_4||C_5|=60$, we have $A_5=A_4C_5$ and hence $I=TC$.
The solvable group $T\times C$ acts orthogonally on $\mathbb H\cong\R^4$ by
\[
 (t,c):x\longmapsto txc^{-1}.
\]
Its orbit of $1$ is $TC=I$, so it is transitive on all $120$ vertices.
The kernel of the linear action is the diagonal subgroup $\Delta Z=\{(1,1),(-1,-1)\}$.
Thus the induced group has order $120$ and acts regularly; in particular it is solvable.
\end{proof}

\subsection{An exact-radius extension of the regular 120-cell}

Identify the vertices of the regular $120$-cell with the $600$ tetrahedral cells of the $600$-cell, represented by their suitably scaled outward normal vectors.
Thus the cell set $P\subset\R^4$ is spherical, and the left-right quaternionic rotations act on it.
Put
\[
 R=(I\times I)/\Delta Z,
 \qquad
 \Gamma=(I\times C)/\Delta Z,
 \qquad
 G=(T\times C)/\Delta Z.
\]
The group $R$ is the rotational symmetry group of the $600$-cell, while $|\Gamma|=600$ and $[\Gamma:G]=5$.

\begin{lemma}
\label{lem:app-gamma-regular}
The group $\Gamma$ acts regularly on the $600$ tetrahedral cells.
\end{lemma}

\begin{proof}
Fix a tetrahedral cell $F$.
Its stabilizer $R_F$ in $R$ is the rotational tetrahedral group $A_4$.
If $[a,c]\in\Gamma$ preserves $F$, then $cZ$ belongs to $C/Z\cong C_5$.
On the other hand, restriction of the right projection $R\to I/Z$ to
$R_F$ is a homomorphism, so the order of $cZ$ divides $|A_4|=12$.
Hence $cZ=1$, and the transformation has a representative $[a,1]$ and is
a pure left multiplier.

A nonidentity pure left multiplier fixes no vertex of the $600$-cell.
If it preserved $F$, its permutation of the four vertices would therefore be a fixed-point-free element of $A_4$, hence a double transposition.
Its square would fix all four vertices, forcing $a^2=1$.
The only unit quaternions satisfying this equation are $a=\pm1$; the nonidentity possibility $a=-1$ sends $F$ to its distinct antipodal cell.
This is impossible, so the cell stabilizer in $\Gamma$ is trivial.
Since $|\Gamma|=|P|=600$, the action is regular.
\end{proof}

The same argument, with left and right interchanged, shows that no nonidentity pure right multiplier preserves a tetrahedral cell.

Choose a base cell $x_\ast$ and a transversal $h_1,\ldots,h_5$ for the five right cosets $G\backslash\Gamma$.
By regularity, identify $P$ with $\Gamma$ through $\gamma\mapsto\gamma x_\ast$.
The five $G$-orbit labels of $h_1\gamma x_\ast,\ldots,h_5\gamma x_\ast$ are obtained from the transversal labels by right multiplication by $\gamma$.
The resulting degree-five permutation representation is
\[
 A_5\curvearrowright A_5/A_4,
\]
and its kernel in $\Gamma$ has order ten.
Consequently the type words are the sixty even permutations of five distinct symbols, each occurring ten times.

We next fuse these five orbits.
Let
\[
 D=N_I(C),
 \qquad
 N=(T\times D)/\Delta Z.
\]
Then $D/C\cong C_2$, the group $N$ is solvable, and $G\triangleleft N$.
Choose $d\in D\setminus C$ and put $n=[1,d]$.
The element $n$ normalizes both $\Gamma$ and $G$.
There is a unique $\gamma_\ast\in\Gamma$ such that
$nx_\ast=\gamma_\ast x_\ast$.
Let $q:\Gamma\to I/Z$ be the first-coordinate quotient map.
Conjugation by $n$ fixes $q(\gamma)$ for every $\gamma\in\Gamma$, and
$\ker q=(Z\times C)/\Delta Z\leq G$.
Thus $G(n\gamma n^{-1})=G\gamma$ for every $\gamma\in\Gamma$.
Under the regular identification $P\cong\Gamma$, the action of $n$ therefore
sends the orbit label $G\gamma$ to $G\gamma\gamma_\ast$.
Hence its permutation of the five $G$-orbits is right multiplication by
$\gamma_\ast$, and belongs to the displayed natural $A_5$.

This permutation is nontrivial.
Indeed,
\[
 \operatorname{core}_{\Gamma}(G)=(Z\times C)/\Delta Z,
\]
because $G$ is the inverse image of $A_4$ under the quotient $\Gamma\to A_5$ and $\operatorname{core}_{A_5}(A_4)=1$.
If the permutation were trivial, then $\gamma_\ast$ would lie in this core, and $\gamma_\ast^{-1}n$ would be a nonidentity pure right multiplier fixing the cell $x_\ast$, contrary to the observation following \cref{lem:app-gamma-regular}.
Moreover $n^2\in G$, since $d^2\in C$.
Thus $n$ induces a nontrivial involution of the natural $A_5$, necessarily a double transposition.
It fuses the five $G$-orbits in blocks of sizes $1,2,2$.
The resulting $N$-orbits on $P$ have sizes $120,240,240$, and the fused type words are all thirty words in
\begin{equation}
 \mathcal C=\operatorname{Perm}(01122),
 \label{eq:app-01122}
\end{equation}
each occurring twenty times.

It remains to give an iterated solvable coordinate scheme for \eqref{eq:app-01122}.
Index its five positions by $\mathbb F_5$ and let $L=\operatorname{AGL}(1,5)$.
The group $L$ has two orbits on $\mathcal C$.
The smaller one, of size ten, is
\[
 \mathcal Q=\left\{
 w\in\mathcal C:
 \text{if $x$ is the $0$-position and $\{y,z\}$ are the $1$-positions, then }y+z=2x
 \right\}.
\]
Its setwise stabilizer in $S_5$ is exactly $L$.
To see this, normalize a relation-preserving permutation by an affine map so that it fixes $0$ and $1$.
It must then fix $-1$, preserve $\{2,3\}$, and the relation with midpoint $1$ and endpoints $0,2$ forces it to fix $2$, hence every point.

The six $S_5$-translates of $\mathcal Q$ are indexed by the six Sylow-$5$ subgroups of $S_5$.
The induced degree-six action is two-transitive: the normalizer of a fixed Sylow-$5$ subgroup acts transitively on the other five.
Every word of $\mathcal C$ belongs to exactly
\[
 \frac{6|\mathcal Q|}{|\mathcal C|}=2
\]
of these translates.
Therefore, after choosing coordinate permutations $\pi_1,\ldots,\pi_6$ representing them, the $L$-orbit-type supports
\[
 \{i:\pi_iw\in\mathcal Q\}
\]
run through all fifteen members of $J(6,2)$, each twice.

Finally identify the six positions with the six edges of $K_4$.
The solvable group $S_4$ has two orbits on $J(6,2)$: the twelve adjacent edge-pairs and the three opposite edge-pairs.
Choose a one-factorization $F_1,\ldots,F_5$ of $K_6$ and coordinate permutations $\sigma_i\in S_6$ taking $F_i$ to the three opposite pairs.
Every pair belongs to exactly one $F_i$, so the resulting five-letter type word has weight one.
The cyclic group $C_5$ is transitive on $J(5,1)$.
All supertemplates used here are invariant---first the full multislice $\operatorname{Perm}(01122)$ and then the full layer $J(6,2)$---so this is a valid iterated scheme in the sense of \cref{def:iterated-coordinate-scheme}.

\begin{theorem}
\label{thm:app-120-cell}
The vertex set of the regular $120$-cell embeds isometrically in a finite solvable transitive configuration in $\R^{600}$ having exactly the same circumradius.
Consequently the regular $120$-cell is ncs-Ramsey.
\end{theorem}

\begin{proof}
Apply \cref{lem:iterated-orbit-design} to the $N$-orbits, the five isometries $h_1,\ldots,h_5$, and the iterated scheme just constructed.
The three successive normalized product lengths are $5$, $6$, and $5$, so the final host lies in
\[
 \R^{4\cdot5\cdot6\cdot5}=\R^{600}.
\]
Every acting group is a subgroup of an iterated wreath product of copies of
\[
 N,\qquad \operatorname{AGL}(1,5),\qquad S_4,\qquad C_5,
\]
and is therefore solvable.
Exactness of the circumradius is part of \cref{lem:iterated-orbit-design}.
The ncs-Ramsey conclusion follows from \cref{thm:main}.
\end{proof}

\begin{remark}
The enlargement in \cref{thm:app-120-cell} is genuinely necessary if one insists on a solvable transitive symmetry group.
Modulo central inversion, the full symmetry group of the $120$-cell is $(A_5\times A_5)\rtimes C_2$.
The intersection of a solvable subgroup with $A_5\times A_5$ has solvable projections of order at most twelve, so, after restoring the factor swap and central inversion, its order is at most
\[
 2\cdot2\cdot12^2=576<600.
\]
Thus no solvable subgroup of the symmetry group can be transitive on the $600$ vertices.
\end{remark}

\subsection{All regular polytopes}

\begin{corollary}
\label{cor:app-all-regular-polytopes}
The vertex set of every regular convex polytope is exact-radius solvable subtransitive and hence ncs-Ramsey.
\end{corollary}

\begin{proof}
Use the standard classification of regular convex polytopes; see, for example, \cite{Coxeter1973}.
The infinite families already have solvable vertex-transitive groups.
A cyclic group is transitive on the vertices of a regular polygon and of a regular simplex.
The elementary abelian sign-change group is transitive on a cube, and the signed cyclic shift
\[
 e_1\mapsto e_2\mapsto\cdots\mapsto e_d\mapsto-e_1
\]
generates a cyclic group transitive on the $2d$ vertices of the crosspolytope.

Among the exceptional three-dimensional polytopes, the icosahedron is solvable transitive because $A_5=A_4C_5$ makes $A_4$ regular on the twelve cosets $A_5/C_5$, while the dodecahedron is covered by \cref{prop:app-dodecahedron}.
In dimension four, the signed coordinate group $C_2^4\rtimes S_4$ is solvable and transitive on the twenty-four vectors obtained by permuting $(\pm1,\pm1,0,0)$, a standard model of the $24$-cell.
The $600$-cell is covered by \cref{prop:app-600-cell}, and its dual $120$-cell is covered by \cref{thm:app-120-cell}.
All remaining cases belong to the three infinite families.
Thus every regular convex polytope either is solvable transitive or has the exact-radius solvable transitive extension constructed above, and the conclusion follows from \cref{thm:main}.
\end{proof}

\section{Further reductions for the two-orbit problem}
\label{sec:appendix-two-orbit-reductions}

\subsection{The affine primitive case of the two-orbit problem}
\label{subsec:two-orbit-affine}

We record a reduction which removes the entire affine O'Nan--Scott case
from the two-orbit problem.
Let $H$ act transitively on a finite spherical configuration $P$, let
$G\leq H$ be solvable with exactly two orbits $O$ and $P\setminus O$, and
choose $x\in O$.
Put
\[
 J=H_x,
 \qquad
 M=\operatorname{Stab}_H(O).
\]
The distinct translates of $O$ are indexed by $H/M$, and their binary
support family is
\[
 \mathcal F(H;M,J)
 =\{h(JM/M):h\in H\}\subseteq 2^{H/M}.
\]
Since $G$ is transitive on $O=Mx$, we have the elementary factorization
\begin{equation}
 M=G(M\cap J).
 \label{eq:MGK-factorization}
\end{equation}

\begin{lemma}[compatible transitive subgroup]
\label{lem:compatible-coordinate-subgroup}
Suppose that $A\leq H$ is transitive on $H/M$ and that
$\langle A,G\rangle$ is solvable.
Then $P$ is solvable transitive.
\end{lemma}

\begin{proof}
Transitivity of $A$ on $H/M$ says $H=AM$.
Using \eqref{eq:MGK-factorization},
\[
 H=AM=AG(M\cap J)\subseteq \langle A,G\rangle J.
\]
Thus $H=\langle A,G\rangle J$, so the solvable group
$\langle A,G\rangle$ is transitive on $H/J=P$.
\end{proof}

The compatibility requirement in
\cref{lem:compatible-coordinate-subgroup} is automatic, for example, if
$A$ is solvable and normalized by $G$.
There is also a useful purely local consequence.

\begin{proposition}[a common prime-power case]
\label{prop:common-p-power}
Suppose that $[H:M]$ is a power of a prime $p$ and that the two-orbit
witness $G$ is a $p$-group.
Then $P$ is solvable transitive.
More generally, the same conclusion holds if $G$ normalizes some Sylow
$p$-subgroup of $H$.
\end{proposition}

\begin{proof}
Every Sylow $p$-subgroup $S$ of $H$ is transitive on $H/M$.
Indeed, if $[H:M]=p^e$, then
\[
 |S|=|H|_p=p^e|M|_p.
\]
On the one hand $|S:S\cap M|\leq[H:M]=p^e$, while on the other
$|S\cap M|\leq|M|_p$ and hence
\[
 |S:S\cap M|=\frac{p^e|M|_p}{|S\cap M|}\geq p^e.
\]
Thus equality holds, and the orbit of $M$ under $S$ is all of $H/M$.
If $G$ normalizes $S$, the group $SG$ is solvable and is transitive on
$H/M$, so \cref{lem:compatible-coordinate-subgroup} applies.
If $G$ is a $p$-group, choose a Sylow $p$-subgroup $S$ of $H$ containing
$G$; then $G$ certainly normalizes $S$.
\end{proof}

The following normal-subgroup version is the form useful for primitive
groups.

\begin{theorem}[solvable-normal fusion]
\label{thm:solvable-normal-fusion}
Suppose that $N\trianglelefteq H$ is solvable and transitive on $H/M$.
Then $NG$ is a solvable subgroup of $H$ which is transitive on $P$.
In particular, $P$ itself is solvable transitive and hence
nearcircumsphere-Ramsey.
\end{theorem}

\begin{proof}
The product $NG$ is a subgroup because $N$ is normal, and it is solvable
because it is an extension of the solvable group $N$ by a quotient of the
solvable group $G$.
Transitivity of $N$ gives $H=NM$.
Consequently, by \eqref{eq:MGK-factorization},
\[
 H=NM=NG(M\cap J)=(NG)J.
\]
Hence $NG$ is transitive on $H/J=P$.
\end{proof}

\begin{corollary}[no affine primitive counterexample]
\label{cor:no-affine-primitive-counterexample}
Assume that the action $H\curvearrowright H/M$ is faithful and primitive.
If the solvable radical $\operatorname{Rad}(H)$ is nontrivial, then $P$ is
solvable transitive.
Therefore a primitive counterexample to the two-orbit
nearcircumsphere-Ramsey problem must satisfy
\[
 \operatorname{Rad}(H)=1;
\]
in particular its socle is a direct product of nonabelian simple groups.
Equivalently, the affine O'Nan--Scott case is completely absent from the
residual problem.
\end{corollary}

\begin{proof}
Every nontrivial normal subgroup of a faithful primitive permutation group
is transitive.
Apply \cref{thm:solvable-normal-fusion} with
$N=\operatorname{Rad}(H)$.
If the radical is trivial, every minimal normal subgroup is a direct product
of isomorphic nonabelian simple groups, which gives the assertion about the
socle.
\end{proof}

In fact there are at most two minimal normal subgroups in the residual
primitive group, and if there are two then both are regular on $H/M$.
Indeed, distinct minimal normal subgroups centralize one another, while the
centralizer of a transitive permutation group is semiregular.
Thus two distinct minimal normal subgroups are both regular.
The centralizer of a regular group is itself regular of order $[H:M]$, so a
third transitive minimal normal subgroup in that centralizer would coincide
with the second.

This is stronger than merely constructing a solvable tower for the
compressed support code: in the affine case the original two-orbit
configuration already has a solvable transitive group.
Combined with the subgroup-chain substitution theorem, it says that a
minimal unresolved compressed family has a faithful primitive coordinate
action of nonabelian-socle type.

There is a useful exact target inside such a nonabelian socle.

\begin{proposition}[compatible socle supplement]
\label{prop:compatible-socle-supplement}
Assume that $N\trianglelefteq H$ is transitive on $H/M$.
If there is a solvable subgroup $C\leq N$ such that
\[
 N=C(N\cap J)
 \qquad\text{and}\qquad
 G\leq N_H(C),
\]
then $CG$ is solvable and transitive on $P$.
\end{proposition}

\begin{proof}
Put $K=M\cap J$.
Since $H=NM$ and $M=GK$, every element of $H$ has the form $ngk$ with
$n\in N$, $g\in G$, and $k\in K$.
 Write
\[
 g^{-1}ng=cd,
 \qquad c\in C,\quad d\in N\cap J.
\]
Then
\[
 ngk=gcdk=(gcg^{-1})g(dk)\in CGJ.
\]
Thus $H=CGJ$.
The normalization hypothesis makes $CG$ a subgroup, and it is solvable as
an extension of $C$ by a quotient of $G$.
Hence $CG$ is transitive on $H/J=P$.
\end{proof}

For a faithful primitive coordinate action, every nontrivial normal
subgroup is transitive, so
\cref{prop:compatible-socle-supplement} applies in particular to a minimal
normal subgroup.
It does not assert that the required $G$-normalized supplement always
exists; rather, it isolates the remaining factorization problem inside the
nonabelian socle.

\subsection{A sharper character-theoretic certificate}
\label{subsec:two-orbit-character-certificate}

Retain the preceding notation and suppose that $M$ is the full setwise
stabilizer of $O$.
Write
\[
 a=|O|,
 \qquad b=|P\setminus O|,
 \qquad v=a+b,
 \qquad t=[H:M].
\]
Let $A$ be the $v$-by-$t$ incidence matrix between points of $P$ and
distinct translates of $O$.
Every column of $A$ has weight $a$, while every row has weight
\[
 \lambda=\frac{ta}{v}.
\]
The two-double-coset identity $|M\backslash H/J|=2$ implies that the two
permutation characters $\operatorname{Ind}_M^H1$ and
$\operatorname{Ind}_J^H1$ have, besides the trivial character, exactly one
common irreducible constituent $\chi$, with multiplicity one in each.
Put $d=\chi(1)$.

\begin{proposition}[faithfulness and arithmetic]
\label{prop:faithful-common-constituent}
The character $\chi$ is rational-valued, has Schur index one over
$\mathbb Q$, and
\[
 \dim \chi^M=\dim \chi^J=\dim \chi^G=1.
\]
Moreover,
\begin{equation}
 \ker\chi=\operatorname{core}_H(M).
 \label{eq:kernel-core}
\end{equation}
Thus $\chi$ is faithful after passing to the faithful coordinate quotient.

If
\[
 g=\gcd(a,b),\qquad a=g\alpha,\qquad b=g\beta,
 \qquad \gcd(\alpha,\beta)=1,
\]
then there is a positive integer $c$ for which
\begin{equation}
 t=c(\alpha+\beta),
 \qquad \lambda=c\alpha,
 \label{eq:primitive-orbit-arithmetic}
\end{equation}
and
\begin{equation}
 d\mid cg\alpha\beta.
 \label{eq:character-degree-divisibility}
\end{equation}
In particular, every residual primitive candidate has a faithful rational
irreducible character satisfying
\[
 d+1\leq\min\{v,t\}
 \quad\text{and}\quad
 d\mid cg\alpha\beta.
\]
If, in addition, $\operatorname{core}_H(M)=1$ and a minimal normal subgroup
has the form $N=T^s$, where $T$ is nonabelian simple and $\mu(T)$ denotes
the least degree of a nontrivial complex irreducible representation of
$T$, then also
\begin{equation}
 d\geq s\mu(T).
 \label{eq:socle-degree-lower-bound}
\end{equation}
\end{proposition}

\begin{proof}
Frobenius reciprocity gives
\[
 \left\langle\operatorname{Ind}_M^H1,
 \operatorname{Ind}_J^H1\right\rangle
 =|M\backslash H/J|=2.
\]
After removing the common trivial constituent, uniqueness gives one common
irreducible $\chi$, with multiplicity one on both sides.
Galois conjugation preserves both rational permutation characters, so it
fixes $\chi$; hence $\chi$ is rational-valued.
Its Schur index over $\mathbb Q$ divides its multiplicity in every rational
module containing it, and is therefore one.
Frobenius reciprocity also gives
$\dim\chi^M=\dim\chi^J=1$.

The centered characteristic vector
\[
 y_O={\bf1}_O-\frac av{\bf1}_P
\]
is nonzero, is fixed by $G$, and belongs to the unique $\chi$-summand in
$\mathbb C[P]$ hit by the centered incidence operator.
Thus $\dim\chi^G\geq1$.
On the other hand $G$ has exactly two orbits on $P$, so
$\dim\mathbb C[P]^G=2$; the trivial summand already contributes one
dimension.  It follows that $\dim\chi^G=1$.

The $H$-orbit of $y_O$ is the collection of centered characteristic vectors
of the distinct translates of $O$.
Its stabilizer is exactly $M$, and these vectors span the $\chi$-summand.
An element in $\ker\chi$ therefore fixes every translate of $O$, and hence
lies in $\operatorname{core}_H(M)$.
The reverse containment is immediate, proving
\eqref{eq:kernel-core}.

Since $\lambda=ta/v$ is an integer and
$v=g(\alpha+\beta)$, coprimality gives
$\alpha+\beta\mid t$.
Writing $t=c(\alpha+\beta)$ gives
\eqref{eq:primitive-orbit-arithmetic}.
Let $\mathbf J_{v\times t}$ be the $v$-by-$t$ all-ones matrix and put
\[
 B=A-\frac av\mathbf J_{v\times t}.
\]
This is the centered incidence matrix.
Its image is the common $\chi$-summand, so $\operatorname{rank}B=d$;
multiplicity one and Schur's lemma show that $BB^{\mathsf T}$ has a single
nonzero eigenvalue $\sigma^2$, of multiplicity $d$.
Moreover,
\[
 d\sigma^2=\operatorname{tr}(BB^{\mathsf T})
 =t\left(a-\frac{a^2}{v}\right)=\lambda(v-a),
\]
whence
\[
 \sigma^2=\frac{\lambda(v-a)}d=\frac{cg\alpha\beta}{d}.
\]
On the nontrivial image, $BB^{\mathsf T}$ and the integral matrix
$AA^{\mathsf T}$ agree.
Thus $\sigma^2$ is an eigenvalue of an integral matrix.
Being rational, it is a rational algebraic integer and hence an integer.
This proves \eqref{eq:character-degree-divisibility}.
Finally, the uncentered incidence matrix has rank $d+1$, which is at most
both its number of rows and its number of columns.

For the last assertion, Clifford theory writes
\[
 \chi|_N=e(\theta_1+\cdots+\theta_q),
\]
where the $\theta_i$ are the distinct $H$-conjugates of an irreducible
character $\theta$ of $N$.
Suppose that $\theta$ is nontrivial on precisely $r$ of the $s$ simple
direct factors.
Faithfulness of $\chi$ gives $r\geq1$.
Then $\theta(1)\geq\mu(T)^r$.
The group $H$ is transitive on the $s$ factors of its minimal normal
subgroup, and faithfulness of $\chi$ forces the supports of the conjugates
$\theta_i$ to cover all factors.
Thus $qr\geq s$, and
\[
 d=eq\theta(1)
 \geq \frac{s}{r}\mu(T)^r
 \geq s\mu(T),
\]
because $\mu(T)^{r-1}\geq2^{r-1}\geq r$.
\end{proof}

Two easy boundary cases can also be removed without any group
classification.

\begin{proposition}[singleton orbit]
\label{prop:singleton-two-orbit}
If one of the two $G$-orbits on $P$ is a singleton, then $P$ is a regular
simplex and is therefore solvable transitive.
\end{proposition}

\begin{proof}
Suppose that $G$ fixes $x$ and is transitive on $P\setminus\{x\}$.
Then the point stabilizer $H_x$, which contains $G$, is transitive on
$P\setminus\{x\}$.
As $H$ is transitive on $P$, its action is two-transitive.
Consequently every pair of distinct points of $P$ has the same distance, so
$P$ is a regular simplex in its affine span.
The full symmetric group of its vertices acts by isometries; in particular,
a cyclic subgroup is solvable and transitive.
\end{proof}

For a group-theoretic census it is convenient to remove repetitions in the
support rows.
Put
\[
 D=JM/M\subseteq H/M,
 \qquad J^*=\operatorname{Stab}_H(D).
\]
Replacing $J$ by $J^*$ does not change the support family.

\begin{proposition}[irredundant rows and the second normalizer]
\label{prop:second-normalizer}
Assume that $H\curvearrowright H/M$ is faithful and primitive, and that
$M$ has two orbits on $H/J$.
Then
\[
 |M\backslash H/J^*|=2,
 \qquad \operatorname{core}_H(J^*)=1.
\]
If $\lambda=|D|\neq t/2$, then
\[
 N_H(J^*)=J^*.
\]
If $N_H(J^*)>J^*$, then $\lambda=t/2$, and every
$u\in N_H(J^*)\setminus J^*$ interchanges $D$ and its complement.
\end{proposition}

\begin{proof}
We have $J\leq J^*$ and $J^*M/M=D$.
The two $M$-orbits do not merge after replacing $J$ by $J^*$: every support
in the orbit of $D$ contains the base column $M$, while every support in the
other orbit avoids it.
Thus $|M\backslash H/J^*|=2$.

In the coordinate permutation module put
\[
 w={\bf1}_D-\frac{\lambda}{t}{\bf1}.
\]
The vector $w$ spans the $J^*$-fixed line in the faithful common
$\chi$-module.
The core of $J^*$ fixes the $H$-orbit of $w$, which spans this irreducible
module, so \cref{eq:kernel-core} gives
$\operatorname{core}_H(J^*)=1$.

If $u\in N_H(J^*)$, then $u$ preserves this one-dimensional real fixed
space.
With an invariant inner product, $uw=\pm w$.
The plus sign gives $uD=D$, hence $u\in J^*$ by definition.
In the minus case,
\[
 {\bf1}_{uD}+{\bf1}_D=\frac{2\lambda}{t}{\bf1}.
\]
Since $D$ is neither empty nor the whole coordinate set, the constant on
the right must equal one.
Therefore $\lambda=t/2$ and $uD=D^c$.
\end{proof}

One should not, however, try to remove the remaining nonabelian-socle case
merely by assuming that $M$ is solvable.
The primitive family
\[
 H=S_6,\qquad M=S_3\wr C_2,\qquad J=S_2\times S_4
\]
on the ten unordered $3+3$ equipartitions and the fifteen pairs is a
counterexample to such a shortcut.
Let $O$ be the six-element $M$-orbit consisting of the pairs lying within
one part of the base equipartition.
Since $[H:M]=10$, its support rows have weight
\[
 \lambda=\frac{10|O|}{15}=4.
\]
Thus the associated support code is a $(10,4)$ code.
The group $M$ is solvable, but this code has no one-level solvable
row-transitive coordinate group.
Indeed, for two pair-rows $e,f$, the intersection size of their supports is
one when $e$ and $f$ meet and two when they are disjoint.
Thus support intersections recover the line graph $L(K_6)$.
A solvable coordinate group transitive on the fifteen rows would induce a
solvable group transitive on the two-subsets of a six-set, which is
impossible: every finite solvable $2$-homogeneous permutation group has
prime-power degree~\cite{Kantor1972}.

This particular code is nevertheless resolved by a two-level scheme.
Embed $S_4$ in $S_6$ through its action on the six edges of $K_4$.
On the fifteen pair-rows it has the two orbits consisting of the twelve
adjacent edge-pairs and the three opposite edge-pairs.
Choose a one-factorization $F_1,\ldots,F_5$ of $K_6$ and permutations
$\sigma_i\in S_6$ taking $F_i$ to the three opposite pairs.
Every pair belongs to exactly one $F_i$, so the outer type word lies in
$J(5,1)$.
The induced coordinate permutations preserve the full support family on
the ten equipartitions, making this a valid recursive scheme.

Consequently, after all elementary reductions, a genuinely unresolved
primitive triple $(H,M,J)$ must have all of the following properties:
\begin{enumerate}
\item $M$ is core-free, maximal, nonnormal, and self-normalizing;
\item $\operatorname{Rad}(H)=1$, so the socle is nonabelian;
\item both $G$-orbits on $P$ have size at least two;
\item the pair consisting of the translate degree and the witness group is
      not covered by \cref{prop:common-p-power};
\item the unique common constituent $\chi$ is faithful, rational, has
      Schur index one, and satisfies
      \cref{eq:primitive-orbit-arithmetic,eq:character-degree-divisibility};
\item after identifying repeated support rows, the second stabilizer $J^*$
      is also core-free, and it is self-normalizing unless
      $\lambda=t/2$; in the latter case every additional normalizer element
      complements its support;
\item no solvable subgroup $A$ transitive on $H/M$ is compatible with $G$
      in the sense that $\langle A,G\rangle$ is solvable.
\end{enumerate}
These conditions give a substantially smaller and readily checkable census
than arbitrary primitive two-double-coset pairs.

\section{Fixed-degree block sets for word templates}
\label{sec:appendix-word-templates}

\subsection{The template \texorpdfstring{$112233$}{112233}}
\label{subsec:appendix-112233}

Leader, Russell, and Walters~\cite[Problem~G]{LeaderRussellWalters2012}
asked for a proof of the Block Sets Conjecture for the template $112233$.
Ivan, Leader, and Walters~\cite[Conjecture~4]{IvanLeaderWalters2026} later
conjectured that, for every fixed template, the block degree can be chosen
independently of the number of colors.  We prove both statements for
$112233$.

A map $\Psi:[3]^6\to[3]^N$ is called a \emph{balanced literal role map of degree $d$} if every coordinate of $\Psi(x)$ is either fixed or is one of the six projections $x_i$, and every projection occurs exactly $d$ times.
Thus the restriction of such a map to $\operatorname{Perm}(112233)$ is a block set of degree $d$.
We use \emph{HJ-degree} in the sense of Kanellopoulos and
Karamanlis~\cite[Definition~1.3]{KanellopoulosKaramanlis2020}; it is a
positive integer associated with a chosen cyclic-factor series of a finite
solvable group.

\begin{theorem}
\label{thm:appendix-112233}
There exist a balanced literal role map
\[
 \Psi:[3]^6\longrightarrow[3]^{1344}
\]
of degree $120$ and a solvable coordinate-permutation group $\Omega\leq S_{1344}$ such that
\[
 \Psi\bigl(\operatorname{Perm}(112233)\bigr)
\]
is contained in one $\Omega$-orbit.
Consequently, both the Block Sets Conjecture and its color-independent fixed-degree strengthening hold for the template $112233$.
More precisely, if $h_\Omega$ is an HJ-degree of $\Omega$, then block degree
$120h_\Omega$ suffices for every finite number of colors.
\end{theorem}

\begin{proof}
Put $X=\operatorname{Perm}(112233)$.
For $x\in X$, let $M(x)$ be the perfect matching of $K_6$ whose three edges join positions carrying equal symbols.
Fix a perfect matching $F_0$, and let
\[
 K=\operatorname{Stab}_{S_6}(F_0)\cong C_2\wr S_3.
\]
The group $K$ acts on $X$ by permuting physical positions.

There are five $K$-orbits on $X$, denoted
\[
 E,\quad D,\quad S_1,\quad S_2,\quad S_3.
\]
Here $E$ means $M(x)=F_0$, $D$ means $M(x)\cap F_0=\varnothing$, and $S_a$ means that $M(x)$ and $F_0$ have a unique common edge and that this edge carries the symbol $a$.
Their sizes are
\[
 6,\quad48,\quad12,\quad12,\quad12.
\]
The endpoint swaps and permutations of the three edges of $F_0$ act transitively on each displayed class.
Write
\[
 \tau:X\longrightarrow\Sigma=\{E,D,S_1,S_2,S_3\}
\]
for the orbit map, and put
\[
 \mathcal A=\operatorname{Perm}(EDDDD),
 \qquad
 \mathcal B=\operatorname{Perm}(S_1S_2S_3DD).
\]

We first construct a solvable two-orbit detector for $\mathcal A$ and $\mathcal B$.
Let $\mathbb F_8=\mathbb F_2(\alpha)$, where $\alpha^3+\alpha+1=0$, and consider the four injections $q_t:[5]\to\mathbb F_8$
\begin{equation}
\begin{aligned}
 q_1&=(0,1,\alpha,\alpha+1,\alpha^2),\\
 q_2&=(0,1,\alpha,\alpha^2,\alpha+1),\\
 q_3&=(0,1,\alpha+1,\alpha,\alpha^2+1),\\
 q_4&=(0,1,\alpha+1,\alpha^2+1,\alpha).
\end{aligned}
\label{eq:112233-injections}
\end{equation}
For $z=(z_1,\ldots,z_5)\in\Sigma^5$, form four banks indexed by $\mathbb F_8$.
In bank $t$, put $z_j$ at position $q_t(j)$ and put the fixed symbol $D$ at the three unused positions.
Concatenating the banks defines a literal map
\begin{equation}
 \iota:\Sigma^5\longrightarrow\Sigma^{32}.
 \label{eq:112233-detector}
\end{equation}
Every role occurs once in each bank, so $\iota$ has role degree four.

Let $H=\mathrm{A}\Gamma\mathrm{L}_1(8)$ act on each bank.
Its elements are
\[
 u\longmapsto au^{2^i}+b,
 \qquad
 a\in\mathbb F_8^\times,\quad b\in\mathbb F_8,\quad i\in\{0,1,2\}.
\]
The group $H$ has exactly two orbits on ordered triples of distinct field elements.
After sending the first two entries of $(u,v,w)$ to $(0,1)$, its orbit is determined by
\[
 t=\frac{w-u}{v-u}\in\mathbb F_8\setminus\{0,1\},
\]
up to $t\mapsto t^2$.
The two possibilities are
\begin{equation}
\begin{aligned}
 C_+&=\{\alpha,\alpha^2,\alpha^2+\alpha\},\\
 C_-&=\{\alpha+1,\alpha^2+1,\alpha^2+\alpha+1\}.
\end{aligned}
\label{eq:112233-triple-classes}
\end{equation}

For an increasing triple $I\in\binom{[5]}3$, record by $+$ or $-$ which class in \eqref{eq:112233-triple-classes} contains the normalized ratio of the corresponding points in each of the four rows of \eqref{eq:112233-injections}.
A direct calculation gives
\begin{equation}
\begin{gathered}
\begin{array}{c|ccccc}
 I&123&124&125&134&135\\ \hline
 (\epsilon_1,\epsilon_2,\epsilon_3,\epsilon_4)
   &++--&-++-&+--+&++--&++--
\end{array}
\\[0.6em]
\begin{array}{c|ccccc}
 I&145&234&235&245&345\\ \hline
 (\epsilon_1,\epsilon_2,\epsilon_3,\epsilon_4)
   &-+-+&--++&--++&-+-+&+-+-
\end{array}.
\end{gathered}
\label{eq:112233-sign-table}
\end{equation}
Thus every triple has two banks of each $H$-type.
Reordering a triple applies one of the six fractional-linear transformations to its normalized ratio.
These transformations are defined over $\mathbb F_2$, commute with Frobenius, and therefore either preserve or interchange the two classes in \eqref{eq:112233-triple-classes}.
Consequently the same two--two balance holds for every ordered triple of distinct roles.

Let
\[
 \Gamma=H^4\rtimes A_4
\]
act on the $32$ coordinates of \eqref{eq:112233-detector}, with $A_4$ permuting the four banks.
If $z\in\mathcal A$, every bank of $\iota(z)$ contains one $E$ and seven $D$'s.
Point transitivity of $H$ shows that $\iota(\mathcal A)$ is contained in one $\Gamma$-orbit.

If $z\in\mathcal B$, the positions carrying $S_1,S_2,S_3$, in that order, form an ordered triple of distinct field points in each bank.
By \eqref{eq:112233-sign-table}, precisely two banks have each of the two $H$-types.
Since $A_4$ is transitive on the two-subsets of its four-point set, it aligns the two banks of either type for any two members of $\mathcal B$.
The four independent copies of $H$ then align the ordered triples bank by bank.
Hence $\iota(\mathcal B)$ is also contained in one $\Gamma$-orbit.
These two containing orbits are distinct, since their symbol inventories are respectively
\begin{equation}
 E^4D^{28}
 \qquad\text{and}\qquad
 S_1^4S_2^4S_3^4D^{20}.
 \label{eq:112233-inventories}
\end{equation}

We next use the six one-factorizations of $K_6$.
For completeness, two edge-disjoint perfect matchings extend uniquely to a one-factorization.
Label their alternating $6$-cycle so that the two matchings are
\[
 \{01,23,45\},
 \qquad
 \{12,34,50\}.
\]
The three remaining factors are forced to be
\begin{equation}
 \{03,15,24\},\qquad
 \{04,13,25\},\qquad
 \{02,14,35\}.
 \label{eq:112233-pentad-completion}
\end{equation}
A fixed perfect matching has eight edge-disjoint mates.
Each one-factorization containing it accounts for four such mates, and uniqueness in \eqref{eq:112233-pentad-completion} shows that every perfect matching belongs to exactly two one-factorizations.
Counting incidences gives $15\cdot2/5=6$ one-factorizations.

Denote them by
\[
 \mathcal F_b=\{F_{b,1},\ldots,F_{b,5}\},
 \qquad b\in[6].
\]
For every $(b,j)$, choose $\pi_{b,j}\in S_6$ with $\pi_{b,j}F_{b,j}=F_0$, and define
\begin{equation}
 \lambda_b(x)=
 \bigl(\tau(\pi_{b,1}x),\ldots,\tau(\pi_{b,5}x)\bigr)
 \in\Sigma^5.
 \label{eq:112233-pentad-word}
\end{equation}
If $M(x)\in\mathcal F_b$, then one entry has type $E$ and the other four have type $D$, so $\lambda_b(x)\in\mathcal A$.
If $M(x)\notin\mathcal F_b$, its three edges belong to three distinct factors of $\mathcal F_b$.
These give the types $S_1,S_2,S_3$, while the remaining two factors are disjoint from $M(x)$, so
\begin{equation}
 \lambda_b(x)\in\mathcal A
 \quad\Longleftrightarrow\quad
 M(x)\in\mathcal F_b,
 \qquad
 \lambda_b(x)\in\mathcal B\ \text{otherwise}.
 \label{eq:112233-AB-cases}
\end{equation}
Since every perfect matching belongs to exactly two of the six one-factorizations, precisely two of $\lambda_1(x),\ldots,\lambda_6(x)$ belong to $\mathcal A$.

Choose representatives $r_\sigma\in X$ for the five fibres of $\tau$.
For $b\in[6]$, lift $\iota(\lambda_b(x))$ from orbit labels to actual $X$-words as follows.
At a coordinate of \eqref{eq:112233-detector} occupied by its $j$th formal role, put the actual word $\pi_{b,j}x\in X$; at each fixed $D$-coordinate put $r_D$.
This defines $L_b(x)\in X^{32}$ and gives
\begin{equation}
 \tau^{32}(L_b(x))=\iota(\lambda_b(x)).
 \label{eq:112233-lift}
\end{equation}

Let
\[
 \Lambda=K^{32}\rtimes\Gamma
\]
act imprimitively on the $32$ copies of $X$.
If two words in \eqref{eq:112233-lift} have type words in the same $\Gamma$-orbit, first use $\Gamma$ to align their type coordinates and then use the $32$ independent copies of $K$ to align the actual $X$-letters.
Thus all lifts corresponding to $\mathcal A$ lie in one $\Lambda$-orbit, and all lifts corresponding to $\mathcal B$ lie in another.

Fix $z^0\in\mathcal B$, and define a fixed $\mathcal B$-chunk
\[
 L_0=\bigl(r_{\iota(z^0)_s}:s\in[32]\bigr)\in X^{32}.
\]
Finally, put
\begin{equation}
 \Psi(x)=
 \bigl(L_1(x),\ldots,L_6(x),L_0\bigr)
 \in(X^{32})^7\subseteq[3]^{1344}.
 \label{eq:112233-final-map}
\end{equation}
By \eqref{eq:112233-AB-cases}, the seven outer chunks of $\Psi(x)$ have orbit profile $\mathcal A^2\mathcal B^5$.

Let
\[
 P_7=\{u\mapsto2^iu+b:i\in\{0,1,2\},\ b\in\mathbb F_7\}
 \cong C_7\rtimes C_3.
\]
This group is transitive on the unordered pairs of $\mathbb F_7$, since $\{\pm1,\pm2,\pm4\}=\mathbb F_7^\times$.
Therefore $P_7$ aligns the two $\mathcal A$-positions in any two words of the form \eqref{eq:112233-final-map}, and the seven independent copies of $\Lambda$ then align the corresponding outer chunks.
Consequently
\begin{equation}
 \Psi(X)\text{ is contained in one orbit of }
 \Omega=\Lambda^7\rtimes P_7.
 \label{eq:112233-one-orbit}
\end{equation}

Every group used here is solvable.
Indeed,
\[
 K\cong C_2^3\rtimes S_3,
 \qquad
 H\cong(\mathbb F_8,+)\rtimes(C_7\rtimes C_3),
\]
while $A_4$ and $P_7$ are solvable, and solvability is preserved by direct powers and extensions.
Every action is a literal permutation of physical coordinates: $K$ acts inside the six coordinates of an $X$-letter, $\Gamma$ permutes the $32$ $X$-letters of a chunk, and $P_7$ permutes the seven outer chunks.
Thus $\Omega$ in \eqref{eq:112233-one-orbit} is a concrete subgroup of $S_{1344}$.

It remains to check balance.
In every $L_b(x)$, each of the five roles in \eqref{eq:112233-pentad-word} occurs once in each of the four field banks.
Hence $L_b(x)$ contains $4\cdot5=20$ nonconstant $X$-letters, all of the form $\pi_{b,j}x$.
The six nonconstant outer chunks contain $120$ such letters.
Every $\pi_{b,j}x$ is a permutation of the six positions of $x$ and hence contains every formal role of $x$ exactly once.
Thus every source role occurs exactly $120$ times in \eqref{eq:112233-final-map}.
The seventh chunk and all padding coordinates are fixed, so $\Psi$ extends by the same coordinate formula to a balanced literal role map on $[3]^6$.
Its scalar length is $7\cdot32\cdot6=1344$; $720$ scalar coordinates are active and $624$ are fixed.

Finally, let
\[
 Y=\Omega\Psi(X)\subseteq[3]^{1344}.
\]
This is a transitive $\Omega$-set.
Let $h_\Omega$ be an HJ-degree of the solvable group $\Omega$.
Theorem~1.5 of Kanellopoulos and
Karamanlis~\cite[Theorem~1.5]{KanellopoulosKaramanlis2020}, applied to the
one-orbit action $\Omega\curvearrowright Y$, says that for every $r$ there is
an $N$ such that every $r$-coloring of $Y^N$ admits an $\Omega$-variable word
$W$ of total degree $h_\Omega$ for which $W(Y)$ is monochromatic.

Given an $r$-coloring of $[3]^{1344N}$, restrict it to the concatenated copy of $Y^N$ and apply this theorem.
Then $x\mapsto W(\Psi(x))$ is monochromatic on $X$.
Every active macrocoordinate has the form $g\Psi(x)$ for a fixed $g\in\Omega$.
Since $g$ is a physical coordinate permutation, this is again a balanced literal role map of degree $120$.
There are $h_\Omega$ active macrocoordinates, so after flattening every source role has degree $120h_\Omega$.
This is exactly a monochromatic block set with template $112233$ in the sense of Ivan, Leader, and Walters~\cite{IvanLeaderWalters2026}.
The degree is fixed before the number of colors is chosen, proving their fixed-degree conjecture for this template and, a fortiori, the ordinary Block Sets Conjecture.
\end{proof}

\subsection{An affine-plane family and all three-letter templates of length at most six}

The preceding construction is complemented by a smaller one based on a large set of affine planes.

\begin{proposition}
\label{prop:appendix-affine-template-family}
For every $1\leq k\leq6$, there is a balanced literal role map of degree seven
\[
 \Theta_k:[4]^{k+3}\longrightarrow[4]^{63}
\]
whose restriction to $\operatorname{Perm}(1234^k)$ is contained in one orbit of a solvable coordinate-permutation group.
Consequently the color-independent fixed-degree Block Sets Conjecture holds for every template $1234^k$ in this range.
\end{proposition}

\begin{proof}
Label the nine points by $0,\ldots,8$.
The following seven rows are line sets of seven copies of $\operatorname{AG}(2,3)$; together they partition all $\binom93=84$ triples:
\begin{equation}
\begin{array}{c|llllllllllll}
0&012&036&048&057&138&147&156&237&246&258&345&678\\
1&013&026&045&078&125&148&167&238&247&346&357&568\\
2&016&023&047&058&128&134&157&245&267&356&378&468\\
3&017&025&038&046&124&135&168&236&278&347&458&567\\
4&014&028&035&067&126&137&158&234&257&368&456&478\\
5&015&027&034&068&123&146&178&248&256&358&367&457\\
6&018&024&037&056&127&136&145&235&268&348&467&578
\end{array}
\label{eq:appendix-seven-planes}
\end{equation}
The certificate can be checked directly: in each row every pair occurs once, the twelve triples split into four parallel classes of three, and no triple is repeated between rows.
Thus each row is an affine plane and the count $7\cdot12=84$ completes the partition check.

Put $s=k+3\leq9$ and regard the source roles as the points $0,\ldots,s-1$.
For each row choose an isomorphism from a fixed copy of $\operatorname{AG}(2,3)$ to that row.
In the corresponding nine-coordinate bank, place every source role at its inverse image under this isomorphism and fill the remaining $9-s$ positions with the fixed symbol $4$.
Concatenating the seven banks gives $\Theta_k$.
Each source role occurs once per bank, so the map is literal and has degree seven.

In a word from $\operatorname{Perm}(1234^k)$, the positions of the three singleton symbols $1,2,3$ form an ordered triple of distinct source roles.
By \eqref{eq:appendix-seven-planes}, this triple is collinear in exactly one bank and noncollinear in the other six.
The solvable group $A=\operatorname{AGL}(2,3)$ is transitive both on ordered collinear triples and on ordered noncollinear triples of the affine plane.
Thus $A^7$ aligns two output words bank by bank after a cyclic permutation of the banks has aligned their unique collinear banks.
It follows that
\[
 \Theta_k\bigl(\operatorname{Perm}(1234^k)\bigr)
\]
lies in one orbit of the solvable group $A^7\rtimes C_7$.

Apply the uniform Hales--Jewett theorem for solvable groups as in the final
paragraph of the proof of \cref{thm:appendix-112233}.
If $h_k$ is an HJ-degree of the solvable group $A^7\rtimes C_7$, then its
action on the orbit containing the displayed image has one orbit, and the
resulting block degree is $7h_k$, independent of the palette.
\end{proof}

\begin{corollary}
\label{cor:appendix-short-three-letter-templates}
Every template using at most three symbols and having total length at most six satisfies the color-independent fixed-degree Block Sets Conjecture.
\end{corollary}

\begin{proof}
One-symbol templates are trivial.
Binary templates have fixed degree one by the ordinary finite Ramsey theorem, and for length at most four every permutation family is an orbit of the solvable group $S_n$.
For length five, the two genuine three-symbol multiplicity patterns $(3,1,1)$ and $(2,2,1)$ are coarsenings of $12344$.
For length six, the patterns $(4,1,1)$ and $(3,2,1)$ are coarsenings of $123444$, while $(2,2,2)$ is covered by \cref{thm:appendix-112233}.
The property is inherited by coarsening a template, which proves the assertion.
\end{proof}

\subsection{The distinct-letter template \texorpdfstring{$12345$}{12345}}
\label{subsec:appendix-12345}

The affine-plane certificate for $12344$ can be lifted through a small
pair-splitting gadget.  This gives the smallest distinct-letter template beyond
the range in which the full symmetric group is solvable.

\begin{theorem}
\label{thm:appendix-12345}
There exist a balanced literal role map
\[
 \Psi:[5]^5\longrightarrow[5]^{504}
\]
of degree $35$ and a solvable coordinate-permutation group
$\Omega\leq S_{504}$ such that
\[
 \Psi\bigl(\operatorname{Perm}(12345)\bigr)
\]
is contained in one $\Omega$-orbit.
Consequently, both the Block Sets Conjecture and its color-independent
fixed-degree strengthening hold for the template $12345$.
More precisely, if $h_\Omega$ is an HJ-degree of $\Omega$, then block degree
$35h_\Omega$ suffices for every finite number of colors.
\end{theorem}

\begin{proof}
Put $X=\operatorname{Perm}(12345)$ and define
\[
 q:[5]\longrightarrow[4],\qquad
 q(1)=1,\quad q(2)=2,\quad q(3)=3,\quad q(4)=q(5)=4.
\]
We also write $q(x)=(q(x_1),\ldots,q(x_5))$ for the coordinatewise
extension.  For $x=(x_1,\ldots,x_5)\in[5]^5$ and $j\in[5]$, form the eight-letter
literal word
\begin{equation}
 F_j(x)=\bigl((x_i)_{i\in[5]\setminus\{j\}}\mid x_j,1,2,3\bigr),
 \label{eq:appendix-12345-local-word}
\end{equation}
where the entries in the first block are put in any fixed order.
Let
\[
 H=(S_4\times S_4)\rtimes C_2\cong S_4\wr C_2
\]
act by arbitrary coordinate permutations within the two displayed
four-coordinate blocks and by interchanging the blocks.  The group $H$ is
solvable, and the $H$-orbit of a word is determined exactly by the unordered
pair of its two block multisets.

If $a=x_j$, these unordered pairs are
\begin{equation}
\begin{array}{c|c}
a&\text{two block multisets}\\
\hline
1&\{2345,1123\}\\
2&\{1345,1223\}\\
3&\{1245,1233\}\\
4&\{1235,1234\}\\
5&\{1234,1235\}.
\end{array}
\label{eq:appendix-12345-local-types}
\end{equation}
Thus the first three values give three distinct $H$-orbits, while the last
two give the same fourth orbit.  Label these four local orbit types by
$1,2,3,4$, respectively, and write $\tau_H$ for the resulting orbit-label
map on their union.
Then
\begin{equation}
 \tau_H(F_j(x))=q(x_j)
 \qquad(x\in X,\ j\in[5]).
 \label{eq:appendix-12345-type-identity}
\end{equation}

Use \cref{prop:appendix-affine-template-family} with $k=2$.  It supplies a
balanced literal map
\[
 \Theta_2:[4]^5\longrightarrow[4]^{63}
\]
of degree seven such that
$\Theta_2(\operatorname{Perm}(12344))$ is contained in one orbit of the
solvable group
\[
 \Gamma=\operatorname{AGL}(2,3)^7\rtimes C_7\leq S_{63}.
\]
Every coordinate of $\Theta_2(z)$ is either one of the five projections
$z_j$ or the fixed symbol $4$.  Choose a fixed representative of the fourth
local $H$-type, for instance
\[
 r_4=(1,2,3,5\mid4,1,2,3)\in[5]^8.
\]
At a macrocoordinate $t\in[63]$, replace the formula $z_j$ in $\Theta_2$
by the eight-coordinate word $F_j(x)$, and replace a fixed $4$ by $r_4$.
Concatenating the resulting $63$ blocks defines
\[
 \Psi:[5]^5\longrightarrow[5]^{8\cdot63}.
\]
By \eqref{eq:appendix-12345-type-identity}, the word of local $H$-orbit
types of $\Psi(x)$ is exactly
\begin{equation}
 \tau_H^{\times63}(\Psi(x))=\Theta_2(q(x))
 \qquad(x\in X).
 \label{eq:appendix-12345-outer-types}
\end{equation}
Since $q(x)\in\operatorname{Perm}(12344)$, the group $\Gamma$ aligns the
type words in \eqref{eq:appendix-12345-outer-types} for any two elements of
$X$.  After this alignment, independent copies of $H$ align the actual
eight-coordinate blocks.  Hence all the words $\Psi(x)$ lie in one orbit of
\[
 \Omega=H^{63}\rtimes\Gamma\leq S_{504}.
\]
This group is solvable and acts only by permutations of physical
coordinates; no permutation of the alphabet is being used.

It remains to verify balance.  Each word $F_j(x)$ contains every one of the
five source projections exactly once.  The map $\Theta_2$ has five formal
roles, each occurring seven times, and therefore has $35$ nonconstant
macrocoordinates.  Consequently every source projection occurs exactly
$35$ times in $\Psi$, while the remaining coordinates are fixed.  Thus
$\Psi$ is a balanced literal role map of degree $35$ and has scalar length $504$.

Finally apply the uniform Hales--Jewett theorem for solvable group actions
to the transitive $\Omega$-set
\[
 Y=\Omega\Psi(X).
\]
Kanellopoulos and Karamanlis~\cite[Theorem~1.5]{KanellopoulosKaramanlis2020}
show that, for every finite number $r$ of colors, there is an $N$ such that
every $r$-coloring of $Y^N$ has a monochromatic $\Omega$-variable word of
total degree $h_\Omega$.  Given an $r$-coloring of $[5]^{504N}$, restrict it
to the concatenated copy of $Y^N$ and apply this theorem.  Every active
macrocoordinate has the form $g\Psi(x)$ for a fixed $g\in\Omega$ and is
therefore again a balanced literal role map of degree $35$, since $g$
merely permutes physical coordinates.  Flattening the $h_\Omega$ active
macrocoordinates gives a monochromatic $12345$ block set of degree
$35h_\Omega$.  This degree is independent of $r$.
\end{proof}

The local construction in \eqref{eq:appendix-12345-local-word} has a useful
general form.  For $2\leq m\leq5$, the group $S_{m-1}\wr C_2$ acting on
\[
 \bigl((x_i)_{i\ne j}\mid x_j,1,\ldots,m-2\bigr)
\]
fuses precisely the last two values and separates the first $m-2$ values.
Thus a solvable-orbit certificate for
$12\cdots(m-2)(m-1)^2$ lifts to one for $12\cdots m$, where the initial
string $12\cdots(m-2)$ is interpreted as empty when $m=2$.  At $m=6$ this
particular mechanism calls for $S_5\wr C_2$, which is not solvable; this
explains the natural endpoint of the argument but is not an obstruction to
a different construction.

\begin{corollary}
\label{cor:appendix-all-templates-length-five}
Every template of total length at most five satisfies the
color-independent fixed-degree Block Sets Conjecture.
\end{corollary}

\begin{proof}
For a template of length $n\leq4$, its permutation family is a single orbit
of the solvable group $S_n$ under the identity role map.  Every template of length five is obtained from
$12345$ by identifying some alphabet symbols.  Apply the same identification
to the fixed and variable letters in the map of
\cref{thm:appendix-12345}.  Coordinate permutations commute with this
identification, and every source projection still occurs $35$ times, so the
result remains a balanced one-orbit certificate.  The uniform
Hales--Jewett argument proves the assertion.
\end{proof}

The unsuccessful extension from $12345$ to $123456$ also yielded a general
restriction on the solvable-orbit method.  It applies to padded as well as
unpadded literal maps and is most naturally expressed through the incidence
structure of all translates of the target copy.

\begin{proposition}[Flag divisibility for literal permutation copies]
\label{prop:appendix-literal-flag-divisibility}
Let $n\geq2$ and let $\Psi:[n]^n\to[n]^N$ be a map such that every output
coordinate is either fixed or one of the projections $x_i$, with every
projection occurring the same positive number of times.  Put
\[
 T=\Psi\bigl(\operatorname{Perm}(12\cdots n)\bigr).
\]
Suppose that $T$ is contained in an orbit $O$ of a coordinate-permutation
group $\Omega\leq S_N$.  Let
\[
 \mathcal B=\{gT:g\in\Omega\},
 \qquad b=|\mathcal B|,
\]
where repeated translates are counted only once.  Every point of $O$ lies
in the same number $\lambda$ of members of $\mathcal B$.
Let $P\leq S_n$ be the permutation group induced on the $n$ source-role
classes by the setwise stabilizer $\Omega_{\{T\}}$.  Then
\begin{equation}
 [S_n:P]\mid\lambda,
 \qquad
 |O|=\frac{n!b}{\lambda},
 \label{eq:appendix-literal-flag-divisibility}
\end{equation}
and
\begin{equation}
 b-1\geq n(\lambda-1),
 \qquad
 |O|\geq n!\left(n-\frac{n-1}{\lambda}\right).
 \label{eq:appendix-literal-orbit-bound}
\end{equation}
If $\Omega$ is solvable, then $P$ is solvable.  In particular, for $n=6$
and solvable $\Omega$,
\[
 \lambda\geq10
 \qquad\text{and}\qquad
 |O|\geq3960.
\]
\end{proposition}

\begin{proof}
For $i\in[n]$, let $A_i$ be the set of coordinates on which $\Psi(x)$ is
the projection $x_i$, and let $F$ be the set of fixed coordinates.  The
sets $A_i$ are nonempty and have equal size.  The set $F$ consists exactly
of the coordinates whose value is constant on $T$, since every projection
assumes all $n$ symbols on $\operatorname{Perm}(12\cdots n)$.  Hence every
element of $\Omega_{\{T\}}$ preserves $F$.
In particular, $\Psi$ is injective on
$\operatorname{Perm}(12\cdots n)$, so $|T|=n!$.

On the remaining coordinates define $u\sim v$ when the $u$th and $v$th
coordinates agree in every word of $T$.  Its equivalence classes are
exactly $A_1,\ldots,A_n$: two coordinates in one $A_i$ always agree,
whereas two different projections have different values in every
permutation word.  Thus $\Omega_{\{T\}}$ permutes the $A_i$.  After
identifying $\operatorname{Perm}(12\cdots n)$ with $S_n$, its induced
action on $T$ consequently has the form
\[
 \sigma\longmapsto\sigma\rho
 \qquad(\rho\in P),
\]
up to the immaterial inverse convention.  Thus the induced action is by
right multiplication by elements of $P$, and its orbits are the right
cosets of $P$.  In particular, the setwise stabilizer has
\[
 m=[S_n:P]
\]
orbits on $T$, all of size $|P|$.  Moreover, $P$ is a quotient of a
subgroup of $\Omega$, and is therefore solvable whenever $\Omega$ is.

Consider the flags $(u,B)$ with $u\in B\in\mathcal B$.  The $\Omega$-orbits
on flags correspond to the $\Omega_{\{T\}}$-orbits on $T$, and hence there
are $m$ of them.  Each such flag orbit contains $|P|$ flags over each
block.  If one flag orbit contributes $\mu$ incident blocks through each
point of $O$, double counting gives
\[
 b|P|=|O|\mu.
\]
Counting all flags gives
\[
 bn!=|O|\lambda.
\]
Consequently
\[
 \mu=\frac{\lambda|P|}{n!}=\frac{\lambda}{m}.
\]
Since $\mu$ is an integer, $m\mid\lambda$, and the total flag count gives
the second identity in \eqref{eq:appendix-literal-flag-divisibility}.

We next show that two distinct translates of $T$ meet in at most
$(n-1)!$ words.  More generally, compare two balanced literal copies of
the same positive degree.  If their fixed-coordinate sets differ, then
some coordinate is a projection $x_i$ in one copy and is fixed, say to
$a$, in the other.  A common word must satisfy $x_i=a$, leaving at most
$(n-1)!$ specializations.  If the fixed-coordinate sets agree but their
fixed words disagree, the copies are disjoint.

It remains to consider two copies with the same fixed set and fixed word.
Compare their two partitions of the variable coordinates into $n$ role
classes.  In a common word, a class of the first partition cannot meet two
classes of the second: it is constant, while distinct roles in a
permutation word have distinct values.  Since all classes in both
partitions have the same positive size, each class of one partition must
equal a class of the other.  The partitions agree up to relabelling, and
the two literal copies are equal.  This proves the intersection bound.

Fix $T\in\mathcal B$.  Counting pairs $(u,B)$ with
$u\in T\cap B$ and $B\neq T$ now yields
\[
 n!(\lambda-1)
 =\sum_{B\in\mathcal B\setminus\{T\}}|T\cap B|
 \leq(b-1)(n-1)!.
\]
Hence $b-1\geq n(\lambda-1)$.  Together with
$|O|=n!b/\lambda$, this gives
\[
 |O|\geq
 \frac{n!\bigl(n(\lambda-1)+1\bigr)}{\lambda}
 =n!\left(n-\frac{n-1}{\lambda}\right).
\]

Finally suppose that $n=6$ and $\Omega$ is solvable.  Every solvable
subgroup of $S_6$ has order at most $72$.  To see this, an elementary
orbit-partition check gives at most $48$ in the intransitive case: an orbit
of size five contributes at most $20$, since a solvable transitive group of
degree five embeds in $\operatorname{AGL}(1,5)$, while among the remaining
partitions the maximum is $4+2$, giving $24\cdot2=48$.  A transitive
imprimitive group
embeds in $S_2\wr S_3$ or $S_3\wr S_2$, of orders $48$ and $72$,
respectively.  A solvable primitive permutation group has a minimal normal
subgroup which is elementary abelian and transitive, hence regular, so its
degree is a prime power.  Thus there is no such group of degree six.  The bound $72$ is attained
by $S_3\wr S_2$.  Thus $|P|\leq72$, whence $[S_6:P]\geq10$.
Divisibility gives $\lambda\geq10$, and
\[
 |O|\geq720\left(6-\frac5{10}\right)=3960.
\]
\end{proof}

\begin{corollary}[Pure-projection obstruction]
\label{cor:appendix-pure-projection-obstruction}
If $\Psi$ in
\cref{prop:appendix-literal-flag-divisibility} has no fixed coordinates,
then distinct members of $\mathcal B$ are disjoint, $\lambda=1$, and
$P=S_n$.  Consequently, for $n\geq5$, a positive-degree pure-projection
literal copy of $\operatorname{Perm}(12\cdots n)$ cannot be contained in
one orbit of a solvable coordinate-permutation group.
\end{corollary}

\begin{proof}
With no fixed coordinates, the intersection argument in the proof of
\cref{prop:appendix-literal-flag-divisibility} shows that two translates
are either equal or disjoint.  Thus $\lambda=1$.  Since
$[S_n:P]\mid\lambda$, we have $P=S_n$, which is not solvable for $n\geq5$.
\end{proof}

This obstructs the solvable one-orbit literal-certificate method, not the
Block Sets Conjecture itself.  In particular, it does not decide the
$123456$ case.  For $123456$, it shows only that a balanced literal
certificate of this kind must use fixed padding and that its containing
orbit must have at least $3960$ states.

\section{A sharp obstruction for full simplex products}
\label{sec:appendix-simplex-products}

The orbit-design constructions above use restricted products.
The next theorem shows that this restriction is essential for an unbalanced constant-weight layer: allowing arbitrary full products of regular-simplex alphabets cannot make the radius approach the intrinsic circumradius.

For $2\leq k\leq v-2$, write
\[
 J(v,k)=\{\mathbf 1_A:A\in\tbinom{[v]}k\}.
\]
Its intrinsic squared circumradius is
\begin{equation}
 \rho(J(v,k))^2=\frac{k(v-k)}v.
 \label{eq:appendix-hypersimplex-radius}
\end{equation}

\begin{theorem}[Full simplex-product optimum]
\label{thm:appendix-simplex-product-optimum}
Suppose that an isometric copy of $J(v,k)$ is contained in a finite full Cartesian product of centered regular-simplex alphabets.
If the full product has radius $R$, then
\begin{equation}
 R^2\geq\frac v4.
 \label{eq:appendix-simplex-product-bound}
\end{equation}
Equality is attained by the ordinary binary cube.
Moreover, every nonconstant coordinate of any such product representation partitions $J(v,k)$ into a coordinate star
\[
 \{A\in\tbinom{[v]}k:i\in A\}
\]
and its complement, and equality in \eqref{eq:appendix-simplex-product-bound} forces every nonconstant factor to be binary.
Consequently
\begin{equation}
 R^2-\rho(J(v,k))^2
 \geq\frac{(v-2k)^2}{4v}.
 \label{eq:appendix-simplex-product-gap}
\end{equation}
\end{theorem}

\begin{proof}
Center the hypersimplex vertices by putting
\[
 x_A=\mathbf 1_A-\frac kv\mathbf 1.
\]
Their Gram matrix $Q$ has range equal to the first nonconstant Johnson eigenspace $E_1$ and has rank $v-1$.

Center the displayed copy in the product at its centroid.
Its Gram matrix is still $Q$ and decomposes as
\[
 Q=\sum_jQ_j,
\]
where $Q_j\succeq0$ is the centered Gram contribution of the $j$th factor.
If $z\in\ker Q$, then
\[
 0=z^{\mathsf T}Qz=\sum_jz^{\mathsf T}Q_jz,
\]
so $z\in\ker Q_j$ for every $j$.
Hence
\begin{equation}
 \operatorname{range}Q_j\subseteq E_1.
 \label{eq:appendix-factor-range}
\end{equation}

Let $f_j$ color the vertices of the slice by the simplex letter used in factor $j$.
Every used subset of the vertices of a regular simplex is affinely independent.
Consequently the centered factor matrix has rank one less than the number of used letters, and its column space is exactly the span of the centered color-class indicators of $f_j$.
In particular, the range of $Q_j$ contains every such centered indicator.
By \eqref{eq:appendix-factor-range}, every uncentered color-class indicator belongs to $E_0\oplus E_1$.

We use the following elementary Boolean degree-one fact on the slice.
If
\[
 h(A)=c+\sum_{i\in A}a_i\in\{0,1\}
 \qquad(A\in\tbinom{[v]}k),
\]
then $h$ is constant, the indicator of a coordinate star, or the complement of such an indicator.
Indeed, comparing $S\cup\{i\}$ and $S\cup\{j\}$ for a common $(k-1)$-set $S$ gives
\[
 a_i-a_j\in\{-1,0,1\}
 \qquad(i\ne j).
\]
Thus the coefficients have at most two levels, and if there are two then they differ by one.
If $I$ is the set of indices at the higher level, then $h(A)$ is an affine
function of $|A\cap I|$.
The possible intersection sizes form the integer interval
\[
 \max\{0,k-(v-|I|)\}\leq |A\cap I|\leq\min\{k,|I|\}.
\]
When $2\leq|I|\leq v-2$, the hypotheses $2\leq k\leq v-2$ imply that this
interval contains at least three consecutive integers, contradicting the
Boolean range.
The remaining cases give a star, its complement, or a constant.

If a factor is nonconstant, its nonempty color classes partition the slice and are therefore pairwise disjoint.
Two distinct stars intersect because $k\geq2$, and two distinct complements intersect because $v-k\geq2$.
The star at $i$ also intersects the complement of the star at $j$ whenever $i\ne j$.
It follows that the factor uses exactly two letters and tests membership of a single ground coordinate $i$.

Let $d_j$ be the squared edge length of factor $j$, and let $W_i$ be the sum of $d_j$ over the factors testing coordinate $i$.
For distinct $i,\ell$, choose a $(k-1)$-set $S$ disjoint from them and compare $S\cup\{i\}$ with $S\cup\{\ell\}$.
Their squared distance is two, so
\[
 W_i+W_\ell=2.
\]
This holds for every pair of distinct indices, whence $W_i=1$ for all $i$.

A full regular simplex with $q_j$ vertices and squared edge length $d_j$ has squared radius
\[
 d_j\frac{q_j-1}{2q_j}\geq\frac{d_j}{4},
\]
with equality only when $q_j=2$.
Constant factors contribute nonnegatively to the full product radius.
Summing over the nonconstant factors therefore gives
\[
 R^2\geq\frac14\sum_{i=1}^vW_i=\frac v4.
\]
The centered binary cube with unit squared edge length attains equality, and equality forces every nonconstant factor to be binary.
Combining the bound with \eqref{eq:appendix-hypersimplex-radius} gives \eqref{eq:appendix-simplex-product-gap}.
\end{proof}

For the first full layer not covered by \cref{thm:small-johnson-layers}, namely $J(10,4)$, the theorem gives
\[
 R^2\geq\frac{10}{4}=2.5,
 \qquad
 \rho(J(10,4))^2=\frac{24}{10}=2.4.
\]
Thus no sequence of full regular-simplex products can resolve this layer at near-circumradius.
Any successful near-circumradius construction for this layer must use restricted products, orbit fusion, or genuinely different geometry.

\section{Further Kneser--shift results}
\label{sec:appendix-kneser-shift}

\subsection{The two-color Kneser--shift problem for arbitrary uniformity}
\label{subsec:two-color-composite-kneser-shift}

The proof of the Kneser--shift theorem uses the primality of $p$, and we do
not know how to remove it for an arbitrary number of colors.
For two colors, however, the composite case admits a short elementary
argument.
For every integer $p\geq2$, define $C(p,r)$ exactly as after
\cref{thm:kneser-shift}, without imposing that $p$ is prime, and put
$C(p,r)=\infty$ if no such threshold exists.

\begin{lemma}[Two-color uniformity reduction]
\label{lem:two-color-uniformity-reduction}
Let $p\geq3$ and $n\geq pk+1$.
Every proper two-coloring of $\KSh_p(n,k)$ induces a proper two-coloring of
$\KSh_{p-1}(n,k)$.
\end{lemma}

\begin{proof}
Suppose that $c$ is a proper two-coloring of $\KSh_p(n,k)$.
Fix an ordered $(p-2)$-tuple
\[
 W=(A_1,\ldots,A_{p-2})
\]
of pairwise disjoint $k$-sets.
We first show that $c(W,X)$ is independent of the choice of the $k$-set
$X$ disjoint from every member of $W$.

The $k$-subsets of
\[
 U=[n]\setminus(A_1\cup\cdots\cup A_{p-2})
\]
form a connected Johnson graph, in which two sets are adjacent when they
differ in one element.
It is therefore enough to compare two such sets $X,X'$ satisfying
$|X\cap X'|=k-1$.
Their union has size $k+1$, and
\[
 |U\setminus(X\cup X')|
 =n-(p-1)k-1\geq k.
\]
Choose a $k$-set $Y$ in this complement.
The vertex $(Y,W)$ is adjacent in $\KSh_p(n,k)$ to both $(W,X)$ and
$(W,X')$.
As only two colors are available, both latter vertices have the unique color
different from the color of $(Y,W)$.
Consequently
\[
 c(W,X)=c(W,X').
\]
Connectivity of the Johnson graph proves the claimed independence.

We may thus define
\[
 \bar c(A_1,\ldots,A_{p-2})
 =c(A_1,\ldots,A_{p-2},X),
\]
where $X$ is any disjoint $k$-set.
To see that $\bar c$ is proper, take pairwise disjoint
$A_1,\ldots,A_{p-1}$ and then choose a $k$-set $A_p$ disjoint from all of
them; this is possible because
\[
 n-(p-1)k\geq k+1.
\]
The definition gives
\[
 \bar c(A_1,\ldots,A_{p-2})=c(A_1,\ldots,A_{p-1})
\]
and
\[
 \bar c(A_2,\ldots,A_{p-1})=c(A_2,\ldots,A_p).
\]
The two vertices on the right are adjacent in $\KSh_p(n,k)$, so their
colors are different.
This is precisely the properness of $\bar c$ on
$\KSh_{p-1}(n,k)$.
\end{proof}

\begin{proposition}[Exact two-color threshold]
\label{prop:exact-two-color-kneser-shift}
For every integer $p\geq2$,
\[
 C(p,2)=
 \begin{cases}
  0,&p\text{ is odd},\\
 1,&p\text{ is even}.
 \end{cases}
\]
Equivalently, in the nonempty-edge range $n\geq pk$, the graph
$\KSh_p(n,k)$ is bipartite if and only if $n=pk$ and $p$ is even.
In particular, the Kneser--shift conclusion holds for every composite $p$
when there are two colors.
\end{proposition}

\begin{proof}
First suppose that $n\geq pk+1$ and that $p\geq3$.
If $\KSh_p(n,k)$ had a proper two-coloring, repeated application of
\cref{lem:two-color-uniformity-reduction} would give a proper two-coloring
of
\[
 \KSh_2(n,k)=\KG(n,k).
\]
This is impossible, since Lov\'asz's theorem~\cite{Lovasz1978} gives
\[
 \chi(\KG(n,k))=n-2k+2\geq3.
\]
For $p=2$, the same conclusion for $n\geq2k+1$ follows directly from the
same formula.

It remains to inspect zero excess, so let $n=pk$.
For $p\geq3$, the complement of the union of a vertex
$(A_1,\ldots,A_{p-1})$ is the unique remaining $k$-set $A_p$.
Its two neighbors are obtained by a left or right cyclic shift of the
ordered partition $(A_1,\ldots,A_p)$.
Hence every connected component of $\KSh_p(pk,k)$ is a cycle of length
$p$.
For $p=2$, the graph $\KSh_2(2k,k)=\KG(2k,k)$ is a matching, each $k$-set
being adjacent only to its complement.
Thus zero excess is two-colorable exactly when $p$ is even, while for odd
$p$ it is already a disjoint union of odd cycles.
Combining this observation with the positive-excess argument proves the
formula.
\end{proof}

There is also a useful projection which, for even uniformity, makes the
period-two obstruction in the history proof completely explicit.

\begin{proposition}[Alternating-union projection]
\label{prop:alternating-union-projection}
 Let $p\geq2$ and put $s=\lfloor p/2\rfloor$.
For every $n\geq pk$, the map
\[
 \pi(A_1,\ldots,A_{p-1})
 =\bigcup_{\substack{1\leq i\leq p-1\\ i\text{ odd}}}A_i
\]
is a graph homomorphism
\[
 \KSh_p(n,k)\longrightarrow\KG(n,sk).
\]
Consequently
\[
 \chi(\KSh_p(n,k))
 \leq \chi(\KG(n,sk))
 =n-2\lfloor p/2\rfloor k+2.
\]
In particular, for every even $p$ and every $r\geq2$,
\[
 C(p,r)\geq r-1.
\]
\end{proposition}

\begin{proof}
If
\[
 (A_1,\ldots,A_{p-1})
 \quad\text{and}\quad
 (A_2,\ldots,A_p)
\]
are adjacent, then $A_1,\ldots,A_p$ are pairwise disjoint.
Their images under $\pi$ are respectively
\[
 \bigcup_{\substack{1\leq i\leq p-1\\ i\text{ odd}}}A_i
 \quad\text{and}\quad
 \bigcup_{\substack{2\leq i\leq p\\ i\text{ even}}}A_i.
\]
These are disjoint $sk$-sets, which proves the homomorphism claim.
Lov\'asz's formula gives
\[
 \chi(\KG(n,sk))=n-2sk+2
 =n-2\lfloor p/2\rfloor k+2.
\]
If $p$ is even, this is $n-pk+2$.
Finally, in that case take $n=pk+r-2$.
The displayed homomorphism supplies a proper coloring with at most $r$
colors at excess $r-2$, so the threshold cannot be smaller than $r-1$.
\end{proof}

The alternating-union projection is only an upper bound for the chromatic
number, so it does not settle the composite-uniformity problem for
$r\geq3$.
It does, however, identify the period-two history vectors
$(E,F,E,F)$ with a genuine global feature of even Kneser--shift graphs,
rather than merely an artefact of the proof.

\subsection{A sharper quantitative Kneser--shift bound}
\label{sec:appendix-sharper-ksh-bound}

The rank-sum labeling in the proof of the Kneser--shift theorem can be compressed.
This improves the leading constant in its quantitative upper bound.

We use the following version of the Kneser--Tucker lemma, which is exactly \cref{lem:kneser-tucker}.
Let $p$ be prime and let $\P_p(n,k)$ be the poset of nonzero $p$-tuples
$\mathcal A=(\mathcal A_1,\ldots,\mathcal A_p)$ of pairwise cross-disjoint families of $k$-sets, ordered coordinatewise.
If
\[
 \mu=(\mu_1,\mu_2):\P_p(n,k)\longrightarrow\mathbb Z_p\times[m]
\]
is equivariant under cyclic coordinate shift and satisfies
\begin{equation}
 \mathcal A\leq\mathcal B\ \text{ and }\
 \mu_2(\mathcal A)=\mu_2(\mathcal B)
 \quad\Longrightarrow\quad
 \mu_1(\mathcal A)=\mu_1(\mathcal B),
 \label{eq:appendix-ksh-tie}
\end{equation}
then
\begin{equation}
 n-pk<m-1.
 \label{eq:appendix-ksh-tucker-conclusion}
\end{equation}
This is the only topological input below.

Recall the finite posets used for the recursive histories:
\[
 T_{p-1}=[r]\quad\text{with the antichain order},
 \qquad
 T_j=\mathcal D(T_{j+1})\quad(0\leq j\leq p-2),
\]
where $\mathcal D(T)$ is the poset of downsets of $T$.
Put
\[
 M=|T_1|=t_{p-1}(r).
\]

\begin{proposition}\label{prop:appendix-sharper-ksh-bound}
For every prime $p$ and every $r\in\mathbb N$,
\begin{equation}
 C(p,r)\leq (p-1)t_{p-1}(r)-1.
 \label{eq:appendix-sharper-ksh-bound}
\end{equation}
In particular,
\begin{equation}
 C(3,r)\leq 2^{r+1}-1.
 \label{eq:appendix-sharper-ksh-triangle}
\end{equation}
\end{proposition}

\begin{proof}
Suppose that $c$ is a proper $r$-coloring of $\KSh_p(n,k)$.
The recursive construction in the proof of \cref{thm:kneser-shift} assigns to every
$\mathcal A\in\P_p(n,k)$ a vector
\[
 D(\mathcal A)=(D_1(\mathcal A),\ldots,D_p(\mathcal A))
 \in\mathcal D(T_1)^p
\]
with the following three properties:
\begin{enumerate}[label=\textup{(\roman*)}]
\item $D(\mathcal A)$ is nonconstant;
\item the assignment is equivariant under cyclic coordinate shift;
\item if $\mathcal A\leq\mathcal B$, then
$D_i(\mathcal A)\subseteq D_i(\mathcal B)$ for every $i\in[p]$.
\end{enumerate}
We construct a more economical signed labeling of the possible history vectors.

Fix an arbitrary ordering
\[
 T_1=\{x_1,\ldots,x_M\}.
\]
For a nonconstant vector
$E=(E_1,\ldots,E_p)\in\mathcal D(T_1)^p$, define
\[
 I_j(E)=\{i\in\mathbb Z_p:x_j\in E_i\}.
\]
Because $E$ is nonconstant, there is an index $j$ for which
$I_j(E)$ is a nonempty proper subset of $\mathbb Z_p$.
Let $j(E)$ be the least such index and put $I(E)=I_{j(E)}(E)$.

The cyclic group $\mathbb Z_p$ acts freely by translation on its nonempty proper subsets.
Indeed, since $p$ is prime, a subset fixed by a nonzero translation is either empty or all of $\mathbb Z_p$.
Consequently, we may choose a map
\[
 \eta:\{I:\varnothing\neq I\subsetneq\mathbb Z_p\}
 \longrightarrow\mathbb Z_p
\]
such that
\begin{equation}
 \eta(I+a)=\eta(I)+a
 \qquad(a\in\mathbb Z_p).
 \label{eq:appendix-phase-map}
\end{equation}
For example, choose one representative from each translation orbit, give it value $0$, and extend by \eqref{eq:appendix-phase-map}.

Define
\begin{align*}
 \nu_1(E)&=\eta(I(E)),\\
 \nu_2(E)&=(p-1)(j(E)-1)+|I(E)|.
\end{align*}
Then
\[
 \nu=(\nu_1,\nu_2):
 \bigl(\mathcal D(T_1)^p\bigr)^*
 \longrightarrow
 \mathbb Z_p\times[(p-1)M]
\]
is equivariant, where the star means that the constant vectors are omitted.
Indeed, a cyclic shift translates $I(E)$, while leaving both $j(E)$ and $|I(E)|$ unchanged.

We verify the tie condition.
Suppose that $E_i\subseteq F_i$ for all $i$ and that
$\nu_2(E)=\nu_2(F)$.
The definition of $\nu_2$ uniquely determines the pair
$(j(E),|I(E)|)$, so
\[
 j(E)=j(F)
 \qquad\text{and}\qquad
 |I(E)|=|I(F)|.
\]
At their common selected element $x_{j(E)}$, coordinatewise containment gives
\[
 I(E)\subseteq I(F).
\]
The two sets have equal cardinality, hence $I(E)=I(F)$ and therefore
$\nu_1(E)=\nu_1(F)$.

The composite labeling
\[
 \mu=\nu\circ D:
 \P_p(n,k)\longrightarrow
 \mathbb Z_p\times[(p-1)M]
\]
is thus equivariant and satisfies \eqref{eq:appendix-ksh-tie}.
Applying the Kneser--Tucker lemma with $m=(p-1)M$ gives
\[
 n-pk<(p-1)M-1.
\]
It follows from the definition of the optimal threshold that
\[
 C(p,r)\leq(p-1)M-1=(p-1)t_{p-1}(r)-1.
\]
For $p=3$, one has $T_1=2^{[r]}$ and hence $M=2^r$, giving
\eqref{eq:appendix-sharper-ksh-triangle}.
\end{proof}

For comparison, the rank-sum labeling used in the main proof gives
$C(p,r)\leq p\,t_{p-1}(r)-2$.
The improvement comes from recording only the first element of $T_1$ whose
membership pattern across the $p$ history downsets is nonconstant, together with
the size of that pattern.
For $p=2$, \cref{prop:appendix-sharper-ksh-bound} recovers the exact bound
$C(2,r)=r-1$.

There is also a small improvement to the lower bound for $p=3$.
As in the main text, write
\[
 B(s)=\binom{s}{\lfloor s/2\rfloor},
\]
and set $B(0)=1$.

\begin{proposition}\label{prop:appendix-sharper-ksh-lower-bound}
For every $r\geq1$,
\begin{equation}
 C(3,r)\geq 2B(r-1)-2.
 \label{eq:appendix-sharper-ksh-lower-bound}
\end{equation}
Consequently,
\begin{equation}
 2\binom{r-1}{\lfloor(r-1)/2\rfloor}-2
 \leq C(3,r)\leq2^{r+1}-1.
 \label{eq:appendix-combined-ksh-bounds}
\end{equation}
The new lower bound agrees with $B(r)-2$ when $r$ is even and is strictly larger when $r\geq3$ is odd.
\end{proposition}

\begin{proof}
For $r\leq2$ the asserted lower bound is zero and follows from
$C(3,r)\geq0$, so assume that $r\geq3$.
Put $q=r-1$ and $t=B(q)$.
Choose distinct sets
\[
 F_1,\ldots,F_t\in\binom{[q]}{\lfloor q/2\rfloor}.
\]
We construct a proper $r$-coloring of $\KSh_3(2t,1)$.
Identify the ground set with $[t]\times\{0,1\}$.
For two distinct points $x=(i,a)$ and $y=(j,b)$, color the ordered pair $(x,y)$ as follows:
\[
 c(x,y)=
\begin{cases}
 r,&i=j,\\
 \min(F_j\setminus F_i),&i\neq j.
\end{cases}
\]
When $i\neq j$, the difference $F_j\setminus F_i$ is nonempty because the two sets are distinct and have equal size.

Consider three distinct points
\[
 x=(i,a),\qquad y=(j,b),\qquad z=(\ell,d).
\]
We show that $c(x,y)\neq c(y,z)$.
If $i,j,\ell$ are pairwise distinct, then
\[
 c(x,y)\in F_j,
 \qquad
 c(y,z)\notin F_j,
\]
so the colors differ.
If $i=j$, then $c(x,y)=r$.
The equality $j=\ell$ is impossible as well, since a two-element fiber cannot contain the three distinct points $x,y,z$; hence $c(y,z)\in[q]$.
The case $j=\ell$ is symmetric.
Finally, if $i=\ell\neq j$, then
\[
 c(x,y)\in F_j,
 \qquad
 c(y,z)\in F_i\setminus F_j,
\]
and again the colors differ.
These cases exhaust all possibilities, so the coloring is proper.

It exists at
\[
 n=2B(r-1),\qquad k=1,
\]
whose excess is $n-3=2B(r-1)-3$.
Therefore the least forcing excess satisfies
\[
 C(3,r)\geq2B(r-1)-2.
\]
For even $r=2m$,
$2B(r-1)=2\binom{2m-1}{m-1}=\binom{2m}{m}=B(r)$.
For odd $r=2m+1$, one has
\[
 2B(r-1)=2\binom{2m}{m}
 >\binom{2m+1}{m}=B(r),
\]
which proves the final assertion.
\end{proof}

\section{Palette-free canonical consequences}
\label{sec:appendix-canonical}

\subsection{A palette-free form of the rotation lemma}
\label{subsec:canonical-rotation}

The dense rotation theorem has a useful canonical consequence which does not
require a bound on the original palette.
We state it in the language of equivalence relations.

\begin{lemma}[Dense kernel rotation]
\label{lem:dense-kernel-rotation}
Let $q$ be prime, let $X$ be a finite alphabet, and fix $u\in X$.
For every $m\in\mathbb N$ and $\eta>0$ there are $N,k\in\mathbb N$ with the
following property.
For every equivalence relation $\sim$ on $X^N$, there are pairwise disjoint
blocks
\[
 B_{b,1},\ldots,B_{b,q}\qquad (b\in[m]),
\]
all of size $k$, covering at least $(1-\eta)N$ coordinates, such that the
following holds.
Fix $b$, two lists
\[
 (x_1,\ldots,x_{q-1}),\ (y_1,\ldots,y_{q-1})\in X^{q-1},
\]
and arbitrary (possibly different) constant assignments on all the other
selected blocks.
Put $u$ on every unselected coordinate.
Then the two resulting words with patterns
\[
 (x_1,\ldots,x_{q-1},u),\qquad
 (y_1,\ldots,y_{q-1},u)
\]
on $(B_{b,1},\ldots,B_{b,q})$ are equivalent if and only if the two words
obtained by replacing these patterns by
\[
 (u,x_1,\ldots,x_{q-1}),\qquad
 (u,y_1,\ldots,y_{q-1})
\]
are equivalent.
\end{lemma}

\begin{proof}
Define a two-coloring of $(X\times X)^N$ by
\[
 h\bigl((a_1,b_1),\ldots,(a_N,b_N)\bigr)=
 \begin{cases}
  1,&(a_1,\ldots,a_N)\sim(b_1,\ldots,b_N),\\
  0,&\text{otherwise}.
 \end{cases}
\]
Apply the dense rotation theorem to $h$, with alphabet $X\times X$ and
baseline $(u,u)$.
An arbitrary constant letter on another selected block is now an arbitrary
ordered pair of letters of $X$, so the contextual conclusion of that theorem
is exactly the asserted equivalence of kernels.
\end{proof}

For a cyclic alphabet this already replaces the first two steps in Shaw's
invariantization induction by one dense step.

\begin{corollary}[Dense one-edge interchangeability]
\label{cor:dense-one-edge-interchangeability}
Let $p$ be prime and identify $C_p$ with $\mathbb Z_p$.
Fix $a\in\mathbb Z_p$.
For every $m\in\mathbb N$ and $\eta>0$ there are $N,K\in\mathbb N$ such that
every equivalence relation on $C_p^N$ has a $C_p$-block copy of $C_p^m$,
of active proportion at least $1-\eta$ and block size $K$, on which the
induced relation is $\{a,a+1\}$-interchangeable.
Here interchangeability means that changing a fixed occurrence of $a$ to
$a+1$, or conversely, does not change any labelled pullback kernel.
The block copy may be chosen label-balanced.
\end{corollary}

\begin{proof}
Apply \cref{lem:dense-kernel-rotation} with $q=p$ and baseline $u=a$.
For each macrocoordinate $b$, encode $x\in\mathbb Z_p$ on its $p$ selected
blocks by
\[
 \sigma(x)=(x,x+1,\ldots,x+p-1).
\]
The word $\sigma(a+1)$ ends in $a$, and moving this last entry to the front
turns it into $\sigma(a)$.
The contextual assertion of \cref{lem:dense-kernel-rotation}, applied to an
arbitrary pair of substitutions into the remaining wildcards, therefore
says precisely that the two corresponding pullback kernels are equal.
Each translation label occurs on exactly one of the $p$ equal constituent
blocks, so the resulting block map is label-balanced.
\end{proof}

Taking $a=1$, \cref{cor:dense-one-edge-interchangeability} produces
$\{1,2\}$-interchangeability directly from an arbitrary equivalence
relation.
Thus the palette-independent difficulty in a dense version of Shaw's
argument can be moved past the initial, vacuous $[1]$-interchangeability
stage.
What remains is a dense mechanism producing $[2]$-swappability (and then its
higher analogues) while preserving the interchangeability already obtained.
The lemma above does not provide this coordinate swappability.

\subsection{The three-face obstruction}
\label{subsec:three-face-obstruction}

We next record why a black-box dense replacement for Shaw's original
swappability lemma would solve an additive Kneser--Ramsey problem which is
open even for two colors.

Let $X$ be an alphabet containing distinct letters $1,2,3$, with baseline
$1$, and let $\sim$ be an equivalence relation on $X^N$.
For disjoint $k$-sets $A,B\subset[N]$, let
\[
 \kappa(A,B)
\]
denote the labelled pullback of $\sim$ to the two-dimensional face which is
constant on $A$ and on $B$ and equal to $1$ elsewhere.

\begin{proposition}[Exact Kneser obstruction]
\label{prop:exact-kneser-obstruction}
Suppose that for some integer $D\geq0$ the following statement holds for every
$k$ and every equivalence relation on $X^{3k+D}$: there are pairwise
disjoint $k$-sets $A,B,C$ such that
\begin{equation}
 \label{eq:appendix-three-face-kernels}
 \kappa(B,C)=\kappa(A,C)=\kappa(A,B).
\end{equation}
Then
\[
 KG(3k+D,k)\longrightarrow_2 K_3
\]
for every $k$.
The implication remains valid if the hypothesis is restricted to
$[1]$-interchangeable equivalence relations.
\end{proposition}

\begin{proof}
Start with an arbitrary red--blue coloring of the edges of
$KG(3k+D,k)$.
For every red edge $\{A,B\}$, make
\[
 2^A3^B1^{[N]\setminus(A\cup B)}
 \quad\text{and}\quad
 2^B3^A1^{[N]\setminus(A\cup B)}
\]
a two-element equivalence class, and leave all other words singleton.
These pairs are disjoint: a word occurring in one of them uniquely recovers
its set of $2$-coordinates and its set of $3$-coordinates.
Consequently, $\kappa(A,B)$ is discrete on a blue edge, whereas on a red
edge its only nonsingleton class is
\[
 (2,3)\sim(3,2).
\]
Thus \eqref{eq:appendix-three-face-kernels} gives a monochromatic triangle
in the original edge coloring.
The last assertion is automatic: with the set of allowed fixed letters equal
to $\{1\}$, there is no nontrivial letter substitution to make.
\end{proof}

In particular, the additional interchangeability present in the later
stages of Shaw's induction does not simplify its original first stage: at
that point the assumption is vacuous.
For $s$ edge colors, take a baseline letter $1$ and distinct pairs
$(2_i,3_i)$, one for each color, and on an edge of color $i$ identify only
the two words obtained by placing $(2_i,3_i)$ and $(3_i,2_i)$ on its two
blocks.
The resulting labelled face kernels are equal exactly on equally colored
edges.
Thus, applied with this larger alphabet, the usual dense three-face lemma
would imply the constant-excess assertion
\[
 \forall s\ \exists D_s\ \forall k:\qquad
 KG(3k+D_s,k)\longrightarrow_s K_3.
\]
The preceding palette-free rotation lemma gives a possible way around this
first bottleneck, but it leaves the structured $[2]$-swappability problem
open.

\subsection{A two-color lower bound for Kneser triangles}
\label{subsec:kneser-triangle-lower}

\begin{theorem}[Five-point profile coloring]
\label{thm:five-point-profile}
For every $k\geq1$, the edges of $KG(3k+2,k)$ admit a two-coloring with no
monochromatic triangle.
Consequently, in the notation where $R_k^{KG}(3,3)$ is the least $n$ such
that every two-coloring of $KG(n,k)$ contains a monochromatic triangle,
\[
 R_k^{KG}(3,3)\geq3k+3.
\]
For $k\geq3$, this improves by one the general lower bound of Heath, McCourt,
Parker, Schwieder, and Zerbib~\cite{HeathEtAl2025}; for $k=2$ it recovers
their exact value $R_2^{KG}(3,3)=9$, and for $k=1$ it is the classical
identity $R(3,3)=6$.
\end{theorem}

\begin{proof}
Fix a distinguished set $T=\mathbb Z_5$ and give a $k$-set $X$ the profile
$P_X=X\cap T$.
For disjoint profiles $A,B\subseteq T$, color the corresponding edge red by
the following rule.
\begin{enumerate}
\item If one profile is empty, the edge is red exactly when the other has
size at least two.
\item If both profiles have size at least two, the edge is blue.
\item Otherwise at least one profile is a singleton.
The edge is red exactly when, for every singleton $\{x\}$ among $A,B$, the
other profile avoids $x+1$ modulo five.
\end{enumerate}

For three pairwise disjoint $k$-sets, the union of their profiles has size at
least three, since the complement of $T$ has only $3k-3$ points.
Order the three profile sizes increasingly.
If the smallest is zero, the possibilities are
\[
 (0,0,\geq3),\qquad (0,1,\geq2),\qquad
 (0,\geq2,\geq2).
\]
The empty--empty and empty--singleton edges are blue, the empty--large
edges are red, and the large--large edge in the last case is blue.

Otherwise the smallest size is one, since three disjoint profiles of size
at least two do not fit in $T$.
For type $(1,2,2)$, the two $2$-sets partition the four points outside the
singleton; exactly one contains its cyclic successor, so the two incidences
with the singleton have different colors.
For type $(1,1,c)$ with $c\geq2$, an all-red triangle would require the two
singleton points to be nonadjacent and the large profile to avoid both of
their distinct successors, leaving room for at most one point.
An all-blue triangle would require the two singleton points to be adjacent,
but then the edge from the predecessor singleton to the disjoint large
profile is red.
Finally, for type $(1,1,1)$, red means nonadjacency and blue means adjacency
in the $5$-cycle; both the $5$-cycle and its complement are triangle-free.
Thus every triangle uses both colors.
\end{proof}

\begin{corollary}[The obstruction occurs for genuine face kernels]
\label{cor:genuine-kernel-obstruction}
For alphabet $\{1,2,3\}$, the strong dense three-face statement fails at
excess two even when its colors are genuine labelled pullback kernels of a
single equivalence relation.
\end{corollary}

\begin{proof}
Apply the equivalence-relation construction in the proof of
\cref{prop:exact-kneser-obstruction} to the red edges of the coloring in
\cref{thm:five-point-profile}.
The two possible kernels are the discrete relation and the relation whose
only nonsingleton class is $(2,3)\sim(3,2)$.
A triple with all three deletion kernels equal would therefore be a
monochromatic triangle, which does not exist.
\end{proof}

\subsection{Prime-degree transitive configurations}
\label{subsec:prime-degree-transitive}

Shaw~\cite{ShawPrime2026} states his theorem for regular polygons, but its proof gives the
following slightly more general conclusion.

\begin{proposition}[Prime-cardinality transitive configurations]
\label{prop:prime-cardinality-canonical}
Let $P$ be a finite Euclidean configuration with $|P|=p$ prime.
If the Euclidean symmetry group of $P$ acts transitively on $P$, then every
power $P^k$ is canonically Ramsey.
More precisely, Shaw's construction gives, for every $k$, an $n$ such that
\[
 p^{-p/2}P^n\xrightarrow[\mathrm{MR}]{}P^k.
\]
\end{proposition}

\begin{proof}
Let $G$ be a transitive permutation group on $P$ induced by Euclidean
symmetries.
Orbit--stabilizer gives $p\mid |G|$, so Cauchy's theorem supplies an element
$\tau\in G$ of order $p$.
Since $P$ has $p$ points, $\tau$ acts as a single $p$-cycle.

The combinatorial part of Shaw's proof uses only an alphabet of size $p$
together with such a cyclic permutation.
Its standard emulator
\[
 x\longmapsto(x,\tau x,\ldots,\tau^{p-1}x)
\]
is a scaled isometric embedding: for all $x,y\in P$,
\[
 \sum_{i=0}^{p-1}\|\tau^i x-\tau^i y\|^2
 =p\|x-y\|^2.
\]
Every other embedding used in the invariantization is a product of copies
of this emulator and fixed coordinates.
In Shaw's normalization, the invariantization contributes the scale factor
$p^{(p-1)/2}$ and the final displayed emulator contributes another
$\sqrt p$.
The total factor is therefore $p^{p/2}$, which accounts for the
$p^{-p/2}$ in the statement.
Thus the same proof applies with $P$ in place of the regular polygon,
and likewise coordinatewise to $P^k$.
\end{proof}

The proposition is a canonical statement, not a near-circumradius one.
The ordinary Ramsey extraction in Shaw's invariantization may leave an
arbitrarily large inactive part, so the radius of the product witness is not
controlled near the circumradius of $P$.

\section{Further obstructions for the Leader--Russell--Walters kite}
\label{sec:lrw-further}

Fix a transcendental number $a\in(-1,1)$ and put $b=\sqrt{1-a^2}$.
We use the notation
\[
 L=(-1,0),\qquad R=(1,0),\qquad
 U=(a,b),\qquad D=(a,-b),
\]
and write $P_a=\{L,R,U,D\}$.
This is one of the cyclic quadrilaterals of Leader, Russell, and
Walters~\cite{LeaderRussellWalters2011} which is not contained in any finite
transitive set.
Its circumradius is one, and its unique affine dependence, up to a scalar
multiple, is
\begin{equation}
 (a-1)L-(a+1)R+U+D=0.
 \label{eq:lrw-affine}
\end{equation}

The results below do not decide whether $P_a$ is ncs-Ramsey.
They show that three natural ways of approaching the problem cannot work:
powers of one fixed alphabet can only approach the optimal radius when they
already attain it, the natural six-letter grids admit explicit kite-free
colorings at every bounded radius, and the radius divergence in the
generalized-prism construction persists for every choice of interpolation
path.

\subsection{The exact fixed-alphabet product optimum}

Let $X$ be a finite spherical set whose circumcenter is the origin and whose
circumradius is $R_X$.
Call a map $f:P_a\to X$ an \emph{atom} if it preserves
\eqref{eq:lrw-affine}, and put
\[
 s(f)=\frac14\|f(L)-f(R)\|^2,
 \qquad
 t(f)=\left\langle
       \frac{f(R)-f(L)}2,\frac{f(U)-f(D)}2
      \right\rangle .
\]
Let
\[
 s_X=\max\{s(f):f\colon P_a\to X\text{ is an atom}\}.
\]

\begin{proposition}\label{prop:lrw-fixed-alphabet}
The least squared circumradius of a weighted orthogonal Cartesian power of $X$
which contains a congruent copy of $P_a$ is
\[
 \frac{R_X^2}{s_X},
\]
with the usual value $+\infty$ when $s_X=0$.
Moreover, this value is one if and only if $X$ contains an atom which is a
copy of $R_XP_a$ with circumcenter at the origin.
Consequently, if weighted powers of one fixed centered alphabet contain
copies of $P_a$ and have circumradii tending to one, then one weighted power
already has circumradius one and contains $P_a$; otherwise all of its
weighted powers have a uniform positive circumradius gap.
\end{proposition}

\begin{proof}
Suppose first that $x,y,u,d$ lie on the sphere containing $X$ and satisfy
the affine relation in \eqref{eq:lrw-affine}.
Put
\[
 h=\frac{x+y}{2},\qquad
 \ell=\frac{y-x}{2},\qquad
 q=\frac{u-d}{2}.
\]
The affine relation gives
\[
 x=h-\ell,\qquad y=h+\ell,\qquad
 u=h+a\ell+q,\qquad d=h+a\ell-q.
\]
Equality of the four norms gives
\[
 h\perp\ell,\qquad q\perp(h+a\ell),
 \qquad \|q\|^2=(1-a^2)\|\ell\|^2.
\]
Writing $s=\|\ell\|^2$ and $t=\langle\ell,q\rangle$, a direct expansion
shows that weighted atoms $f$ with weights $\lambda_f\geq0$ produce a copy
of $P_a$ exactly when
\begin{equation}
 \sum_f\lambda_fs(f)=1,
 \qquad
 \sum_f\lambda_ft(f)=0.
 \label{eq:lrw-atom-conditions}
\end{equation}

After discarding factors of weight zero, every coordinate map occurring in a
product copy is an atom.
Indeed, for the coefficient vector $z=(a-1,-a-1,1,1)$ and every
squared-distance matrix $E$ one has
\[
 z^{\mathsf T}Ez=-2\left\|\sum_i z_iq_i\right\|^2.
\]
The left-hand side vanishes for $P_a$.
In an orthogonal product it is a sum of nonpositive coordinate
contributions, and hence every coordinate contribution vanishes.

The first equality in \eqref{eq:lrw-atom-conditions} now gives
\[
 1\leq s_X\sum_f\lambda_f.
\]
The squared radius of the weighted power is
$R_X^2\sum_f\lambda_f$, proving the lower bound.
Choose an atom $f$ attaining $s_X$.
Precomposing it with the reflection which interchanges $U$ and $D$ preserves
$s(f)$ and changes the sign of $t(f)$.
Giving these two atoms weight $1/(2s_X)$ each satisfies
\eqref{eq:lrw-atom-conditions} and attains the lower bound.

Finally, $s_X\leq R_X^2$ because a chord of the radius-$R_X$ sphere has
length at most $2R_X$.
If equality holds, a maximizing atom has $h=0$ in the notation above.
Since $a\ne0$, the relation $q\perp(h+a\ell)$ then gives
$q\perp\ell$; together with
$\|q\|^2=(1-a^2)\|\ell\|^2$, this says that the four images of the atom
form a copy of $R_XP_a$ centered at the origin.
The converse is immediate, and the assertion about asymptotic powers follows
because the displayed minimum is attained.
\end{proof}

\subsection{A kite-free coloring of every natural product grid}

Consider the six-point set
\[
 X_a=\{(\pm1,0)\}\cup\{(\pm a,\pm b)\}
      =O_1\cup O_2,
\]
where $O_1=\{(\pm1,0)\}$ and
$O_2=\{(\pm a,\pm b)\}$.
The Klein four-group generated by the two coordinate sign changes acts on
$X_a$ with orbits $O_1$ and $O_2$.

\begin{lemma}\label{lem:lrw-natural-atoms}
Every atom $f:P_a\to X_a$ is either constant or is obtained from the
inclusion $P_a\subset X_a$ by a coordinate sign change.
\end{lemma}

\begin{proof}
For $V\in X_a$, write $\sigma(V)=V_2/b\in\{0,1,-1\}$.
The second coordinate of the atom relation, divided by $b$, is
\[
 a\bigl(\sigma(f(L))-\sigma(f(R))\bigr)
 -\sigma(f(L))-\sigma(f(R))+\sigma(f(U))+\sigma(f(D))=0.
\]
Transcendence of $a$ implies
\[
 \sigma(f(L))=\sigma(f(R)),\qquad
 \sigma(f(U))+\sigma(f(D))=2\sigma(f(L)).
\]
If the common value is $1$ or $-1$, comparison of the coefficients of
$a^2$ and $a$ in the first coordinate forces all four images to agree.
If the common value is zero and $f(U),f(D)\in O_1$, comparison of the
coefficients of $a$ and $1$ again gives a constant map.
In the remaining case $f(L),f(R)$ are opposite points of $O_1$ and
$f(U),f(D)$ are vertically opposite points of $O_2$ whose common horizontal
sign is forced by the first two images.
These are precisely the four coordinate sign changes of the inclusion.
\end{proof}

\begin{proposition}\label{prop:lrw-grid-coloring}
Let $\omega_1,\ldots,\omega_N>0$, let $s>0$, and suppose that the scaled
weighted grid
\[
 s\bigl(\sqrt{\omega_1}X_a\times\cdots\times
          \sqrt{\omega_N}X_a\bigr),
\]
possibly after appending a coordinate common to every grid point, lies on a
sphere of radius $\mathcal R$ centered at the origin.
Then the grid has a coloring with
\[
 q=\lfloor\mathcal R^2\rfloor+1
\]
colors and no monochromatic copy of $P_a$.
In particular, if $1\leq\mathcal R<\sqrt2$, two colors suffice.
\end{proposition}

\begin{proof}
Put $W=\sum_j\omega_j$.
For a grid word $v=(v_1,\ldots,v_N)$, define
\[
 p(v)=\frac1W\sum_{j:v_j\in O_1}\omega_j\in[0,1].
\]
Partition $[0,1]$ into $q$ intervals of length $1/q$, with the final
endpoint assigned to the last interval, and color $v$ by the interval
containing $p(v)$.

Suppose four grid words form a copy of $P_a$.
By \cref{lem:lrw-natural-atoms}, every nonconstant coordinate map is a
coordinate sign change of the inclusion.
Let $A$ be the total weight of the nonconstant coordinates.
The squared distance between the words playing $L$ and $R$ is $4s^2A$, so
congruence with $P_a$ gives
\[
 s^2A=1.
\]
Every active coordinate is in $O_1$ for the words playing $L,R$ and in
$O_2$ for the words playing $U,D$.
Constant coordinates contribute equally, and therefore
\[
 |p(L)-p(U)|=\frac AW=\frac1{s^2W}.
\]
The common padding coordinate, if present, contributes only to the ambient
radius, so $\mathcal R^2\geq s^2W$.
Consequently
\[
 |p(L)-p(U)|\geq\frac1{\mathcal R^2}>\frac1q,
\]
and the two words receive different colors.
\end{proof}

\subsection{The cyclic-path generalized-prism method has an unavoidable
radius gap}

For $\lambda\in\mathbb R$, put
\[
 Z_\lambda=(O_1\times\{0\})\cup(O_2\times\{\lambda\}).
\]
For $\lambda\ne0$, the generalized-prism theorem shows that $Z_\lambda$
is solvable subtransitive.
Its circumcenter and circumradius are
\[
 (0,0,\lambda/2),\qquad
 \rho_\lambda=\sqrt{1+\lambda^2/4}.
\]
The known construction chooses an interpolation path between one point of
$O_1$ and one point of $O_2$, places the path in an orthogonal product, and
uses a cyclic coordinate shift.
We next show that no choice of path can make this construction
near-circumradius.

\begin{lemma}\label{lem:lrw-path-energy}
Let $x,y$ be distinct unit vectors, put $d=\|x-y\|$, and let
\[
 z_0=x,z_1,\ldots,z_n=y
\]
be arbitrary vectors in a real Hilbert space.
Set
\[
 E=\sum_{i=1}^n\|z_i-z_{i-1}\|^2,
 \qquad
 S=\sum_{i=0}^n\|z_i\|^2.
\]
If $0<E<d^2/4$, then
\begin{equation}
 S\geq \frac{(d^2/4-E)^2}{4E}.
 \label{eq:lrw-path-bound}
\end{equation}
\end{lemma}

\begin{proof}
Put $w=(x-y)/d$ and $u_i=\langle z_i,w\rangle$.
The equality $\|x\|=\|y\|$ gives
\[
 u_0=d/2,\qquad u_n=-d/2.
\]
At a sign change of the sequence choose, from the two consecutive terms,
one of smaller absolute value, say $u_k$.
Then $u_k^2\leq E$.
Join $u_k$ to the endpoint on the corresponding side of this sign change.
Telescoping the squares along that subpath and applying Cauchy--Schwarz gives
\[
\begin{aligned}
 d^2/4-E
 &\leq |u_{\mathrm{end}}^2-u_k^2|\\
 &\leq \sum_i|u_i-u_{i-1}|\,|u_i+u_{i-1}|\\
 &\leq \sqrt E
   \left(\sum_i(|u_i|+|u_{i-1}|)^2\right)^{1/2}\\
 &\leq 2\sqrt{ES}.
\end{aligned}
\]
Squaring proves \eqref{eq:lrw-path-bound}.
\end{proof}

\begin{proposition}\label{prop:lrw-cyclic-path}
Choose $x\in O_1$ and $y\in O_2$, and put $d=\|x-y\|$.
Consider any transitive host obtained as follows.
Choose an arbitrary path $z_0=x,\ldots,z_n=y$ in the sign-change plane,
take the orbit of $(z_0,\ldots,z_n)$ under
\[
 V_4^{\,n+1}\rtimes C_{n+1},
\]
where $C_{n+1}$ cyclically permutes the coordinates, and require this orbit
to contain $Z_\lambda$ through the usual two consecutive cyclic shifts.
If $0<|\lambda|<d/2$, then the radius $R$ of the host satisfies
\begin{equation}
 R^2\geq
 \frac{(d^2/4-\lambda^2)^2}{4\lambda^2}.
 \label{eq:lrw-cyclic-radius}
\end{equation}
The same conclusion holds if the construction is first performed at another
common scale and then scaled to contain $Z_\lambda$.
\end{proposition}

\begin{proof}
The cyclic orbit contains the two families
\[
 A_g=(gx,z_1,\ldots,z_n),\qquad
 B_g=(gy,z_0,\ldots,z_{n-1})qquad(g\in V_4).
\]
For every $g,h\in V_4$,
\[
 \|A_g-B_h\|^2
 =\|gx-hy\|^2+
   \sum_{i=1}^n\|z_i-z_{i-1}\|^2.
\]
Thus the additional squared distance between the two layers is
\[
 E=\sum_{i=1}^n\|z_i-z_{i-1}\|^2=\lambda^2.
\]
The independent sign changes make the orbit barycenter zero, so its squared
radius is
\[
 R^2=\sum_{i=0}^n\|z_i\|^2=S.
\]
Now \cref{lem:lrw-path-energy} gives
\eqref{eq:lrw-cyclic-radius}.

For the scaled version, perform the final common scaling first and normalize
the path so that its endpoints are the prescribed unit vectors $x,y$.
Isometric containment of $Z_\lambda$ then forces the normalized energy to be
$\lambda^2$, while the normalized sum of squared path norms is the squared
radius of the scaled host.  The same application of
\cref{lem:lrw-path-energy} gives the result.
\end{proof}

\begin{corollary}\label{cor:lrw-cyclic-divergence}
For fixed $a$ and fixed choices of $x\in O_1$, $y\in O_2$, every
cyclic-path generalized-prism host has
\[
 R=\Omega_a(1/|\lambda|)
 \qquad\text{as }\lambda\longrightarrow0,
\]
while $\rho_\lambda\to1$.
The straight interpolation used in the generalized-prism proof has the
matching order $R=O_a(1/|\lambda|)$ along
$\lambda=d/\sqrt n$.
Thus its order of divergence is optimal within the entire cyclic-path
scheme.
\end{corollary}

\subsection{A projective sufficient condition and its limitations}

Put $\alpha=|a|$.
The kites $P_a$ and $P_\alpha$ are congruent, so the sign of $a$ is
irrelevant in this subsection.
Fix a separable infinite-dimensional real Hilbert space.
Let $\mathcal H_\alpha$ be the $3$-uniform hypergraph whose vertices are its
unoriented one-dimensional subspaces, or \emph{axes}.
Three distinct axes form an edge if they admit oriented unit representatives
$x,y,z$ satisfying
\begin{equation}
 y+z=2\alpha x.
 \label{eq:lrw-projective-edge}
\end{equation}
Equivalently, for some unit vector $e\perp x$,
\[
 y=\alpha x+\sqrt{1-\alpha^2}\,e,
 \qquad
 z=\alpha x-\sqrt{1-\alpha^2}\,e.
\]
By the compactness theorem for hypergraph colorings,
$\chi(\mathcal H_\alpha)=\infty$ is equivalent to the occurrence of finite
subhypergraphs of arbitrarily large chromatic number.

We first record an orbitwise consequence of the dense block theorem which
does not require a transitive action on the alphabet.

\begin{lemma}[Orbitwise geometric reduction]
\label{lem:lrw-orbitwise}
Let a finite solvable group $G$ act by isometries on a finite set $X$ lying
on a sphere of radius $\rho$ centered at a point fixed by $G$.
For every $r\in\mathbb N$ and every $\varepsilon>0$, every $r$-coloring of a
sufficiently high-dimensional sphere of radius $\rho+\varepsilon$ contains an
isometric copy $\phi(X)$ such that
\[
 Gx=Gy\quad\Longrightarrow\quad
 c(\phi(x))=c(\phi(y)).
\]
\end{lemma}

\begin{proof}
Translate the fixed center to the origin.
Apply the dense block theorem to $G\curvearrowright X$ with one variable
and with active block size $k$ and length $n$ chosen so densely that
\[
 \rho\sqrt{n/k}\leq\rho+\varepsilon.
\]
Scale the resulting $X^n$ grid by $k^{-1/2}$ and append one coordinate,
common to all words, to place the grid on the prescribed sphere.
On the active block every coordinate has the form $gx$ for an isometry
$g\in G$, while the other coordinates are fixed.
The restriction of the block map is therefore an isometry on $X$.
Orbit-insensitivity gives the asserted equality of colors.
\end{proof}

\begin{proposition}[Projective fusion criterion]
\label{prop:lrw-projective-reduction}
If $\chi(\mathcal H_\alpha)=\infty$, then the LRW kite $P_a$ is
ncs-Ramsey.
\end{proposition}

\begin{proof}
Fix $r$ and $\varepsilon>0$, and choose a finite subhypergraph
$F\subset\mathcal H_\alpha$ with $\chi(F)>r$.
Choose one unit representative of every axis of $F$, and let $X$ consist of
these vectors and their negatives.
Apply \cref{lem:lrw-orbitwise} to the antipodal action of $C_2$ on $X$.
In the resulting copy, the two points belonging to each axis have one common
color.

An isometry of $X$ extends to an affine isometry of its affine hull.
Since $X=-X$, it consequently has the form
\[
 \phi(v)=h+Tv,
\]
where $T$ is a linear isometry on $\operatorname{span}X$.
The two points $h+Tv,h-Tv$ both lie on the sphere of radius $1+\varepsilon$.
Subtracting their squared-norm identities gives
$h\perp T\operatorname{span}X$.

The common colors of the antipodal pairs define an $r$-coloring of the
vertices of $F$, and hence $F$ has a monochromatic edge.
Orient this edge as in \eqref{eq:lrw-projective-edge}.
Then
\[
 h-Tx,\qquad h+Tx,\qquad h+Ty,\qquad h+Tz
\]
are monochromatic and form a congruent copy of $P_\alpha$, and hence of
$P_a$.
\end{proof}

This reduces the LRW problem to a precise projective coloring question, but
the chromatic number of $\mathcal H_\alpha$ remains open.
The next statements explain why two tempting simplifications do not settle
it.

Let $\widetilde{\mathcal H}_\alpha$ be the $3$-uniform hypergraph on
oriented unit vectors in the same Hilbert space, whose
edges are the triples satisfying \eqref{eq:lrw-projective-edge}.

\begin{proposition}\label{prop:lrw-oriented-coloring}
The chromatic numbers of the restrictions of
$\widetilde{\mathcal H}_\alpha$ to finite-dimensional subspaces are bounded
uniformly in the ambient dimension.
\end{proposition}

\begin{proof}
The three vectors form a fixed triangle on the unit sphere, and their unique
linear dependence has coefficient vector
\[
 (1,1,-2|a|).
\]
No nonempty subsum of these coefficients vanishes, because $a$ is
transcendental.
Graham's Rado-theoretic necessary condition for a set to be
circumsphere-Ramsey~\cite[Theorem~1]{Graham1983} therefore gives a fixed finite
coloring, in every dimension, with no monochromatic orthogonal image of
this triangle.
This is precisely a proper coloring of $\widetilde{\mathcal H}_\alpha$.
\end{proof}

The projective hypergraph used in the possible positive route is different:
its vertices are axes, and an edge only has to admit some choice of signs for
which the displayed relation holds.
A coloring of axes must be invariant under $v\mapsto-v$, whereas the
coloring supplied by the preceding Rado argument is not required to be
antipodally invariant.
Thus a proof that the projective hypergraph has infinite chromatic number,
if one exists, must genuinely exploit the unoriented sign choices; it cannot
be inherited from the oriented relation.

There is a second obstruction.
Call a finite subhypergraph of $\mathcal H_\alpha$ \emph{coherently
orientable} if one can designate on every edge one vertex as its center and
the other two as its endpoints, and choose one unit representative of every
axis, so that \eqref{eq:lrw-projective-edge} holds simultaneously with
those designations on every edge.

\begin{proposition}\label{prop:lrw-coherent-coloring}
Every finite coherently orientable subhypergraph of
$\mathcal H_\alpha$ is $2$-colorable.
\end{proposition}

\begin{proof}
It is enough to prove the assertion hereditarily for every induced
subhypergraph.
For a vertex set $S$, let $M_S(t)$ be the matrix whose columns are indexed
by $S$ and whose rows are indexed by the edges contained in $S$, with
coefficients $1,1,-t$ at the two endpoints and the center of each edge.
The realized vectors give
\[
 \operatorname{rank}M_S(2\alpha)<|S|.
\]
Every $|S|\times|S|$ minor (if any) is an integral polynomial in $t$.
Since $\alpha$ is transcendental, all these minors vanish identically, and
hence $M_S(-2)$ has a nonzero kernel vector $v$.
Every row of $M_S(-2)$ has coefficients $1,1,2$, all positive.
Color the positive and negative coordinates of $v$ differently and recurse
on its zero set.
An edge exposed at the first nonzero stage contains both signs, so the
resulting two-coloring is proper.
\end{proof}

Thus only sign-inconsistent finite subhypergraphs can have chromatic number
greater than two.
The proposition gives no coloring of such subhypergraphs and therefore does
not decide the projective chromatic question.

\end{document}